\documentclass[twoside,11pt]{article}

\RequirePackage{epsfig}
\RequirePackage{amssymb}
\RequirePackage{hyperref}
\RequirePackage{natbib}
\RequirePackage{graphicx}

\bibpunct{(}{)}{;}{a}{,}{,}
 
\newenvironment{keywords}
{\bgroup\leftskip 20pt\rightskip 20pt \small\noindent{\bf Keywords:} }%
{\par\egroup\vskip 0.25ex}

\newcommand{\BlackBox}{\rule{1.5ex}{1.5ex}}  

\newenvironment{proof}{\par\noindent{\bf Proof\ }}{\hfill\BlackBox\\[2mm]}
\newtheorem{example}{Example} 
\newtheorem{theorem}{Theorem}
\newtheorem{lemma}[theorem]{Lemma} 
\newtheorem{proposition}[theorem]{Proposition}

\RequirePackage{subfig}
\RequirePackage{stmaryrd}

\hypersetup{
    colorlinks,%
    citecolor=black,%
    filecolor=black,%
    linkcolor=black,%
    urlcolor=black
}

\RequirePackage{url}
\RequirePackage{amsmath,amsfonts}
\RequirePackage{algorithm2e}

\newcommand{\egaldef}{:=} 

\newcommand{\un}{\mathbf{1}} 

\newcommand{\R}{\mathbb{R}} 
\newcommand{\N}{\mathbb{N}} 
 
\renewcommand{\H}{\mathcal{H}}
\newcommand{\X}{\mathcal{X}}

\newcommand{\F}{\mathcal{F}}
\newcommand{\cT}{\mathcal{T}}
\newcommand{\cH}{\mathcal{H}}

\newcommand{\card}{\mathrm{Card}} 
\newcommand{\argmin}{\mathrm{argmin}} 

\newcommand{\paren}[1]{\left(  #1 \right)} 
\newcommand{\parenj}[1]{\mathopen{}\left( #1  \right) \mathclose{}} 
\newcommand{\parens}[1]{( #1 )} 
\newcommand{\sparen}[1]{( #1 )} 
\newcommand{\parenb}[1]{\bigl( #1 \bigr)}

\newcommand{\croch}[1]{\left[ #1 \right]} 
\newcommand{\crochj}[1]{\mathopen{}\left[ #1 \right] \mathclose{}} 
\newcommand{\crochs}[1]{[ #1 ]} 
\newcommand{\scroch}[1]{[ #1 ]} 
\newcommand{\crochb}[1]{\bigl[ #1 \bigr]}
\newcommand{\crochB}[1]{\Bigl[ #1 \Bigr]}

\newcommand{\crochBb}[1]{\Biggl[ #1 \Biggr]}

\newcommand{\set}[1]{\left\{ #1 \right\}}
\newcommand{\acc}[1]{\set{#1}}
\newcommand{\setj}[1]{\mathopen{}\left\{ #1 \right\} \mathclose{}}
\newcommand{\sets}[1]{\{ #1 \}}
\newcommand{\sset}[1]{\{ #1 \}}
\newcommand{\setb}[1]{\bigl\{ #1 \bigr\}}

\newcommand{\absj}[1]{\mathopen{} \left\lvert #1 \right\rvert \mathclose{}} 
\newcommand{\abss}[1]{\lvert #1 \rvert}
\newcommand{\absb}[1]{\bigl\lvert #1 \bigr\rvert}
\newcommand{\absB}[1]{\Bigl\lvert #1 \Bigr\rvert}

\newcommand{\norm}[1]{\left \lVert #1 \right\rVert}

\newcommand{\norms}[1]{\lVert #1 \rVert}
\newcommand{\snorm}[1]{\norms{#1}}
\newcommand{\normb}[1]{\bigl \lVert #1 \bigr\rVert}

\newcommand{\normBb}[1]{\Biggl \lVert #1 \Biggr\rVert}

\newcommand{\normH}[1]{\norm{#1}_{\H}}
\newcommand{\normHs}[1]{\norms{#1}_{\H}}
\newcommand{\normHb}[1]{\normb{#1}_{\H}}

\newcommand{\normHBb}[1]{\normBb{#1}_{\H}}

\newcommand{\snormH}[1]{\snorm{#1}_{\H}}

\newcommand{\prodscal}[2]{\left\langle #1 , \, #2 \right\rangle} 

\newcommand{\prodscals}[2]{\langle #1 , \, #2 \rangle} 

\newcommand{\prodscalb}[2]{\bigl\langle #1 , \, #2 \bigr\rangle}

\newcommand{\prodH}[2]{\prodscal{#1}{#2}_{\H}} 
\newcommand{\prodHb}[2]{\prodscalb{#1}{#2}_{\H}} 

\renewcommand{\P}{\mathbb{P}}
\newcommand{\E}{\mathbb{E}} 
\newcommand{\var}{\mathrm{Var}} 

\newcommand{\bayes}{\mu^{\star}}
\newcommand{\ERM}{\widehat{\mu}}

\newcommand{\M}{\mathcal{M}}
\newcommand{\mM}{m \in \M}

\newcommand{\Dh}{\widehat{D}}
\newcommand{\tauh}{\widehat{\tau}}

\newcommand{\taus}{\tau^{\star}}

\newcommand{\Ds}{D^{\star}}

\DeclareMathOperator{\pen}{pen}
\newcommand{\penid}{\pen_{\mathrm{id}}} 

\newcommand{\vmax}{v_{\max}}

\newcommand{\tr}{\mathrm{tr}}

\newcommand{\Risk}[1]{\mathcal{R}\parenj{#1}}
\newcommand{\Risks}[1]{\mathcal{R}\parens{#1}}

\newcommand{\Remp}[1]{\widehat{\mathcal{R}}_n\parenj{#1}}
\newcommand{\Remps}[1]{\widehat{\mathcal{R}}_n\parens{#1}}
\newcommand{\Rempb}[1]{\widehat{\mathcal{R}}_n\parenb{#1}}

\newcommand{\grandO}{\mathcal{O}}

\newcommand{\blackbox}{\ensuremath{\blacksquare}}

\newcommand{\inter}[1]{\mathopen{} \left\llbracket #1 \right\rrbracket \mathclose{}}

\newcommand{\e}{\mathrm{e}}
\renewcommand{\d}{\,\mathrm{d}}

\newcommand{\Top}{\textbf{Top: }}

\newcommand{\Bottom}{\textbf{Bottom: }}

\newcommand{\klin}{k^{\mathrm{lin}}}
\newcommand{\kGau}{k^{\mathrm{G}}}
\newcommand{\kHer}{k^{\mathrm{H}}}
\newcommand{\kchi}{k^{\chi^2}}

\newcommand{\sh}{\widehat{s}}
\newcommand{\cC}{\mathcal{C}}

\newcommand{\figtroishspace}{\hspace*{-.14\textwidth}}
\newcommand{\figtroiswidth}{.4\textwidth}

\begin{document}


\title{A Kernel Multiple Change-point Algorithm via Model Selection}

\author{Sylvain Arlot  \\ \texttt{sylvain.arlot@u-psud.fr} \\
       Laboratoire de Math\'ematiques d'Orsay \\
Univ. Paris-Sud, CNRS, Inria, Universit\'e Paris-Saclay \\
91405 Orsay, France
       \and\ 
       Alain Celisse \\ \texttt{celisse@math.univ-lille1.fr} \\
Laboratoire de Math\'ematiques Paul Painlev\'e\\
UMR 8524 CNRS-Universit\'e Lille 1 \\
\textsc{Modal}  Project-Team \\
F-59\,655 Villeneuve d'Ascq Cedex, France      
		\and\ 
		Zaid Harchaoui \\ \texttt{zaid@uw.edu} \\
Department of Statistics\\
University of Washington\\ 
Seattle, WA. USA
}

\maketitle


\begin{abstract}
We tackle the change-point problem with data belonging to a general set. 
We build a penalty for choosing the number of change-points in the kernel-based method of \citet{Har_Cap:2007}. 
This penalty generalizes the one proposed by \citet{Leb:2005} for a one-dimensional signal changing only through its mean. 
We prove a non-asymptotic oracle inequality for the proposed method, thanks to a new concentration result for some function of Hilbert-space valued random variables. 
Experiments on synthetic and real data illustrate the accuracy of our method, showing that it can detect changes in the whole distribution of data, even when the mean and variance are constant. 
\end{abstract}


\section{Introduction}
\label{sec.intro}

The change-point problem has been tackled in numerous papers in the statistics and machine learning literature \citep{Bro_Dar:1993,Car:etal:1994,Tar_Bas_Nik:2014,Tru_Oud_Vay:2019}. 
Given a time series, the goal is to split it into 
homogeneous segments, in which the marginal distribution 
of the observations ---their mean or their variance, for instance--- is constant.  
When the number of change-points is known, this problem reduces to estimating the change-point locations as precisely as possible; 
in general, the number of change-points itself must be estimated. 
This problem arises in a wide range of applications, 
such as bioinformatics~\citep{Pic_etal:2005,CXPKX_2012}, 
neuroscience~ \citep{Park_Wang_Nobauer_Vaziri_Priebe:2015}, 
audio signal processing~\citep{Wu_Hsieh:2006}, 
temporal video segmentation~\citep{Koprinska_Carrato:2001}, 
hacker-attacks detection~\citep{Wang_Tang_Park_Priebe:2014}, 
social sciences~\citep{KoWa_2006} 
and econometrics~\citep{McCu_2009}.

\medskip

\paragraph{Related work}
%
%
A large part of the literature on change-point detection deals with observations in $\R$ or $\R^d$ and focuses on detecting 
changes arising in the mean and/or the variance of the signal 
\citep{Gijbels_Hall_Kneip:1999,Pic_etal:2005,Arlot:Celisse:2011,Bertin_Collilieux_Lebarbier:2014}. 
To this end, parametric models are often involved to derive change-point detection procedures. 
For instance, \citet{Com_Roz:2004}, \citet{Leb:2005}, \citet{PicLebBudRob:2011} 
and \citet{Cuevas_Chiken_Pignatiello:2012} make a Gaussian assumption, 
while \citet{Frick_Munk_Sieling:2014} and \citet{Cleynen_Lebarbier:2014} consider an exponential family.

%
The challenging problem of detecting abrupt changes in the full distribution of the data has been recently addressed in the nonparametric setting.
However, the corresponding procedures suffer several limitations since they are limited to real-valued data or they assume that the number of true change-points is known.
For instance, \citet{Zou_Yin_Feng_Wang:2014} design a strategy based on empirical cumulative distribution functions that allows to recover an unknown number of change-points by use of BIC, but only applies to $\R$-valued data. 
The strategy of \citet{Matteson_James:2014} applies to multivariate data, 
but it is time-consuming due to an intensive permutation use, 
and fully justified only in an asymptotic setting 
when there is a single change-point \citep{Biau_Bleakley_Mason:2015}. 
The kernel-based procedure proposed by \citet{Har_Cap:2007} enables to deal not only with vectorial data but also with structured data in the sense of~\citet{gartner:2008}, 
but it assumes that the number of change-points to recover is known, which reduces its practical interest when no such information is available.
Finally, many of these procedures are theoretically grounded only by asymptotic results, 
which makes their finite-sample performance questionable. 

%
Other attempts have been made to design change-point detection procedures allowing to deal with complex data (that are not necessarily vectors). However, the resulting procedures do not allow to detect more than one or two changes arising in particular features of the distribution.
For instance, \citet{Che_Zha:2015} describe a strategy based on a dissimilarity measure between individuals to compute a graph from which a statistical test allows to detect only one or two change-points. For a graph-valued time series, \citet{Wang_Tang_Park_Priebe:2014} design specific scan statistics to test whether one change arises in the connectivity matrix.

\medskip

\paragraph{Main contributions} 
We first describe a new efficient multiple change-point detection procedure (KCP) allowing to deal with univariate, multivariate or complex data (DNA sequences or graphs, for instance) as soon as a positive semidefinite kernel can be defined for them. 
Among several assets, this procedure is nonparametric and does not require to know the true number of change-points in advance. 
Furthermore, it allows to detect abrupt changes arising in the full distribution of the data by using a characteristic kernel; 
it can also focus on changes in specific features of the distribution by choosing an appropriate kernel. 

Secondly, our procedure (KCP) is theoretically grounded with a finite-sample optimality result, namely an oracle inequality in terms of quadratic risk,
stating that its performance is almost the same as that of the best one within the class we consider (Theorem~\ref{thm.oracle.large}). 
As argued by \citet{Leb:2005} for instance, 
such a guarantee is what we want for a change-point detection procedure. 
It means that the procedure detects only changes that are ``large enough'' given the noise level and the amount of data available, 
which is necessary to avoid having many false positives. 
A crucial point is that Theorem~\ref{thm.oracle.large} 
holds true for any value of the sample size~$n$; 
in particular it can be smaller than the dimensionality of the data. 
Note that contrary to previous oracle inequalities in the change-point detection framework, 
our result requires neither the variance to be constant 
nor the data to be Gaussian.

Thirdly, we settle a new concentration inequality for the quadratic norm of sums of independent Hilbert-valued vectors with exponential tails, which is a key result to derive our non-asymptotic oracle inequality with a large collection of candidate segmentations. 
The application domain of our exponential concentration inequality is not limited to change-point detection.

%
Let us finally mention that since the first version of the present work \citep{Arl_Cel_Har:2012:v1}, 
KCP has been successfully applied on different practical examples. 
\citet{Cel_Mor_Mar_Rig:2016:journal} illustrate that KCP outperforms state-of-the-art approaches on biological data. 
\citet{Cab_etal:2017} show that KCP with a Gaussian kernel outperforms three non-parametric methods 
for detecting correlation changes in synthetic multivariate time series, 
and provide an application to some data from behavioral sciences. 
Applying KCP to running empirical correlations \citep{Cab_etal:2018} 
or to the autocorrelations of a multivariate time series \citep{Cab_etal:2018:AR}  can make it focus on a specific kind of change 
---in the covariance between coordinates or in the autocorrelation structure of each coordinate, respectively---, as illustrated on synthetic data experiments and two real-world datasets from psychology.

\medskip

\paragraph{Outline} 
Motivating examples are first provided in Section~\ref{sec.chpt.pb} to highlight the wide applicability of our procedure to various important settings.
A comprehensive description of our kernel change-point detection algorithm (KCP, or Algorithm~\ref{algo}) is provided in Section~\ref{sec.algo}, where we also discuss algorithmic aspects as well as the practical choice of influential parameters (Section~\ref{sec.practice}).
Section~\ref{sec.oracle} exposes some important ideas underlying KCP 
and then states the main theoretical results of the paper (Proposition~\ref{pro.conc.quad} and Theorem~\ref{thm.oracle.large}).
Proofs of these main results have been collected in Section~\ref{sec.mainproofs}, while technical details have been deferred to Appendices~\ref{app.additional.proofs} and~\ref{sec.supmat}.
The practical performance of the kernel change-point detection algorithm is illustrated by 
experiments on synthetic data in Section~\ref{sec.synthetic.data} 
and on real data in Section~\ref{subsec.real.experiments}. 
Section~\ref{sec.conclusion} concludes the paper by a short discussion.

\paragraph{Notation.} 
For any $a<b$, we denote by $\inter{a,b} \egaldef [a,b] \cap \N$ the set of integers between $a$ and $b$.


\section{The change-point problem}
\label{sec.chpt.pb}

Let $\X$ be some measurable set and $X_1, \ldots, X_n \in \X$ a sequence of independent $\X$-valued random variables. 
For any $i \in \sets{1, \ldots, n }$, we denote by $P_{X_i}$ the distribution of $X_i$. 
The change-point problem can then be summarized as follows: 
Given $(X_i)_{1 \leq i \leq n}$, the goal is to find the locations of the abrupt changes along the sequence $P_{X_1}, \ldots, P_{X_n}$. 
Note that the case of dependent time series is often considered in the change-point literature 
\citep{Lavielle:Moulines:2000,Bardet_Kammoun:2008,Bar_Ken_Win:2010,Chang_Guo_Yao:2014}; 
as a first step, this paper focuses on the independent case for simplicity. 

An important example to have in mind is when $X_i$ corresponds to 
the observation at time $t_i = i/n$ of some random process on $[0,1]$, 
and we assume that this process is stationary over $[t^{\star}_{\ell} , t^{\star}_{\ell+1})$, 
$\ell=0, \ldots, D^{\star}-1$, 
for some fixed sequence $0 = t^{\star}_0 < t^{\star}_1 < \cdots < t^{\star}_{D^{\star}} = 1$. 
Then, the change-point problem is equivalent to localizing the change-points 
$t^{\star}_1, \ldots, t^{\star}_{D^{\star}-1} \in [0,1]$, 
which should be possible as the sample size $n$ tends to infinity. 
Note that we never make such an asymptotic assumption in the paper, 
where all theoretical results are non-asymptotic. 

\medbreak

Let us now detail some motivating examples of the change-point problem. 

\begin{example} \label{ex.Rd-moy}
The set $\X$ is $\R$ or $\R^d$, 
and the sequence $(P_{X_i})_{1 \leq i \leq n}$ changes only 
through its mean. 
This is the most classical setting, 
for which numerous methods have been proposed and analyzed 
in the one-dimensional setting \citep{Com_Roz:2004,Zha_Sie:2007,Boysen:etal:2009,Koko:2011,Fryzlewicz:2014} as well as the multi-dimensional case \citep{PicLebBudRob:2011,BleVer:2011,Hocking_Rigaill_Vert_Bach:2013,Soh_Cha:2014,Col_Leb_Rob:2015}.  
\end{example}

\begin{example} \label{ex.Rd-var}
The set $\X$ is $\R$ or $\R^d$, 
and the sequence $(P_{X_i})_{1 \leq i \leq n}$ changes only 
through its mean and/or its variance (or covariance matrix). 
This setting is rather classical, at least in the one-dimensional case, 
and several methods have been proposed for it 
\citep{Andreou_Ghysels:2002,Pic_etal:2005,Fryzlewicz_Rao:2014,Cab_etal:2017}.
\end{example}

\begin{example} \label{ex.Rd-loi}
The set $\X$ is $\R$ or $\R^d$, 
and no assumption is made on the changes in the sequence $(P_{X_i})_{1 \leq i \leq n}$. 
For instance, when data are centered and normalized, 
as in the audio-track example \citep{Rabiner:Schaefer:2007}, 
the mean and the variance of the $X_i$ can be constant, 
and only higher-order moments of $(P_{X_i})_{1 \leq i \leq n}$ are changing. 
Only a few recent papers deal with (an unknown number of) multiple change-points in a fully nonparametric framework: 
\citet{Zou_Yin_Feng_Wang:2014} for $\X=\R$, 
\citet{Matteson_James:2014} for $\X=\R^d$. 
Note that assuming $\X=\R$ and adding some further restrictions on the maximal order 
of the moments for which a change can arise in the sequence $(P_{X_i})_{1 \leq i \leq n}$, 
it is nevertheless possible to consider the multivariate sequence 
$ \parens{ \parens{ p_j(X_i)}_{0 \leq j \leq d} }_{1 \leq i \leq n}$, 
where $p_j$ is a polynomial of degree $j$ for $j \in \sets{0, \ldots, d}$, 
and to use a method made for detecting changes in the mean (Example~\ref{ex.Rd-moy}). 
For instance with $\R$-valued data, one can take $p_j(X)=X^j$ for every $1\leq j\leq d$, 
or $p_j$ equal to the $j$-th Hermite polynomial, as proposed by \cite{Laj_Arl_Bac:2014:icml}.
\end{example}

\begin{example} \label{ex.histos}
The set $\X$ is the $d$-dimensional simplex 
$\sset{(p_1, \ldots, p_d) \in [0,1]^d \text{ such that } p_1 + \dots + p_d = 1}$. 
For instance, audio and video data are often represented by histogram features \citep{oliva2001modeling,lowe2004distinctive,Rabiner:Schaefer:2007}, 
as done in Section~\ref{subsec.real.experiments}. 
%
In such cases, it is a bad idea to do as if $\X$ were $\R^d$-valued, 
since the Euclidean norm on $\R^d$ is usually a bad distance measure between histogram data. 
\end{example}

\begin{example} \label{ex.graphes}
The set $\X$ is a set of graphs. 
For instance, the $X_i$ can represent a social network \citep{KoWa_2006} 
or a biological network \citep{CXPKX_2012} 
that is changing over time \citep{Che_Zha:2015}. 
Then, detecting meaningful changes in the structure of a time-varying network 
is a change-point problem. 
In the case of social networks, this can be used for detecting 
the rise of an economic crisis \citep{McCu_2009}.
\end{example}

\begin{example} \label{ex.texte}
The set $\X$ is a set of texts (strings). 
For instance, text analysis can try to localize possible changes of authorship within 
a given text \citep{Che_Zha:2015}. 
\end{example}

\begin{example} \label{ex.ADN}
The set $\X$ is a subset of $\{A,T,C,G\}^{\N}$, the set of DNA sequences. 
For instance, an important question in phylogenetics is to find recombination events  from the genome of individuals of a given species \citep{KK_book_2010,Ane:2011}. 
This can be achieved from a multiple alignment of DNA sequences \citep{SchTsuVert_2004} by detecting abrupt changes (change-points) in the phylogenetic tree at each DNA position, 
that is, by solving a change-point problem. 
\end{example}

\begin{example} \label{ex.images}
The set $\X$ is a set of images. 
For instance, video shot boundary detection \citep{Cot:2006} or scene detection in videos \citep{All_Mad_Ye:2016} 
can be cast as change-point detection problems. 
\end{example}

\begin{example} \label{ex.functional}
The set $\X$ is an infinite-dimensional functional space. 
Such functional data arise in various fields \citep[see for instance][Chapter~2]{Fer_Vie:2006}, 
and the problem of testing whether there is a change or not 
in a functional time series has been considered recently 
\citep{Fer_Vie:2006,Ber_Gab_Hor_Kok:2009,Sharipov_Tewes_Wendler:2014}. 
\end{example}

Other kinds of data could be considered, such as 
counting data \citep{Cleynen_Lebarbier:2014,Ala_Gai_Gui:2015}, 
qualitative descriptors, 
as well as composite data, that is, data $X_i$ that are mixing several above examples. 

\medbreak

The goal of the paper is to propose 
a change-point algorithm that is 
(i) general enough to handle all these situations 
(up to the choice of an appropriate similarity measure on $\X$), 
(ii) in a non parametric framework, 
(iii) with an unknown number of change-points, and 
(iv) that we can analyze theoretically in all these examples simultaneously.

Note also that we want our algorithm to output a set of change-points that are ``close to'' the true ones, at least when $n$ is large enough.
But in settings where the signal-to-noise ratio is not large enough to recover all true change-points (for a given $n$), we do not want to have false positives.
This motivates the non-asymptotic analysis of our algorithm that we make in this paper. 
Since our algorithm relies on a model selection procedure, 
we prove in Section~\ref{sec.oracle} an oracle inequality, 
as usually done in non-asymptotic model selection theory.

\section{Detecting changes in the distribution with kernels}
\label{sec.algo}

Our approach for solving the general change-point problem uses positive semidefinite kernels. 
It can be sketched as follows. 

\subsection{Kernel change-point algorithm} \label{sec.algo.def}

For any integer $D \in \inter{ 1, n+1 }$, the set of sequences of $(D-1)$ change-points is defined by 
\begin{align}
\cT_n^D \egaldef \acc{ (\tau_0, \ldots, \tau_D) \in \N^{D+1} \, / \, 0 = \tau_0 < \tau_1 < \tau_2 < \cdots < \tau_{D}= n }
\label{eq.collection.segmentations.K}  
\end{align}
where $\tau_1, \ldots, \tau_{D-1}$ are the change-points, 
and $\tau_0,\tau_D$ are just added for notational convenience. 
Any $\tau \in \cT_n^D$ is called a \emph{segmentation} (of $\sets{1, \ldots, n}$) into $D_{\tau} \egaldef D$ segments.

\medbreak

Let $k: \X \times \X \to \R$ be a positive semidefinite kernel, 
that is, a measurable function $\X \times \X \to \R$ such that for any $x_1, \ldots, x_n \in \X$, 
the $n \times n$ matrix $(k(x_i,x_j))_{1 \leq i,j \leq n}$ is positive semidefinite. 
Examples of such kernels are given in Section~\ref{sec.chpt.high-dim.examples}. 
Then, we measure the quality of any candidate segmentation $\tau \in \cT_n^D$ with the \emph{kernel least-squares criterion} introduced by \cite{Har_Cap:2007}: 
\begin{equation} \label{eq.Remp}
\Remp{\tau} 
\egaldef \frac{1}{n} \sum_{i=1}^n k(X_i,X_i) 
- \frac{1}{n} \sum_{\ell=1}^D \croch{ \frac{1}{\tau_{\ell} - \tau_{\ell-1}} \sum_{i=\tau_{\ell-1}+1}^{\tau_{\ell}} \sum_{j=\tau_{\ell-1}+1}^{\tau_{\ell}} k(X_i,X_j)  }
\, .
\end{equation}
In particular when $\X=\R$ and $k(x,y)=xy$, we recover the usual least-squares criterion 
\[ 
\Remp{\tau} 
= \frac{1}{n}  \sum_{\ell=1}^D \sum_{i=\tau_{\ell-1}+1}^{\tau_{\ell}} \paren{X_i - \overline{X}_{ \inter{ \tau_{\ell-1}+1,\tau_{\ell} } }}^2 
\quad \text{where} \quad 
\overline{X}_{ \inter{ \tau_{\ell-1}+1,\tau_{\ell} } } \egaldef \frac{1}{\tau_{\ell} - \tau_{\ell-1}} \sum_{j=\tau_{\ell-1}+1}^{\tau_{\ell}} X_j
\,  . 
\]
Note that Eq.~\eqref{empirical.risk.loss} in Section~\ref{subsec.abstract.formulation.algorithm} provides an equivalent formula for $\Remp{\tau}$, 
which is helpful for understanding its meaning. 
Given the criterion \eqref{eq.Remp}, we cast the choice of $\tau$ as a model selection problem (as thoroughly detailed in Section~\ref{sec.oracle}), which leads to Algorithm~\ref{algo} below, 
that we now briefly comment on. 
\begin{algorithm}{}
\caption{\label{algo} kernel change-point algorithm (KCP)}
\begin{tabular}{ll}
\hline\\
\textbf{Input:} & observations:\quad $X_1, \ldots, X_n \in \X$, \\
& 
kernel:\quad  $k: \X\times \X \rightarrow \R$, \\
      & constants:\quad $c_1,c_2 >0$ and $D_{\max} \in \inter{1,n-1}$. 
\\[2mm]
\textbf{Step 1:}& $\forall D\in \inter{1,D_{\max}}$, compute (by dynamic programming): \\ 
& \hspace*{1,5cm} $\tauh(D) \in \argmin_{ \tau \in \cT_n^D } \setb{ \Remp{\tau} }$ \quad and \quad 
$\Rempb{\tauh(D)}$ 
\\[2mm]
\textbf{Step 2:} & find: \\
& \hspace*{1.5cm}
$ \displaystyle \Dh \in \argmin_{1 \leq D \leq  D_{\max} } \acc{ \Rempb{\widehat{\tau}(D) }  + \frac{ 1}{n} \parenj{ c_1 \log \binom{n-1}{D-1} + c_2 D} } $.
\\[3mm]
\textbf{Output:} & sequence of change-points: $\tauh = \tauh \parenb{ \Dh }$. \\
\hline \\
\end{tabular}
\end{algorithm}
\begin{itemize}
\item Step~1 of KCP consists in choosing the ``best'' segmentation with $D$ segments, 
that is, the minimizer of the kernel least-squares criterion $\Remp{\cdot}$ over $\cT_n^D$, for every $D \in \inter{1,D_{\max}}$. 

\item Step~2 of KCP chooses $D$ by model selection, using a penalized empirical criterion. 
A major contribution of this paper lies in the building and theoretical justification of the penalty 
$n^{-1} \sparen{ c_1 \log \binom{n-1}{D-1} + c_2 D}$, see Sections~\ref{sec.oracle}--\ref{sec.mainproofs}; 
a simplified penalty, of the form 
$\frac{D}{n} \sparen{ c_1 \log\sparen{\frac{n}{D}} + c_2} $, 
would also be possible, see Section~\ref{sec.large.collection}.  

\item Practical issues (computational complexity and choice of constants $c_1,c_2,D_{\max}$) are discussed in Section~\ref{sec.practice}. 
Let us only emphasize here that KCP is computationally tractable; its most expensive part is the minimization problem of Step~1, 
which can be done by dynamic programming \citep[see][]{Har_Cap:2007,Cel_Mor_Mar_Rig:2016:journal}. 
An implementation of KCP in python can be found in the ruptures package \cite{Tru_Oud_Vay:2018:python}. 
\end{itemize}

\subsection{Examples of kernels}
\label{sec.chpt.high-dim.examples}
KCP can be used with various sets $\X$ (not necessarily vector spaces) as long as a positive semidefinite kernel on $\X$ is available.
An important issue is to design relevant kernels, that are able to capture important features of the data for a given change-point problem, 
including non-vectorial data ---for instance, simplicial data (histograms), texts or graphs  (networks), 
see Section~\ref{sec.chpt.pb}. 
The question of choosing a kernel is discussed in Section~\ref{sec.conclusion.choix-k}. 

Classical kernels can be found in the books by \citet{Sch_Smo:2001}, \citet{Shawe:Cristianini:2004} and \citet{SchTsuVert_2004} for instance. 
Let us mention a few of them: 
\begin{itemize}
\item When $\X=\R^d$, $\klin(x,y) = \prodscals{x}{y}_{\R^d}$ defines the \emph{linear kernel}. 
When $d=1$, KCP then coincides with the algorithm proposed by \cite{Leb:2005}. 

\item When $\X=\R^d$, $\kGau_{h}(x,y) = \exp\scroch{ - \norm{x-y}^2 / (2 h^2) }$ defines the \emph{Gaussian kernel} with bandwidth $h>0$, 
which is used in the experiments of Section~\ref{sec.synthetic.data}. 

\item When $\X=\R^d$, $k^L_{h}(x,y) = \exp\scroch{ - \norms{x-y} / h }$ defines the \emph{Laplace kernel} with bandwidth $h>0$. 

\item When $\X=\R^d$, $k^{\mathrm{e}}_h(x,y) = \exp(  \prodscals{x}{y}_{\R^d} / h ) $ defines the \emph{exponential kernel} with bandwidth $h>0$. 
Note that, unlike the Gaussian and Laplace kernels, the exponential kernel is not translation-invariant.

\item When $\X=\R$, $\kHer_h(x,y) = \sum_{j=1}^5  H_{j,h}(x) H_{j,h}(y)$, corresponds to the Hermite kernel, where 
$ H_{j,h}(x) = 2^{j+1} \sqrt{\pi j!} \e^{-x^2/(2 h^2)}(-1)^j \e^{-x^2/2} \parens{\partial/\partial x}^j \parens{\e^{-x^2/2}}$
denotes the $j$-th Hermite function with bandwidth $h>0$. This kernel is used in Section~\ref{sec.synthetic.data}.

\item When $\X$ is the $d$-dimensional simplex as in Example~\ref{ex.histos}, the $\chi^2$-kernel can be defined by
$\kchi_{h} (x,y) = \exp\paren{ -\frac{1}{h \cdot d}\sum_{i=1}^d \frac{\paren{ x_i - y_i }^2}{x_i+y_i} }$ for some bandwidth $h>0$. 
An illustration of its behavior is provided in the simulation experiments of Sections~\ref{sec.synthetic.data} and~\ref{subsec.real.experiments}. 

\end{itemize}
Note that more generally, \citet{Sejdi_Sriper_Gretton_Fukumizu:2013}
proved that positive semidefinite kernels can be defined on any set $\X$ for which a semimetric of negative type is used to measure closeness between points. The so-called \emph{energy distance} between probability measures is an example \citep{Matteson_James:2014}. 
In addition, specific kernels have been designed for various kinds of structured data, 
including all the examples of Section~\ref{sec.chpt.pb} \citep{Cut_Fuk_Ver:2005,She:2012,Rakotomamonjy:Canu:2005,Vedaldi:2012}. 
Kernels can also be built from convolutional neural networks, with successful applications in computer vision \citep{mairal:etal:2014,Paulin:etal:2017}. 

Let us finally remark that KCP 
can also be used when $k$ is not a positive semidefinite 
kernel; its computational complexity remains unchanged, 
but we might loose the theoretical guarantees 
of Section~\ref{sec.oracle}. 

\subsection{Practical issues} \label{sec.practice}

\paragraph{Computational complexity.}
The discrete optimization problem at Step~1 of KCP is apparently hard to solve since, for each $D$, 
there are $\binom{n-1}{D-1}$ segmentations of $\sets{1, \ldots, n}$ into $D$ segments.  
Fortunately, as shown by \citet{Har_Cap:2007}, this optimization problem can be solved efficiently by 
dynamic programming. In the special case of a linear kernel, we recover the classical dynamic programming algorithm for detecting changes in mean~\citep{Fisher:1958,Auger_Lawrence:1989,Kay:1993}.

Denoting by $\cC_k$ the cost of computing $k(x,y)$ for some given $x,y \in \X$, the computational cost of a naive implementation of Step~1 
---computing each coefficient $(i,j)$ of the \emph{cost matrix} independently--- 
then is $\grandO\parens{ \cC_k  n^2 + D_{\max}n^4 }$ in time and $\grandO(D_{\max}n+n^2)$ in space. The computational complexity can actually be $\grandO((\cC_k+D_{\max})  n^2 )$ in time 
and $\grandO(D_{\max} n)$ in space as soon as one 
	either uses the summed area table or integral image technique as in~\citep{potapov:etal:2014} 
	or optimizes the interplay of the dynamic programming recursions and cost matrix computations~\citep{Cel_Mor_Mar_Rig:2016:journal}. 
For given constants $D_{\max}$ and $c_1,c_2$, Step~2 is straightforward since it consists in a minimization problem among $D_{\max}$ terms already stored in memory. 
Therefore, the overall complexity of KCP is at most $\grandO\parens{ (\cC_k + D_{\max}) n^2 }$ in time and $\grandO(D_{\max} n)$ in space.

\paragraph{Setting the constants $c_1,c_2$.}
At Step~2 of KCP, two constants $c_1,c_2>0$ appear in the penalty term. 
Theoretical guarantees (Theorem~\ref{thm.oracle.large} in Section~\ref{sec.oracle}) 
suggest to take $c_1=c_2=c$ large enough, 
but the lower bound on $c$ in Theorem~\ref{thm.oracle.large} is pessimistic, 
and the optimal value of $c$ certainly depends on unknown features of the data 
such as their ``variance'', as discussed after Theorem~\ref{thm.oracle.large}. 
In practice the constants $c_1,c_2$ must be chosen from data. 
To do so, we propose a fully data-driven method, 
based upon the ``slope heuristics'' \citep{2012_BauMauMich}, 
that is explained in Section~\ref{subsubsec.synthetic.proc}. 
Another way of choosing $c_1,c_2$ is described in supplementary material (Section~\ref{sec.choix-c1-c2.variance}).

\paragraph{Setting the constant $D_{\max}$.}
KCP requires to specify the maximal dimension $D_{\max}$ of the segmentations considered, 
a choice that has three main consequences. 
First, the computational complexity of KCP is affine in $D_{\max}$, 
as discussed above. 
Second, if $D_{\max}$ is too small ---smaller than the number of true change-points that can be detected---, 
the segmentation $\widehat{\tau}$ provided by the algorithm will necessarily be too coarse. 
Third, when the slope heuristics is used for choosing $c_1,c_2$, taking $D_{\max}$ larger than the true number of change-points 
might not be sufficient: better values for $c_1,c_2$ can be obtained by taking $D_{\max}$ larger, up to $n$. 
From our experiments, it seems that $D_{\max} \approx n/\sqrt{\log n}$ is large enough to provide good results.

\subsection{Related change-point algorithms}\label{sec.algo.related}
In addition to the references given in the Introduction, 
let us mention a few change-point algorithms 
to which KCP is more closely related. 

First, some two-sample (or homogeneity) tests based on kernels have been suggested. They tackle a simpler problem than the general change-point problem described in Section~\ref{sec.chpt.pb}. 
Among them, \citet{Gre_etal:2012} proposed a two-sample test based on a U-statistic of order two, called the 
maximum mean discrepancy (MMD). 
A related family of two-sample tests, called $B$-tests, has been proposed by \citet{Zar_Gre_Bla:2013}; 
$B$-tests have also been used by \citet{Li_Xie_Dai_Son:2015:nips,Li_Xie_Dai_Son:2015:journal} for localizing a single change-point. 
\citet{Har_Bac_Mou:2008} proposed a studentized kernel-based test statistic for testing homogeneity. 
Resampling methods ---(block) bootstrap and permutations--- have 
also been proposed for choosing the threshold of 
several kernel two-sample tests 
\citep{Fro_etal:2012,Chw_Sej_Gre:2014,Sharipov_Tewes_Wendler:2014}. 

Second, \citet{Har_Cap:2007} proposed a kernel change-point 
algorithm when the true number of segments $D^{\star}$ is 
known, 
which corresponds to Step~1 of KCP. 
The present paper proposes a data-driven choice of $D$ 
for which theoretical guarantees are proved. 

Third, when $\X=\R$ and $k(x,y)=xy$, $\Remp{\tau}$ is the 
usual least-squares risk and Step~2 of KCP 
is similar to the penalization procedures proposed by 
\citet{Com_Roz:2004} and \citet{Leb:2005} 
for detecting changes in the mean of a one-dimensional signal. 
We refer readers familiar with model selection techniques to 
Section~\ref{subsec.abstract.formulation.algorithm} for an 
equivalent formulation of KCP ---in more 
abstract terms--- that clearly emphasizes the links between 
KCP and these penalization procedures.

\section{Theoretical analysis}\label{sec.oracle}

We now provide theoretical guarantees for KCP. 
We start by reformulating it in an abstract way, 
which enlightens how it works. 

\subsection{Abstract formulation of KCP}
\label{subsec.abstract.formulation.algorithm}
Let $\H = \H_k$ denote the reproducing kernel Hilbert space (RKHS) associated 
with the positive semidefinite kernel $k:\ \X\times \X \to \R$. 
The canonical feature map $\Phi: \X \mapsto \cH$ is then defined by 
$\Phi(x) = k(x,\cdot)\in\H$ for every $x\in\X$. 
A detailed presentation of positive semidefinite kernels and related notions 
can be found in several books \citep{Sch_Smo:2001,Cucker:DXZ:2007,Ste_Chr:2008}. 

Let us define 
$Y_i = \Phi(X_i) \in \H$ 
for every $i \in \sets{1, \ldots, n}$, 
$Y = (Y_i)_{1 \leq i \leq n} \in \H^n$, 
$ \cT_n \egaldef \bigcup_{D=1}^n \cT_n^D $
the set of segmentations ---see Eq.~\eqref{eq.collection.segmentations.K}---, 
and for every $\tau \in \cT_n$, 
\begin{align}
F_{\tau} \egaldef \setj{ f = (f_1,\ldots,f_n) \in \H^n \text{ s.t. } 
  f_{\tau_{\ell-1}+1} = \cdots = f_{\tau_{\ell}} \quad \forall 1\leq \ell \leq D_{\tau} 
}
\enspace ,
\label{model.piecewise.property}
\end{align}
which is a linear subspace of $\H^n$.
We also define on $\H^n$ the canonical scalar product by 
$\prodscal{f}{g} \egaldef \sum_{i=1}^n \prodH{f_i}{g_i}$ for $f,g \in \H^n$, 
and we denote by $\norm{\cdot}$ the corresponding norm. 
Then, for any $g \in \H^n$, 
\begin{equation}
\label{def.Pitau} 
\Pi_{\tau} g \egaldef \argmin_{f \in F_{\tau}} \setj{ \norm{f - g}^2 }
\end{equation} 
is the orthogonal projection of $g\in\H^n$ onto $F_{\tau}$, 
and satisfies 
\begin{align}
\label{eq.regressogram.Hilbert}
\forall g \in \H^n , \,
\forall 1\leq \ell \leq D_\tau \, , \, 
\forall i \in \inter{ \tau_{\ell-1} + 1, \tau_{\ell}} , \quad 
(\Pi_{\tau} g)_i &= \frac{1}{\tau_{\ell} - \tau_{\ell-1}} \sum_{j =\tau_{\ell-1} + 1 }^{\tau_{\ell}} g_j 
\enspace .
\end{align}
The proof of this statement has been deferred to Appendix~\ref{app.proofs.orthog.proj}.

Following \citet{Har_Cap:2007}, 
the empirical risk $\Remps{\tau}$ defined by Eq.~\eqref{eq.Remp} can be rewritten as 
\begin{align}\label{empirical.risk.loss}
  \Remp{\tau} = \frac{1}{n} \norm{ Y - \ERM_{\tau} }^2 
\qquad \text{where} \qquad 
\ERM_{\tau} = \Pi_{\tau} Y
\enspace ,
\end{align}
as proved in Appendix~\ref{app.proofs.orthog.proj}.

For each $D\in \inter{ 1,D_{\max} }$,  Step~1 of KCP consists in finding a segmentation $\tauh(D)$ in $D$ segments such that
\begin{align*}
  \tauh(D) \in 
\argmin_{\tau \in \cT_n^D} \setj{ \normb{ Y - \ERM_{\tau} }^2 } 
= \argmin_{\tau \in \cT_n^D} \setj{ \inf_{f \in F_{\tau}} \sum_{i=1}^n \normb{ \Phi(X_i) - f_i }^2 } 
\enspace ,
\end{align*}
which is the ``kernelized'' version of the classical least-squares change-point algorithm  \citep{Leb:2005}. 
Since the penalized criterion of Step~2 is similar to that of \citet{Com_Roz:2004} and \citet{Leb:2005}, 
we can see KCP as a ``kernelization'' of these penalized least-squares change-point procedures. 

\medbreak

Let us emphasize that building a theoretically-grounded penalty for such a kernel least-squares change-point algorithm 
is not straightforward. For instance, we cannot apply the model selection results by \citet{Bir_Mas:2002} that were used by \citet{Com_Roz:2004} and \citet{Leb:2005}. 
Indeed, a Gaussian homoscedastic assumption is not realistic for general Hilbert-valued data, 
and we have to consider possibly heteroscedastic data
for which we assume only that $Y_i = \Phi(X_i)$ is bounded in $\cH$ 
---see Assumption~\eqref{hyp.donnees-bornees} in Section~\ref{sec.oracle.hyp}. 
Note that unbounded data $X_i$ can satisfy Assumption~\eqref{hyp.donnees-bornees}, for instance by choosing a bounded kernel such as the Gaussian or Laplace ones. 
In addition, dealing with Hilbert-valued random variables instead of (multivariate) real variables requires a new concentration inequality, 
see Proposition~\ref{pro.conc.quad} in Section~\ref{sec.concentration.hilbert}.

\subsection{Intuitive analysis}
\label{sec.oracle.intuition}

Section~\ref{subsec.abstract.formulation.algorithm} shows 
that KCP can be seen as a kernelization 
of change-point algorithms focusing on changes of the mean 
of the signal \citep[for instance]{Leb:2005}. 
Therefore, KCP is looking for changes in the ``mean'' of $Y_i = \Phi(X_i) \in \cH$, 
provided that such a notion can be defined. 

If $\H$ is separable and $\E\crochs{k(X_i,X_i)}<+\infty$, 
we can define the (Bochner) mean $\bayes_i \in \H$ of 
$\Phi(X_i)$ \citep{Led_Tal:1991}, 
also called the mean element of $P_{X_i}$, by 
\begin{align}
\forall g \in \H  , 
\qquad 
\prodH{\bayes_i}{g} = \E\crochb{g(X_i)} = \E\crochb{\prodH{Y_i}{g}}
\enspace . \label{eq.mean.element}
\end{align}
Then, we can write 
\begin{align*}
  \forall 1 \leq i \leq n,\qquad Y_i = \bayes_i + \varepsilon_i \in\H 
  \qquad \text{where} \qquad 
  \varepsilon_i \egaldef Y_i - \bayes_i 
  \enspace . 
\end{align*}
The variables $\parens{ \varepsilon_i}_{1\leq i\leq n}$ 
are independent and centered ---that is, $\forall g \in \cH$, $\E\crochs{\prodH{\varepsilon_i}{g}}=0$. 
So, we can understand 
$\ERM_{\tau}$ as the least-squares estimator over $F_{\tau}$ of $\bayes=\parens{\bayes_1,\ldots,\bayes_n} \in \cH^n$. 

\medbreak 

An interesting case is when $k$ is a \emph{characteristic kernel} \citep{Fuk_etal:2008}, 
or equivalently, when $\H_k$ is \emph{probability-determining} \citep{Fuk_Bac_Jor:2004,Fuk_Bac_Jor:2004:nips}. 
Then any change in the distribution $P_{X_i}$ 
induces a change in the mean element $\bayes_i$. 
In such settings, we can expect KCP to be able to 
detect \emph{any change} in the distribution $P_{X_i}$, 
at least asymptotically. 
For instance the Gaussian kernel is characteristic \citep[Theorem~4]{Fuk_Bac_Jor:2004:nips}, 
and general sufficient conditions for $k$ to be characteristic 
are known \citep{SGFSL_2010,Sri_Fuk_Lan:2011}. 

Note that \citet{Sharipov_Tewes_Wendler:2014} suggest to use 
$k_{\leq}(x,y) = \un_{x \leq y}$ as a ``kernel'' 
within a two-sample test, 
in order to look for any change of the distribution of 
real-valued data $X_i$ (Example~\ref{ex.Rd-loi}). 
This idea is similar to our proposal of using 
KCP with a characteristic kernel for 
tackling Example~\ref{ex.Rd-loi}, even if 
we do not advise to take $k=k_{\leq}$ within KCP. 
Indeed, when $k=k_{\leq}$, 
$\Remp{\tau} = \frac{1}{2} - \frac{D_{\tau}}{2n}$ 
as soon as the $X_i$ are all different 
so that KCP becomes useless. 
This illustrates that using a kernel which is not 
symmetric positive definite 
should be done cautiously. 

\subsection{Notation and assumptions}
\label{sec.oracle.hyp}

Throughout the paper, we assume that 
\emph{$\H$ is separable}, which is kind of a minimal assumption for two reasons: it allows to define the mean element ---see Eq.~\eqref{eq.mean.element}---, and most reasonable examples satisfy this requirement \citep[][p.~4]{Die_Bac:2014:v1}.
Let us further assume
\begin{align} 
\label{hyp.donnees-bornees}  \tag{\bf Db}
\exists M \in (0,+\infty) \, , \quad 
\forall i \in \sets{1, \ldots, n}, \qquad 
\normHs{Y_i}^2 = \normHs{\Phi(X_i)}^2 
= k(X_i,X_i) \leq M^2 
\quad \text{a.s.}
\end{align}
For every $1\leq i \leq n$, we also define the ``variance'' of $Y_i$ by 
\begin{align}
\label{eq.v_i}
v_i \egaldef \E\crochB{ \normHb{\Phi(X_i) - \bayes_i}^2}
 = \E\crochb{k(X_i,X_i)} - \normHb{\bayes_i}^2
 = \E\crochb{k(X_i,X_i) - k(X_i, X'_i)} 
\end{align}
where $X'_i$ is an independent copy of $X_i$, 
and  
$\vmax  \egaldef   \max_{1 \leq i \leq n}  v_i$. 
Let us make a few remarks. 
\begin{itemize}
\item If \eqref{hyp.donnees-bornees} holds true, then 
the mean element $\bayes_i$ exists since $\E\crochs{\sqrt{k(X_i,X_i)}} < \infty$, 
the variances $v_i$ are finite and smaller than $\vmax \leq M^2$. 

\item If \eqref{hyp.donnees-bornees} holds true, then 
$Y_i$ admits a covariance operator $\Sigma_i$ 
that is trace-class and $v_i = \tr(\Sigma_i)$. 

\item If $k$ is translation invariant, that is,
$\X$ is a vector space and 
$k(x,x^{\prime}) = \overline{k}(x-x^{\prime})$ for every 
$x,x' \in \X$, and some measurable function $\overline{k}: \X \to \R$, 
then \eqref{hyp.donnees-bornees} holds true with $M^2 = k(0)$ 
and $v_i = k(0) - \snormH{\bayes_i}^2$. 
For instance the Gaussian and Laplace kernels are translation invariant (see Section~\ref{sec.chpt.high-dim.examples}). 

\item Let us consider the case of the linear kernel 
$(x,y) \mapsto \prodscal{x}{y}$ 
on $\X=\R^d$. 
If $\E\crochs{\norms{X_i}_{\R^d}^2}<\infty$, then, 
$v_i = \tr(\Sigma_i)$ where $\Sigma_i$ is the covariance matrix of $X_i$. 
In addition, \eqref{hyp.donnees-bornees} holds true if 
and only if $\norm{X_i}_{\R^d} \leq M$ a.s.\@ for all $i$. 

\end{itemize}

\subsection{Concentration inequality for some quadratic form of Hilbert-valued random variables} 
\label{sec.concentration.hilbert}

Our main theoretical result, stated in 
Section~\ref{sec.large.collection}, 
relies on two concentration inequalities for some 
linear and quadratic functionals of Hilbert-valued vectors.
Here we state the concentration result 
that we prove for the quadratic term, 
which is significantly different from existing results 
and can be of independent interest. 

\begin{proposition}[Concentration of the quadratic term] \label{pro.conc.quad}
Let $\tau\in \cT_n$ and recall that $\Pi_{\tau}$ is 
the orthogonal projection onto $F_{\tau}$ in $\H^n$ defined by Eq.~\eqref{def.Pitau}. 
Let $X_1, \ldots, X_n$ be independent $\X$-valued random variables 
and assume that \eqref{hyp.donnees-bornees} holds true, 
so that we can define 
$\varepsilon = (\varepsilon_1, \ldots, \varepsilon_n) \in \cH^n$ 
as in Section~\ref{subsec.abstract.formulation.algorithm}. 
Then for every $x>0$, with probability at least $1-\e^{-x}$, 
\begin{align*} 
 \norms{\Pi_{\tau} \varepsilon}^2 - \E\crochj{\norms{\Pi_{\tau}  \varepsilon }^2}  
&\leq 
\frac{14 M^2}{3} \parenj{ x + 2 \sqrt{2 x D_{\tau}}}
\enspace .
\end{align*}
\end{proposition}
Proposition~\ref{pro.conc.quad} is proved in Section~\ref{sec.proofs.quad}. 
The proof relies on a combination of Bernstein's and Pinelis-Sakhanenko's inequalities. 
Note that the proof of Proposition~\ref{pro.conc.quad} also 
shows that for every $x>0$, with probability at least $1-\e^{-x}$, 
\begin{align*} 
 \norms{\Pi_{\tau} \varepsilon}^2 - \E\crochj{\norms{\Pi_{\tau}  \varepsilon }^2}  
&\geq 
- \frac{14 M^2}{3} \parenj{ x + 2 \sqrt{2 x D_{\tau}}}
\enspace .
\end{align*}

Previous concentration results for quantities such as 
$ \norms{\Pi_{\tau} \varepsilon}^2$ 
or $ \norms{\Pi_{\tau} \varepsilon}$
do not imply Proposition~\ref{pro.conc.quad} ---even up to 
numerical constants. 
Indeed, they either assume that $\varepsilon$ is a Gaussian 
vector, or they involve much larger deviation terms (see Section~\ref{sec.Talagrand} for a detailed discussion 
of these results). 

\subsection{Oracle inequality for KCP}
\label{sec.large.collection}

Similarly to the results of \citet{Com_Roz:2004} and \citet{Leb:2005} in the one-dimensional case, 
we state below a non-asymptotic oracle inequality for KCP. 
First, we define the quadratic risk of any $\mu \in \H^n$ as an estimator of $\bayes$ by 
\begin{align*}
\Risks{\mu} 
= \frac{1}{n} \norms{ \mu - \bayes }^2 
= \frac{1}{n} \sum_{i=1}^n \normHs{ \mu_{i} - \bayes_i }^2 
\enspace .
\end{align*}

\begin{theorem}
\label{thm.oracle.large} 
We consider the framework and notation introduced in Sections~\ref{sec.chpt.pb}--\ref{sec.oracle}. 
Let $C \geq 0$ be some constant. 
Assume that \eqref{hyp.donnees-bornees} holds true and that 
$\pen: \cT_n\to \R$ is some penalty function satisfying 
\begin{align}
\label{eq.thm.oracle-ineq.sansVmin.hyp-pen}
\forall \tau \in \cT_n, \qquad \pen(\tau) \geq \frac{C M^2}{n} \crochj{ \log \binom{n-1}{D_{\tau} - 1} + D_{\tau} }
\enspace . 
\end{align}
Then, some numerical constant $L_1>0$ exists such that the following holds\textup{:} 
if $C \geq L_1$, for every $y \geq 0$, 
an event of probability at least $1 - \e^{-y}$ exists on which, for every 
\begin{align}
\label{eq.thm.oracle-ineq.sansVmin.def-tauhat}
\tauh \in \argmin_{\tau \in \cT_n} \set{ \Remp{\tau} + \pen(\tau) }
\enspace , 
\end{align}
we have 
\begin{equation}
\label{eq.thm.oracle-ineq.sansVmin.resultat}
\Risk{\ERM_{\tauh}}  
\leq  2 \inf_{\tau\in \cT_n} \setb{   \Risk{\ERM_{\tau}} + \pen(\tau) } +  \frac{83 y M^2 } {n} 
\enspace .  
\end{equation}
\end{theorem}
Theorem~\ref{thm.oracle.large} is proved in Section~\ref{sec.proof.thm.oracle.large}. 
In addition, Section~\ref{sec.proof.outline} provides some insight about the construction of the penalty 
suggested by Eq.~\eqref{eq.thm.oracle-ineq.sansVmin.hyp-pen}. 
In a few words, the idea is to take a penalty such that the empirical criterion 
$\Remp{\tau} + \pen(\tau)$ in Eq.~\eqref{eq.thm.oracle-ineq.sansVmin.def-tauhat} 
mimics (approximately) the oracle criterion $\Risk{\ERM_{\tau}}$. 
At least, the penalty must be \emph{large enough} so that 
$\Remp{\tau} + \pen(\tau) \geq \Risk{\ERM_{\tau}}$ holds true 
\emph{simultaneously} for all $\tau \in \cT_n$ 
(up to technical details, see Section~\ref{sec.mainproofs}).

\medbreak

%
Theorem~\ref{thm.oracle.large} applies to the segmentation $\tauh$ provided by KCP when $c_1,c_2 \geq L_1 M^2$. 
%
%
Theorem~\ref{thm.oracle.large} shows that $\ERM_{\tauh}$ estimates well the ``mean'' $\bayes\in \H^n$ of the transformed time 
series $Y_1 = \Phi(X_1), \ldots, Y_n = \Phi(X_n)$. 
Such a non-asymptotic oracle inequality is the usual way to theoretically validate a model selection procedure \citep[for instance]{Bir_Mas:2002}. 
It is therefore a natural way to theoretically validate any change-point detection procedure based on model selection. 
As argued by \citet{Leb:2005} for instance, 
proving such a non-asymptotic result is necessary for taking into account situations where some changes 
are too small to be detected ---they are ``below the noise level''.
By defining the performance of $\tauh$ as the quadratic risk of $\ERM_{\tauh}$ as an estimator of $\bayes$, 
a non-asymptotic oracle inequality such as Eq.~\eqref{eq.thm.oracle-ineq.sansVmin.resultat} is the natural way to prove that 
KCP works well for finite sample size 
and for a set $\X$ that can have a large dimensionality 
(possibly much larger than the sample size~$n$). 
The consistency of KCP for estimating the change-point locations, which is outside the scope of this paper, 
is discussed in Section~\ref{sec.conclu.consistency}.

%
%
The constant $2$ in front of the first term in Eq.~\eqref{eq.thm.oracle-ineq.sansVmin.resultat} has no special meaning, and could be replaced by any quantity strictly larger than $1$, at the price of enlarging $L_1$ and $83$. 

\medbreak

%
The value $2 L_1 M^2$ suggested by Theorem~\ref{thm.oracle.large} for the constants $c_1,c_2$ within KCP 
should not be used in practice because it is likely to lead to a conservative choice for two reasons. 
%
First, the minimal value $L_1$ for the constant $C$ suggested by the proof of Theorem~\ref{thm.oracle.large} depends on the numerical constants appearing in the deviation bounds of Propositions~\ref{pro.conc.quad} and~\ref{prop.concentration.lineaire}, which probably are not optimal. 
%
%
Second, the constant $M^2$ in the penalty is probably pessimistic in several frameworks. 
For instance with the linear kernel and Gaussian data belonging to $\X=\R$, 
\eqref{hyp.donnees-bornees} is not satisfied, but other similar oracle inequalities have been proved 
with $M^2$ replaced by the residual variance \citep{Leb:2005}. 
%
%
In practice, as we do in the experiments of Section~\ref{sec.synthetic.data}, 
we recommend to use a data-driven value for the leading constant $C$ in the penalty, as explained in Section~\ref{sec.practice}.

%
Theorem~\ref{thm.oracle.large} also applies to KCP with 
simplified penalty shapes. 
Indeed, for any $D \in \sets{1, \ldots, n}$, 
\[ 
\binom{n-1}{D - 1} 
= \frac{D}{n} \binom{n}{D} 
\leq \binom{n}{D} 
\leq \frac{n^D}{D!} 
\leq \paren{ \frac{n \e}{D} }^D
\]
so that Theorem~\ref{thm.oracle.large} applies to the penalty 
$\frac{D}{n} \sparen{ c_1 \log\sparen{\frac{n}{D}} + c_2} $ 
---similar to the one of \citet{Leb:2005}--- 
as soon as $c_1,c_2 \geq 2 L_1 M^2$. 
A BIC-type penalty $C D \log(n) / n$ 
is also covered by Theorem~\ref{thm.oracle.large} 
provided that $C \geq 3.9 L_1 M^2$ and $n \geq 2$, 
even if we do not recommend to use it given our experiments 
---see Section~\ref{sec.synthetic.data.results}.

%
A nice feature of Theorem~\ref{thm.oracle.large} is that it holds under mild assumptions: 
we only need the data $X_i$ to be independent and to have \eqref{hyp.donnees-bornees} 
satisfied. 
As noticed in Section~\ref{sec.oracle.hyp}, 
\eqref{hyp.donnees-bornees} holds true for translation-invariant kernel
such as the Gaussian and Laplace kernels. 
Compared to previous results \citep{Com_Roz:2004,Leb:2005}, 
we do not need the data to be Gaussian or homoscedastic. 
Furthermore, the independence assumption can certainly be relaxed: 
to do so, it would be sufficient to prove 
concentration inequalities similar to 
Propositions~\ref{pro.conc.quad} and~\ref{prop.concentration.lineaire}
for some dependent $X_i$. 

\medbreak 

%
%
In the particular setting where $\X=\R$ and $k$ is the linear kernel $(x,y) \mapsto xy$, 
Theorem~\ref{thm.oracle.large} provides an oracle inequality similar to the one proved by \citet{Leb:2005} for Gaussian and homoscedastic real-valued data. 
The price to pay for extending this result to heteroscedastic Hilbert-valued data is rather mild: 
we only assume \eqref{hyp.donnees-bornees} 
and replace the residual variance by $M^2$.

%
Apart from the results already mentioned, 
a few oracle inequalities have been proved for change-point procedures, 
for real-valued data with a multiplicative penalty \citep{Baraud_Giraud_HUet:2009}, 
for discrete data \citep{Aka:2011}, 
for counting data with a total-variation penalty \citep{Ala_Gai_Gui:2015}, 
for counting data with a penalized maximum-likelihood procedure \citep{Cleynen_Lebarbier:2014} 
and for data distributed according to an exponential family \citep{Cle_Leb:2014b}. 
Among these oracle inequalities, 
only the result by \citet{Aka:2011} is more precise than Theorem~\ref{thm.oracle.large} (there is no $\log(n)$ factor compared to the oracle loss), 
at the price of using a smaller (dyadic) collection of possible segmentations, hence a worse oracle performance in general.

\section{Main proofs} \label{sec.mainproofs} 

We now prove the main results of the paper, 
Theorem~\ref{thm.oracle.large} and Proposition~\ref{pro.conc.quad}. 

\subsection{Outline of the proof of Theorem~\ref{thm.oracle.large}} \label{sec.proof.outline}

As usual for proving an oracle inequality \cite[see][Section~2.2]{Arl:2014:hdr}, 
we remark that 
by Eq.~\eqref{eq.thm.oracle-ineq.sansVmin.def-tauhat}, for every $\tau \in \cT_n$, 
\[ 
\Remp{\ERM_{\tauh}} + \pen(\tauh) \leq \Remp{\ERM_{\tau}} + \pen(\tau) 
\enspace . 
\]
Therefore, 
\begin{gather}
 \label{eq.pr.thm.oracle-ineq.sansVmin.1}
\Risk{\ERM_{\tauh}} + \pen(\tauh) - \penid(\tauh) 
\leq \Risk{\ERM_{\tau}} + \pen(\tau) - \penid(\tau)
\\
\text{where} \qquad 
\forall \tau \in \cT, \qquad 
\penid(\tau) \egaldef \Risk{\ERM_{\tau}} - \Remp{\ERM_{\tau}} + \frac{1}{n} \norm{\varepsilon}^2
\label{eq.ideal.penalty}
\enspace . 
\end{gather}
The idea of the proof is that if we had $\pen(\tau) \geq \penid(\tau)$ for every $\tau \in \cT_n$, we would get an oracle 
inequality similar to Eq.~\eqref{eq.thm.oracle-ineq.sansVmin.resultat}. 
What remains to obtain is a deterministic upper bound on the \emph{ideal penalty} $\penid(\tau)$ that holds true simultaneously for all $\tau \in \cT_n$ on a large probability event. 
To this aim, our approach is to compute 
$\E\crochs{\penid(\tau)}$ and to show that 
$\penid(\tau)$ concentrates around its expectation for every $\tau \in \cT_n$ (Sections~\ref{sec.mainproofs.penid}--\ref{sec.proofs.quad}). Then we use a union bound as detailed in Section~\ref{sec.proof.thm.oracle.large}.
A similar strategy has been used for instance 
by \citet{Com_Roz:2004} and \citet{Leb:2005} 
in the specific context of change-point detection.

Note that we prove below a slightly weaker result than $\pen(\tau) \geq \penid(\tau)$, which is nevertheless 
sufficient to obtain Eq.~\eqref{eq.thm.oracle-ineq.sansVmin.resultat}. 
Remark also that Eq.~\eqref{eq.pr.thm.oracle-ineq.sansVmin.1} would be true if the constant 
$n^{-1} \norm{\varepsilon}^2$ in the definition \eqref{eq.ideal.penalty} of $\penid$ was replaced 
by any quantity independent from $\tau$; the reasons for this specific choice appear in the computations below.

\subsection{Computation of the ideal penalty} \label{sec.mainproofs.penid} 

From Eq.~\eqref{eq.ideal.penalty} it results that for every 
$\tau \in \cT_n$, 
\begin{align}
\notag 
n\times\penid(\tau)  & = \norm{\ERM_{\tau} - \bayes}^2 - \norm{\ERM_{\tau} - Y}^2 + \norm{\varepsilon}^2
\\
\notag 
&= 
\norm{\ERM_{\tau} - \bayes}^2 - \norm{\ERM_{\tau} - \bayes - \varepsilon}^2 + \norm{\varepsilon}^2
\\
\notag 
&= 2 \prodscal{\ERM_{\tau} - \bayes}{\varepsilon}
\\
\notag 
&= 2 \prodscalb{\Pi_{\tau} (\bayes + \varepsilon) - \bayes}{\varepsilon}
\\
\notag 
&= 2 \prodscal{\Pi_{\tau} \bayes - \bayes}{\varepsilon}
+ 2 \prodscal{\Pi_{\tau} \varepsilon}{\varepsilon}
\\
&= 2 \prodscal{\Pi_{\tau} \bayes - \bayes}{\varepsilon}
+ 2 \norm{ \Pi_{\tau} \varepsilon}^2 
\label{eq.penid}
\end{align}
since $\Pi_{\tau}$ is an orthogonal projection. 
The next two sections focus separately 
on the two terms appearing in Eq.~\eqref{eq.penid}.

\subsection{Concentration of the linear term}
\label{sec.mainproofs.linear} 

We prove in Section~\ref{sec.proof.conc-linear} the following concentration inequality for the linear term in Eq.~\eqref{eq.penid}, mostly by applying Bernstein's inequality. 
\begin{proposition}[Concentration of the linear term] \label{prop.concentration.lineaire}
%
If \eqref{hyp.donnees-bornees} holds true, then for every $x > 0$, with probability at least $1-2 \e^{-x}$, 
\begin{equation}
\forall \theta > 0 \, , \quad 
\absB{\prodscalb{(I - \Pi_{\tau}) \bayes}{\Phi(\mathbf{X})-\bayes}} \leq \theta \norm{\Pi_{\tau}\bayes - \bayes}^2 + \paren{ \frac{\vmax}{2 \theta} + \frac{4 M^2}{3} } x \enspace . 
\end{equation}
\end{proposition}

\subsection{Dealing with the quadratic term} \label{sec.proofs.quad}

We now focus on the quadratic term in the right-hand side of Eq.~\eqref{eq.penid}. 

\subsubsection{Preliminary computations}
We start by providing a useful closed-form formula for 
$\norms{\Pi_{\tau} \varepsilon}^2$ and by computing 
its expectation. 
First, a straightforward consequence of Eq.~\eqref{eq.regressogram.Hilbert} is that
\begin{align}
\label{eq.norm-proj-histo}
\norm{\Pi_\tau \varepsilon}^2
&= \sum_{\ell =1}^{D_\tau} \crochj{ \frac{1}{\tau_{\ell}-\tau_{\ell-1}} \normHBb{\sum_{i =\tau_{\ell-1}+1}^{\tau_{\ell}} \varepsilon_i}^2}  
\\
&= \sum_{\ell=1}^{D_\tau} \crochBb{ \frac{1}{\tau_{\ell}-\tau_{\ell-1}} \sum_{\tau_{\ell-1}+1 \leq i,j \leq \tau_{\ell}} \prodH{\varepsilon_i}{\varepsilon_j} } . \label{eq.formule.terme-quad}
\end{align}

Second, we remark that 
for every $i,j \in \set{1 , \ldots, n} $, 
\begin{align}
\notag 
\E\croch{ \prodH{\varepsilon_i}{\varepsilon_j} } 
&= \E\croch{ \prodH{\Phi(X_i)}{\Phi(X_j)} } -  \E\croch{ \prodH{\bayes_i}{\Phi(X_j)} } - \E\croch{ \prodH{\Phi(X_i)}{\bayes_j} } + \prodH{\bayes_i}{\bayes_j} 
\\
\notag 
&= \E\croch{ \prodH{\Phi(X_i)}{\Phi(X_j)} } - \prodH{\bayes_i}{\bayes_j} 
\\
&= \un_{i=j} \paren{ \E\croch{k(X_i,X_i)} - \normH{\bayes_i}^2 } 
= \un_{i=j} v_i .
\label{eq.prodeps} 
\end{align}

Combining Eq.~\eqref{eq.formule.terme-quad} and~\eqref{eq.prodeps}, 
we get 
\begin{align}
\E\croch{ \norm{\Pi_\tau \varepsilon}^2 }
&= \sum_{\ell=1}^{D_\tau} \croch{ \frac{1}{\tau_{\ell}-\tau_{\ell-1}} \sum_{i =\tau_{\ell-1}+1}^{\tau_{\ell}} v_i }  
= \sum_{\ell=1}^{D_\tau} v_{\ell}^\tau \enspace,
\label{eq.E-terme-quadratique}
\end{align}
where $v_{\ell}^{\tau} \egaldef \frac{1}{\tau_{\ell}-\tau_{\ell-1}} \sum_{i =\tau_{\ell-1}+1}^{\tau_{\ell}} v_i$.

\subsubsection{Concentration: proof of Proposition~\ref{pro.conc.quad}}
\label{sec.proof.pro.conc.quad}
This proof is inspired from that of a concentration inequality by \cite{Sauv_2009} in the context of regression with real-valued non-Gaussian noise.
Let us define 
\begin{align*} 
T_{\ell} \egaldef \frac{1}{\tau_{\ell}-\tau_{\ell-1}}  \normH{\sum_{j = \tau_{\ell-1}+1}^{\tau_{\ell}}\varepsilon_j}^2  ,
\qquad \text{so that}  \qquad 
 \norm{\Pi_{\tau} \varepsilon}^2 
 =  \sum_{1\leq \ell\leq D_\tau} T_{\ell} 
\end{align*}
by Eq.~\eqref{eq.norm-proj-histo}. 
Since the real random variables 
$(T_{\ell})_{1\leq \ell \leq D_{\tau}} $ are independent, we get a concentration inequality for 
their sum $\norm{\Pi_{\tau} \varepsilon}^2$ via Bernstein's inequality (Theorem~\ref{thm.Bernstein.ineq}) as long as $T_{\ell}$ satisfies some moment conditions. 
The rest of the proof consists in showing such moment bounds
by using  Pinelis-Sakhanenko's deviation inequality (Proposition~\ref{prop.Pinelis.Sakhanenko}). 

First, note that \eqref{hyp.donnees-bornees} implies that 
$\normHs{\varepsilon_i}\leq 2M$ almost surely for every $i$ by Lemma~\ref{lem.upperbounds.Bernstein.Linear}, 
hence $\normHs{\sum_{i=\tau_{\ell-1}+1}^{\tau_{\ell}} \varepsilon_i} \leq 2 (\tau_{\ell}-\tau_{\ell-1}) M$ a.\@s.\@ for every $1\leq \ell \leq D_\tau$. 
Then for every $q\geq 2$ and $1\leq \ell\leq D_\tau$,  
\begin{align}
\E \croch{ T_{\ell}^q } 
&
= \frac{1}{(\tau_{\ell}-\tau_{\ell-1})^q} \int_0^{2 (\tau_{\ell}-\tau_{\ell-1}) M} 2q x^{2q-1}  \P\croch{\normH{\sum_{i = \tau_{\ell-1}+1}^{\tau_{\ell}}\varepsilon_i} \geq x} \d x 
\enspace .
\label{eq.pr.prop.quadratic.sansVmin.1}
\end{align}

Second, since $\normH{\varepsilon_i}\leq 2M$ almost surely and 
$\E\croch{ \normH{\varepsilon_i}^2 } = v_i \leq M^2$ for every $i$, 
we get that for every $p \geq 2$ and $1\leq \ell\leq D_\tau$, 
\begin{align*}
\sum_{i=\tau_{\ell-1}+1}^{\tau_{\ell}} \E\croch{ \normH{\varepsilon_i}^p} 
%
&\leq \frac{p!}{2} \paren{\sum_{j=\tau_{\ell-1}+1}^{\tau_{\ell}} v_j} \paren{\frac{2M}{3}}^{p-2} 
\leq \frac{p!}{2} \times (\tau_{\ell}-\tau_{\ell-1}) M^2 \times \paren{\frac{2M}{3}}^{p-2} 
\enspace .
\end{align*}
Hence, the assumptions of Pinelis-Sakhanenko's deviation inequality 
\citep{Pin:Sak:1986} ---which is recalled by Proposition~\ref{prop.Pinelis.Sakhanenko}--- are 
satisfied with $c=2M/3$ and $\sigma^2= (\tau_{\ell}-\tau_{\ell-1}) M^2$, 
and we get that for every $x \in [0,\ 2 (\tau_{\ell}-\tau_{\ell-1}) M]$  
\begin{align*}
\P\parenj{\normH{\sum_{i= \tau_{\ell-1}+1}^{\tau_{\ell}}\varepsilon_i} \geq x}
&\leq 
2 \exp\parenj{ - \frac{x^2}{2 \crochj{ (\tau_{\ell}-\tau_{\ell-1}) M^2 + \frac{2 M x}{3} }} }
\\
&\leq 
2 \exp\parenj{ - \frac{3 x^2}{14 (\tau_{\ell}-\tau_{\ell-1}) M^2 }}
\enspace . 
\end{align*}
Together with Eq.~\eqref{eq.pr.prop.quadratic.sansVmin.1}, 
we obtain that 
\begin{align}
\notag 
\E \croch{ T_{\ell}^q } 
&  \leq \frac{4 q}{(\tau_{\ell}-\tau_{\ell-1})^q} \int_0^{2 (\tau_{\ell}-\tau_{\ell-1}) M} x^{2q-1}  \exp \croch{ -\frac{3 x^2}{14 (\tau_{\ell}-\tau_{\ell-1}) M^2 }}
 \d x 
\\
\notag 
%
&\leq 
4 q \paren{\frac{7 M^2}{3}}^q 
\int_0^{+\infty} u^{2q-1}  \exp \croch{ -\frac{u^2}{2}} \d u 
\\
\notag 
&= 
2^{q-1} (q-1)! \times 4 q \paren{\frac{7 M^2}{3}}^q 
\\
&= 
2 \times (q!)  \croch{ \frac{14 M^2}{3} }^{q} 
\enspace ,
\label{eq.pr.prop.quadratic.sansVmin.2}
\end{align}
since for every  $q \geq 1$, 
\begin{equation*}
 \int_0^{+\infty} u^{2q-1} \exp(-u^2/2) \d u = 2^{q-1} (q-1)! \enspace.
\end{equation*}
Finally summing Eq.~\eqref{eq.pr.prop.quadratic.sansVmin.2} over $1\leq \ell \leq D_\tau$, it comes 
\begin{align*}
\sum_{1\leq \ell\leq D_\tau} \E \croch{ T_{\ell}^q } 
&\leq 
2 \times (q!)  \croch{ \frac{14 M^2}{3} }^{q} D_{\tau}
\\
&= 
\frac{q!}{2} \times D_{\tau} \croch{ \frac{28 M^2}{3} }^2 \times \croch{ \frac{14 M^2}{3} }^{q-2}
\enspace .
\end{align*}
Then, condition \eqref{constraint.Bernstein.moment} of 
Bernstein's inequality holds true with 
\[ 
v = D_{\tau} \croch{ \frac{28 M^2}{3} }^2 
\quad \text{and} \quad 
c = \frac{14 M^2}{3}
\enspace . \]
Therefore, Bernstein's inequality \citep[Proposition~2.9]{Mas:2003:St-Flour} 
---which is recalled by Proposition~\ref{thm.Bernstein.ineq}--- shows 
that for every $x>0$, with probability at least $1- \e^{-x}$, 
\begin{align*}
 \norm{\Pi_{\tau} \varepsilon}^2  - \E\croch{ \norm{\Pi_{\tau} \varepsilon}^2  } 
& \leq  
\sqrt{2 v x} + c x
\\
&= 
\sqrt{2 D_{\tau} x} \frac{28 M^2}{3} + \frac{14 M^2}{3} x
\\
&= 
\frac{14 M^2}{3} \paren{ 2 \sqrt{2 D_{\tau} x} +  x }
\enspace .  \qquad \hfill \blackbox
\end{align*}

\subsubsection{Why do we need a new concentration inequality?}
\label{sec.Talagrand}

We now review previous concentration results for quantities 
such as $\norms{\Pi_{\tau} \varepsilon}^2$ or $\norms{\Pi_{\tau} \varepsilon}$, 
showing that they are not sufficient for our needs, 
hence requiring a new result such as Proposition~\ref{pro.conc.quad}. 

%
First, when $\varepsilon \in \R^n$ is a Gaussian isotropic vector, 
$\norms{\Pi_{\tau} \varepsilon}^2$ is a chi-square random variable 
for which concentration tools have been developed. 
Such results have been used by \citet{Bir_Mas:2002} and 
by \citet{Leb:2005} for instance. 
They cannot be applied here since $\varepsilon$ cannot be assumed Gaussian, 
and the $\varepsilon_j$ do not necessarily have the same variance. 

%
Second, Eq.~\eqref{eq.formule.terme-quad} shows that 
$ \norms{\Pi_{\tau} \varepsilon}^2 $ is a U-statistic of order 2. 
Some tight exponential concentration inequalities exist for such 
quantities when $\varepsilon_j \in \R$ \citep{Hou_Rey:2003} 
and when $\varepsilon_j$ belongs to a general measurable set \citep[Theorem~3.4.8]{Gin_Nic:2016}. 
In both results, a term of order $M^2 x^2$ appears in the deviations, 
which is too large because the proof of Theorem~\ref{thm.oracle.large} 
relies on Proposition~\ref{pro.conc.quad} with $x \gg D_{\tau}$:  
we really need a smaller deviation term, 
as in Proposition~\ref{pro.conc.quad} where it is proportional to $M^2 x$.

%
Third, since 
\begin{align*}
\norms{\Pi_\tau \varepsilon} 
= \sup_{f\in\H^n,\norm{f}=1} \absb{\prodscal{f}{\Pi_\tau\varepsilon}} 
= \sup_{f\in\H^n,\norm{f}=1} \absb{\sum_{i=1}^n \prodH{f_i}{(\Pi_\tau\varepsilon)_i}}
\enspace , 
\end{align*}
Talagrand's inequality~\cite[Corollary~12.12]{Bou_Lug_Mas:2011:livre} 
provides a concentration inequality for $\norms{\Pi_\tau \varepsilon}$ 
around its expectation. 
More precisely, we can get the following result, 
which is proved in supplementary material (Section~\ref{app.Talagrand-details}). 
\begin{proposition}
\label{pro.concentration.quadratic.talagrand}
If \eqref{hyp.donnees-bornees} holds true, then for every $x>0$ with probability at least $1-2 \e^{-x}$, 
\begin{equation} 
\label{eq.pro.concentration.quadratic.talagrand.1}
\absb{ \norms{\Pi_{\tau} \varepsilon } - \E\crochb{\norms{\Pi_{\tau} \varepsilon }} }
\leq \sqrt{ 2 x \paren{ 4 M  \E\crochb{\norms{\Pi_{\tau} \varepsilon }} + \max_{1\leq \ell \leq D_\tau}  v_{\ell}^{\tau} }} + \frac{2 M x}{3} 
\enspace . 
\end{equation} 
\end{proposition}
Therefore, in order to get a concentration inequality for $\norms{\Pi_{\tau} \varepsilon}^2$, 
we have to square Eq.~\eqref{eq.pro.concentration.quadratic.talagrand.1} and we necessarily get 
a deviation term of order $M^2 x^2$. 
As with the U-statistics approach, this is too large for our needs. 

%
Fourth, given Eq.~\eqref{eq.norm-proj-histo}, 
it is also natural to think of 
Pinelis-Sakhanenko's inequality~\citep{Pin:Sak:1986}, 
but this result alone is not precise enough because it is a 
\emph{deviation} inequality, 
and not a \emph{concentration} inequality. 
It is nevertheless a key ingredient in our proof of Proposition~\ref{pro.conc.quad}. 

\subsection{Oracle inequality: proof of Theorem~\ref{thm.oracle.large}}
\label{sec.proof.thm.oracle.large}

We now end the proof of Theorem~\ref{thm.oracle.large} as explained in Section~\ref{sec.proof.outline}. 

\noindent {\bf Upper bound on $\penid(\tau)$ for every $\tau \in \cT_n$.} 
First, by Eq.~\eqref{eq.penid} for every $\tau \in \cT_n$, 
\begin{equation}
\penid(\tau)
= \frac{1}{n} \paren{ \norm{\ERM_{\tau} - \mu^{\star}}^2 - \norm{\ERM_{\tau} - Y}^2 + \norm{\varepsilon}^2 }
= \frac{2}{n} \norm{ \Pi_{\tau} \varepsilon }^2 - \frac{2}{n} \prodscalb{ (I - \Pi_{\tau}) \mu^{\star} }{\varepsilon}
\enspace . 
 \label{eq.pr.thm.oracle-ineq.sansVmin.2}
\end{equation}
In other words, $\penid(\tau)$ is the sum of two terms, for which Propositions~\ref{pro.conc.quad} and~\ref{prop.concentration.lineaire} 
provide concentration inequalities. 

On the one hand, by Proposition~\ref{pro.conc.quad} under \eqref{hyp.donnees-bornees}, 
for every $\tau \in \cT_n$ and $x \geq 0$, with probability at least $1- \e^{-x}$ we have 
\begin{align}
\frac{2}{n} \norm{ \Pi_{\tau} \varepsilon }^2
&\leq \frac{2}{n} \paren{ \E\croch{ \norm{ \Pi_{\tau} \varepsilon }^2 } 
+ \frac{14 M^2}{3} \paren{ x + 2 \sqrt{2 x D_{\tau}}} } \label{eq.quad.upper.bound.first}
\\
&\leq \frac{2 M^2}{n} \paren{ D_{\tau} + \frac{14 x}{3} + \frac{28}{3} \sqrt{2 x D_{\tau}} }
 \label{eq.pr.thm.oracle-ineq.sansVmin.3}
\end{align}
since 
\[ 
\E\croch{ \norm{ \Pi_{\tau} \varepsilon }^2  } = \sum_{j=1}^{D_{\tau}} v^{\tau}_{j} \leq D_{\tau} M^2 
\]
by Eq.~\eqref{eq.E-terme-quadratique}. 
On the other hand, by Proposition~\ref{prop.concentration.lineaire} under \eqref{hyp.donnees-bornees}, 
for every $\tau \in \cT_n$ and $x \geq 0$, with probability at least $1- 2 \e^{-x}$ we have 
\begin{align}
\forall \theta>0, \qquad 
\frac{2}{n} \absB{ \prodscalb{ (I - \Pi_{\tau}) \mu^{\star} }{\varepsilon} }
&\leq 
 \frac{2 \theta}{n} \norm{\Pi_{\tau}\bayes - \bayes}^2 + \frac{2}{n} \paren{ \frac{\vmax}{2 \theta} + \frac{4 M^2}{3} } x 
\notag 
\\
&\leq \frac{2 \theta}{n} \norm{\Pi_{\tau}\bayes - \bayes}^2
+ \frac{x M^2}{n} \paren{ \theta^{-1} + \frac{8}{3} }
 \enspace . 
  \label{eq.pr.thm.oracle-ineq.sansVmin.4}
\end{align}

For every $\tau \in \cT_n$ and $x \geq 0$, let $\Omega_x^{\tau}$ be the event on which 
Eq.~\eqref{eq.pr.thm.oracle-ineq.sansVmin.3} and~\eqref{eq.pr.thm.oracle-ineq.sansVmin.4} 
hold true. 
A union bound shows that $\P(\Omega_x^{\tau}) \geq 1 - 3 \e^{-x}$. 
Furthermore, combining Eq.~\eqref{eq.pr.thm.oracle-ineq.sansVmin.2}, \eqref{eq.pr.thm.oracle-ineq.sansVmin.3} and~\eqref{eq.pr.thm.oracle-ineq.sansVmin.4} 
shows that on $\Omega_x^{\tau}$, for every $\theta>0$, 
\begin{align}
\notag 
\penid(\tau)
&\leq \frac{2 M^2}{n} \parenj{ D_{\tau} + \frac{14 x}{3} + \frac{28}{3} \sqrt{2 x D_{\tau}} }
+ \frac{2 \theta}{n} \norms{\Pi_{\tau}\bayes - \bayes}^2
+ \frac{x M^2}{n} \parenj{ \theta^{-1} + \frac{8}{3} }
\\
&\leq 2 \theta  \Risk{\ERM_{\tau}} 
+ \frac{M^2}{n} \crochj{ 2 D_{\tau} + \parenj{ \theta^{-1} + \frac{36}{3} } x + \frac{56}{3} \sqrt{2 x D_{\tau}} }
  \label{eq.pr.thm.oracle-ineq.sansVmin.5}
\end{align}
using that $n^{-1} \norms{\Pi_{\tau}\bayes - \bayes}^2 = \Risks{\Pi_{\tau}\bayes} \leq \Risks{ \ERM_{\tau}}$ 
by definition of the orthogonal projection $\Pi_{\tau}$, 
and 
\begin{align}
\notag 
\penid(\tau)
&\geq - \frac{2}{n} \prodscalb{ (I - \Pi_{\tau}) \mu^{\star} }{\varepsilon}
\\
\notag 
&\geq  - \frac{2 \theta}{n} \norms{\Pi_{\tau}\bayes - \bayes}^2
- \frac{x M^2}{n} \parenj{ \theta^{-1} + \frac{8}{3} }
\\
&\geq 
- 2 \theta  \Risk{\ERM_{\tau}} 
- \frac{x M^2}{n} \parenj{ \theta^{-1} + \frac{8}{3} }
\enspace . 
\label{eq.pr.thm.oracle-ineq.sansVmin.6}
\end{align}

\medskip

\noindent \textbf{Union bound over the models and conclusion.} 
Let $y \geq 0$ be fixed and let us define the event 
$\Omega_y = \bigcap_{\tau \in \cT_n} \Omega_{x(\tau,y)}^{\tau}$ where 
for every $\tau \in \cT_n$, 
\[ x(\tau,y) \egaldef y  + \log\paren{\frac{3}{\e - 1 }} + D_{\tau} + \log \binom{n-1}{D_{\tau} - 1} 
\enspace . \]
Then, since 
\[ \card\set{\tau \in \cT_n\, | \, D_{\tau}=D} = \binom{n-1}{D - 1} \]
for every $D \in \{1, \ldots, n\}$, 
a union bound shows that 
\begin{align*}
\P(\Omega_y) 
\geq 1 - \sum_{\tau \in \cT_n} \P \parenb{ \, \overline \Omega_{x(\tau,y)}^{\tau} }
\geq 
1 - 3 \sum_{D=1}^n \e^{-y - \log\paren{\frac{3}{\e - 1 }} - D}
&
= 1 - (\e - 1) \e^{-y} \sum_{D=1}^n \e^{-D}
\\
&\geq 
1 - \e^{-y}
\enspace . 
\end{align*}

In addition, on $\Omega_y$, for every $\tau \in \cT_n$, 
since Eq.~\eqref{eq.pr.thm.oracle-ineq.sansVmin.5} and~\eqref{eq.pr.thm.oracle-ineq.sansVmin.6} hold true with 
$x=x(\tau,y) \geq D_{\tau}$, 
taking $\theta = 1/6$, we get that 
\begin{align}
- \frac{26}{3} \frac{M^2 x(\tau,y)}{n} 
-\frac{1}{3}  \Risk{\ERM_{\tau}} 
\leq
\penid(\tau)
\leq 
\frac{1}{3}  \Risk{\ERM_{\tau}} 
+ \parenj{ 20 + \frac{56 \sqrt{2}}{3} }  \frac{M^2 x(\tau,y)}{n} 
\enspace . 
\notag 
\end{align}
Let us define 
\[ 
\kappa_1 \egaldef 20 + \frac{56 \sqrt{2}}{3}
\qquad \text{and} \qquad 
\kappa_2 \egaldef \frac{26}{3}
\enspace , 
\]
and assume that $C \geq \kappa_1$. 
Then, using Eq.~\eqref{eq.thm.oracle-ineq.sansVmin.hyp-pen}, 
we have 
\begin{align*}
\penid(\tau)
&\leq 
\frac{1}{3}  \Risk{\ERM_{\tau}} 
+  \pen(\tau) +  \frac{\kappa_1 M^2 \croch{ y + \log(3/(\e-1))}}{n} 
\\
\penid(\tau)
&\geq 
-\frac{1}{3}  \Risk{\ERM_{\tau}} 
- \frac{\kappa_2}{C} \pen(\tau) 
- \frac{\kappa_2 M^2 \croch{ y + \log(3/(\e-1))}}{n}
\enspace . 
\end{align*}
Therefore, by Eq.~\eqref{eq.pr.thm.oracle-ineq.sansVmin.1}, on $\Omega_y$, for every $\tau \in \cT_n$, 
\[
\frac{2}{3} \Risk{\ERM_{\tauh}} -\frac{ \kappa_1 M^2 \croch{ y + \log(3/(\e-1))}}{n} 
\leq \frac{4}{3} \Risk{\ERM_{\tau}} + \paren{1 + \frac{\kappa_2}{C} } \pen(\tau) 
+ \frac{\kappa_2 M^2 \croch{ y + \log(3/(\e-1))}}{n}
\]
hence 
\begin{align*}
\frac{2}{3} \Risk{\ERM_{\tauh}}  
&\leq \frac{4}{3} \Risk{\ERM_{\tau}} + \paren{1 + \frac{\kappa_2}{C} } \pen(\tau) 
+ \frac{(\kappa_1 + \kappa_2) M^2 \croch{ y + \log(3/(\e-1))}}{n}
\\
&\leq \frac{4}{3} \Risk{\ERM_{\tau}} + \paren{1 + \frac{\kappa_2 + (\kappa_1 + \kappa_2) \log(3/(\e-1))}{C} } \pen(\tau) 
+ (\kappa_1 + \kappa_2) \frac{M^2 y}{n}
\end{align*}
since 
$\pen(\tau) \geq C M^2/n$ for every $\tau \in \cT_n$. 
Multiplying both sides by $3/2$, we get that if $C \geq \kappa_1$, on $\Omega_y$, 
\begin{align*}
\Risk{\ERM_{\tauh}}  
&\leq 
\inf_{\tau \in \cT_n} \set{ 2 \Risk{\ERM_{\tau}} 
+ \frac{3}{2} \paren{1 + \frac{\kappa_2 + (\kappa_1 + \kappa_2) \log(3/(\e-1))}{C} } \pen(\tau) 
}
+ \frac{3 (\kappa_1 + \kappa_2)}{2} \frac{M^2 y}{n}
\enspace . 
\end{align*}
Let us finally define 
\[ 
L_1 \egaldef 3 \crochB{ \kappa_2 + (\kappa_1 + \kappa_2) \log\parenb{ 3/(\e-1) } } \geq \kappa_1 
\]
so that 
\[
\frac{3}{2} \paren{1 + \frac{\kappa_2 + (\kappa_1 + \kappa_2) \log(3/(\e-1))}{L_1} } = 2
\enspace . 
\]
Then, we get that if $C \geq L_1$, on $\Omega_y$, 
\begin{align*}
\Risk{\ERM_{\tauh}}  
&\leq 
2 \inf_{\tau \in \cT_n} \set{ \Risk{\ERM_{\tau}} 
+  \pen(\tau) 
}
+ \frac{3 (\kappa_1 + \kappa_2)}{2} \frac{M^2 y}{n}
\end{align*}
and the result follows. \hfill \blackbox


\section{Experiments on synthetic data}
\label{sec.synthetic.data}

This section reports the results of some experiments on synthetic data 
that illustrate the performance of KCP. 

\subsection{Data-generation process}\label{subsubsec.synthetic.framework}

Three scenarios are considered: 
$(i)$ real-valued data with a changing (mean,variance), 
$(ii)$ real-valued data with constant mean and variance, and 
$(iii)$ histogram-valued data as in Example~\ref{ex.histos}.

In the three scenarios, the sample size is $n=1\,000$ 
and the true segmentation $\tau^{\star}$ 
is made of $D^{\star}=11$ segments, with change-points 
$ \tau^{\star}_1 = 100$, $\tau^{\star}_2=130$, $\tau^{\star}_3=220$, $\tau^{\star}_4=320$, $\tau^{\star}_5=370$, $\tau^{\star}_6= 520$,  $\tau^{\star}_7=620$, $\tau^{\star}_8=740$, 
$\tau^{\star}_9=790$, $\tau^{\star}_{10} = 870$ 
(see Figure~\ref{fig.ex}). 
For each sample, we choose randomly the distribution of the $X_i$ 
within each segment of $\tau^{\star}$ as detailed below; 
note that we always make sure that the distribution of $X_i$ 
does change at each change-point $\tau^{\star}_{\ell}$. 

For each scenario, we generate $N=500$ independent samples, 
from which we estimate all quantities that are 
reported in Section~\ref{sec.synthetic.data.results}. 

\paragraph{Scenario~1: Real-valued data with changing (mean, variance).}
The distribution of $X_i \in \R$ 
is randomly picked out from: $\mathcal{B}(10, 0.2)$ (Binomial), $\mathcal{NB}\parens{3, 0.7}$ (Negative-Binomial), $\mathcal{H}(10, 5, 2)$ (Hypergeometric),
$\mathcal{N}(2.5, 0.25)$ (Gaussian), $\gamma\parens{0.5, 5}$ (Gamma), $\mathcal{W}(5,2)$ (Weibull) and $\mathcal{P}ar(1.5,3)$ (Pareto). 
Note that the pair (mean, variance) in each segment changes from that of its neighbors. 
Table~\ref{tab.Sc1.moy-var} summarizes its values.

The distribution within segment $\ell \in \sets{1, \ldots , D^{\star}}$ 
is given by the realization of a random variable 
$S_{\ell} \in \sets{1,\ldots,7}$, 
each integer representing one of the 7 possible distributions. 
The variables $S_{\ell}$ are generated as follows: 
$S_1$ is uniformly chosen among 
$\sets{1,\ldots,7}$, 
and for every $\ell \in \sets{1, \ldots , D^{\star}-1}$, 
given $S_{\ell}$, 
$S_{\ell+1}$ is uniformly chosen among 
$\sets{1,\ldots,7} \backslash \sets{S_{\ell}}$.  
Figure~\ref{fig.ex.Sc1} 
shows one sample generated according to this scenario. 
\begin{figure}
  \centering
  \subfloat[Scenario~1: changes in the pair (mean, variance).]{\label{fig.ex.Sc1} 
  \includegraphics[width = .3\textwidth]{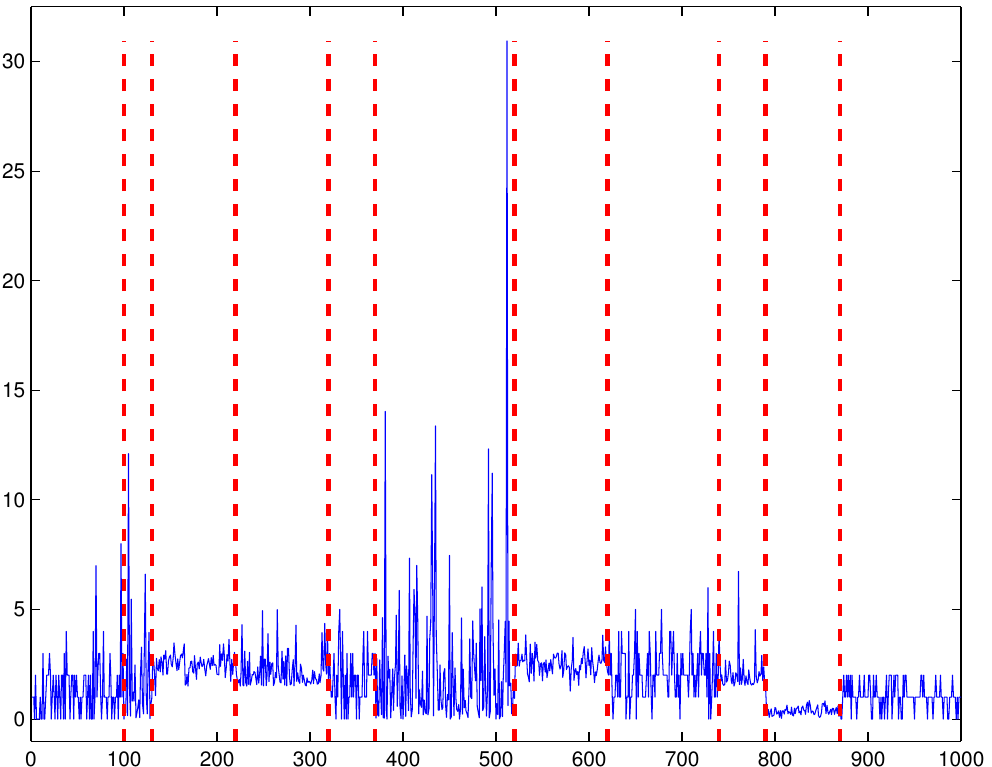}
}
 \hspace*{0.01\textwidth}
  \subfloat[Scenario~2: constant mean and variance.]{\label{fig.ex.Sc2}
  \includegraphics[width = .3\textwidth]{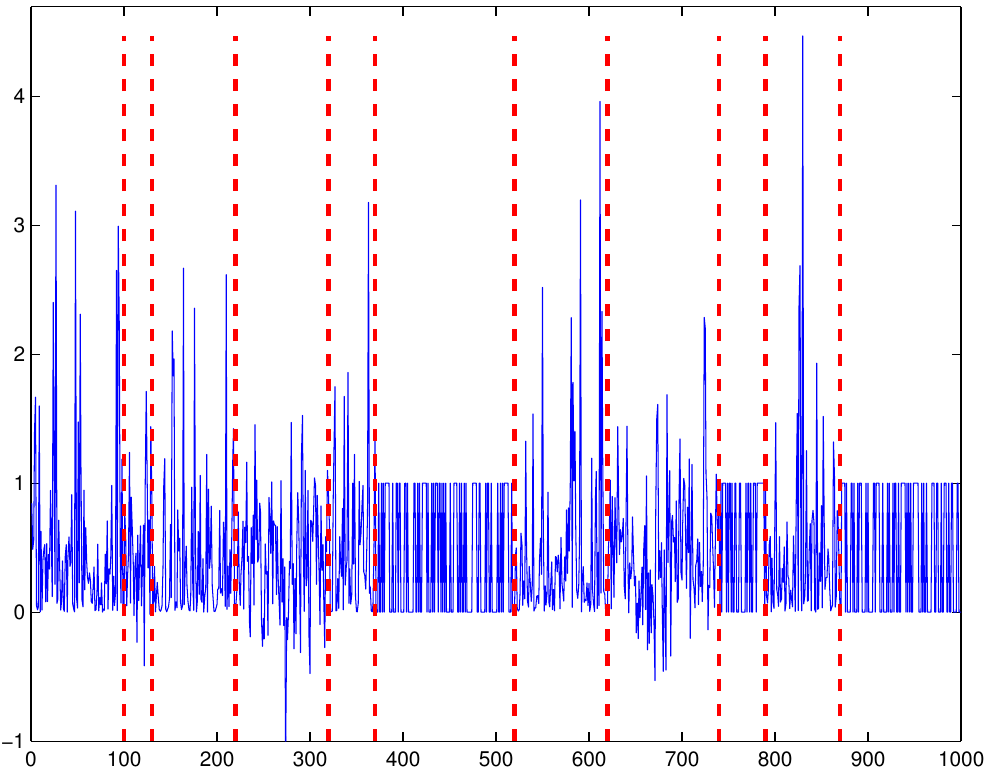}
}
 \hspace*{0.01\textwidth}
  \subfloat[Scenario~3: histogram-valued data (first three coordinates).]{\label{fig.ex.Sc3}
  \includegraphics[width = .3\textwidth]{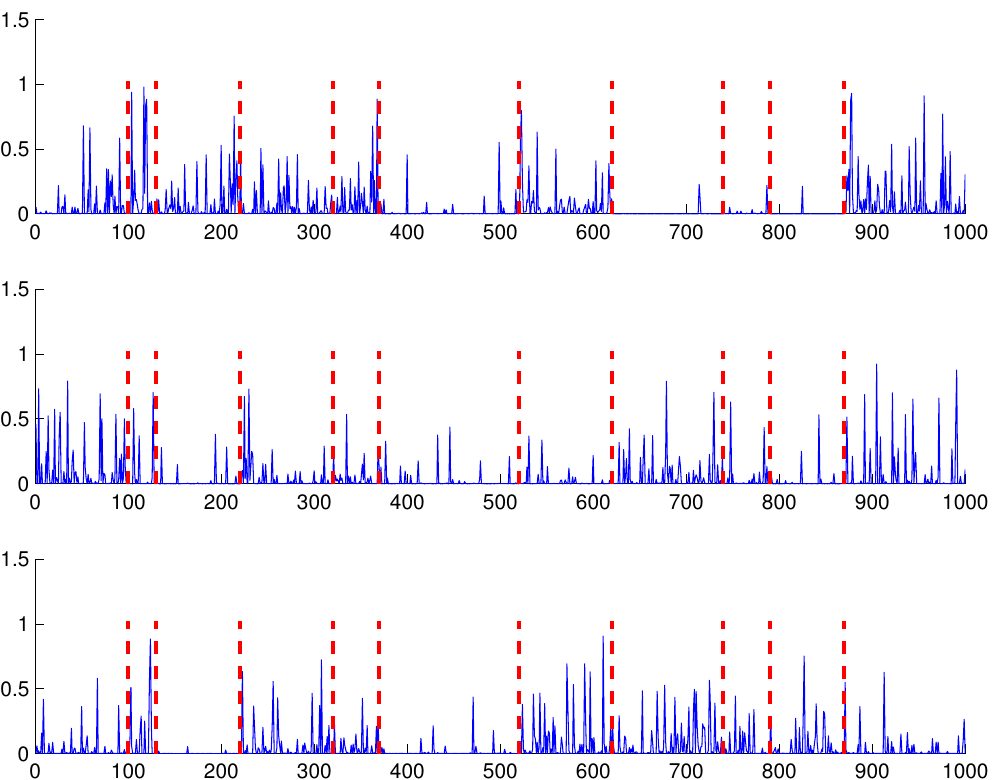}
}
  \caption{Examples of generated signals (blue plain curve) in the three scenarios. 
Red vertical dashed lines visualize the true change-points locations. }
  \label{fig.ex} 
\end{figure}

\paragraph{Scenario~2: Real-valued data with constant mean and variance.}
The distribution of $X_i \in \R$ 
is randomly chosen among (1) $\mathcal{B}(0.5)$ (Bernoulli), (2) $\mathcal{N}(0.5, 0.25)$ (Gaussian) and (3) $\mathcal{E}\paren{0.5}$ (Exponential).
These three distributions have a mean $0.5$ and a variance $0.25$. 

The distribution within segment $\ell \in \sets{1, \ldots , D^{\star}}$ 
is given by the realization of a random variable 
$S_{\ell} \in \sets{1,2,3}$, 
similarly to what is done in Scenario~1 (replacing $7$ by~$3$). 
Figure~\ref{fig.ex.Sc2} 
shows one sample generated according to this scenario. 

\paragraph{Scenario~3: Histogram-valued data.}
The observations $X_i$ belong to the $d$-dimensional simplex with $d=20$
(Example~\ref{ex.histos}), that is, $X_i = (a_1,\ldots,a_d)\in[0,1]^d$ with 
$\sum_{j=1}^d a_j = 1$.
For each $\ell \in \sets{1, \ldots, \Ds}$, we randomly generate $d$ parameter values $p^\ell_1, \ldots, p^\ell_d$ independently with uniform distribution over $[0,c_3]$ with $c_3 = 0.2$\,. 
Then, within the $\ell$-th segment of $\tau^{\star}$, 
$X_i$ follows a Dirichlet distribution with parameter 
$(p^\ell_{1}, \ldots, p^\ell_{d})$. 
Figure~\ref{fig.ex.Sc3} displays the first three coordinates of one sample generated according to this scenario.

\subsection{Parameters of KCP}\label{subsubsec.synthetic.proc}
For each sample, we apply our kernel change-point procedure 
(KCP, that is, Algorithm~\ref{algo}) with the following choices for 
its parameters. 
We always take $D_{\max} = 100$. 

%
For the first two scenarios, 
we consider three kernels: 
\begin{enumerate}
\item[(i)] The linear kernel $\klin (x,y) = xy$.  
\item[(ii)] The Hermite kernel given by $\kHer_{\sigma_H}(x,y)$ defined in Section~\ref{sec.chpt.high-dim.examples}. 
In Scenario~1, $\sigma_H=1$. In Scenario~2, $\sigma_H=0.1$. 
\item[(iii)] The Gaussian kernel $\kGau_{\sigma_G}$ 
defined in Section~\ref{sec.chpt.high-dim.examples}. In Scenario~1, $\sigma_G = 0.1$. In Scenario~2, $\sigma_G = 0.16$.
\end{enumerate}
For Scenario~3, 
we consider the $\chi^2$ kernel $\kchi_{0.1}(x,y)$ 
defined in Section~\ref{sec.chpt.high-dim.examples}, 
and the Gaussian kernel 
$\kGau_{\sigma_G}$ with $\sigma_G=1$.

In each scenario several candidate values have been explored for the bandwidth parameters of the above kernels. We have selected the ones with the most representative results.

\medbreak

%
For choosing the constants $c_1,c_2$ arising from Step~2 of KCP, 
we use the ``slope heuristics'' method, 
and more precisely a variant proposed by 
\citet[Section~4.3.2]{Leb:2002} 
for the calibration of two constants for change-point detection. 
We first perform a linear regression of 
$\Remp{\tauh(D)}$ against $ 1/n \cdot \log\binom{n-1}{D-1}$ and $D/n$ for 
$ D \in [0.6 \times D_{\max} , D_{\max}]$. 
Then, denoting by $\sh_1,\sh_2$ the coefficients obtained, 
we define $c_i = -\alpha \sh_i$ for $i=1,2$, 
with $\alpha=2$. 
The slope heuristics has been justified theoretically in various 
settings \citep[for instance by][for regressograms]{ArMa09}, 
and is supported by numerous experiments \citep{2012_BauMauMich}, 
including for change-point detection \citep{Leb:2002,Leb:2005}. 
A partial theoretical justification has been obtained recently for 
change-point detection \citep{Sor:2017}. 
The intuition behind the slope heuristics is that 
the optimal amount of penalization needed for avoiding to overfit 
with $\tauh \in \argmin_{\tau} \{ \Remp{\tau} + \pen(\tau) \}$ 
is (approximately) proportional to the minimal penalty: 
\[
\pen_{\mathrm{optimal}} (\tau) \approx \alpha \pen_{\mathrm{minimal}} (\tau)
\]
for some constant $\alpha >1$, equal to~$2$ in several settings. 
The linear regression step described above corresponds to estimating the 
minimal penalty: 
\[
\pen_{\mathrm{minimal}} (\tau) \approx - \sh_1  \cdot \frac{1}{n}\log \binom{n-1}{D_{\tau}-1} - \sh_2 \frac{D_{\tau}}{n}
\, \cdot
\]
Then, multiplying it by $\alpha$ leads to an estimation of the optimal penalty. 
In our experiments, we considered several values of $\alpha \in [0.8, 2.5 ]$. 
Remarkably, the performance of the procedure is not too sensitive to the value of $\alpha$ provided $\alpha \in [ 1.7 , 2.2 ]$. 
We only report the results for $\alpha=2$ because 
it corresponds to the classical advice 
when using the slope heuristics, and it is among the best choices for $\alpha$ according to our experiments.

\subsection{Results}
\label{sec.synthetic.data.results}
We now summarize the results of our experiments. 
\paragraph{Distance between segmentations.} 
In order to assess the quality of the segmentation $\tauh$ 
as an estimator of the true segmentation $\taus$, 
we consider two measures of distance between segmentations. 
For any $\tau,\tau' \in \cT_n$, 
we define the Hausdorff distance between $\tau$ and $\tau'$ by 
\[
d_H(\tau,\tau') 
\egaldef 
\max\setj{ 
\max_{1 \leq i \leq D_{\tau}-1}
\min_{1 \leq j \leq D_{\tau'}-1}
\absj{ \tau_i - \tau'_j }
,
\max_{1 \leq j \leq D_{\tau'}-1}
\min_{1 \leq i \leq D_{\tau}-1}
\absj{ \tau_i - \tau'_j }
}
\]
and the 
Frobenius distance between $\tau$ and $\tau'$ \citep[see][]{Laj_Arl_Bac:2014:icml} by 
\begin{gather*}
d_F(\tau,\tau') \egaldef 
\norm{M^{\tau} - M^{\tau'}}_F 
= \sqrt{ \sum_{1 \leq i,j \leq n} \parens{ M^{\tau}_{i,j} - M^{\tau'}_{i,j}}^2 },
\\
\text{where} \qquad 
M^{\tau}_{i,j} = \frac{\un_{\acc{i \text{ and } j \text{ belong 
to the same segment of } \tau }} }{\card(\text{segment of } \tau \text{ containing } i \text{ and } j)}
\enspace . 
\end{gather*}
Note that $M^{\tau} = \Pi_{\tau}$ the projection matrix onto $F_{\tau}$ when $\cH=\R$, that is, for the linear kernel on $\X=\R$.
The Hausdorff distance is probably more classical in the 
change-point literature, 
but Figure~\ref{fig.Sc1.dist}  
shows that the Frobenius distance is more informative 
for comparing $(\tauh(D))_{D > \Ds}$. 
Indeed, when $D$ is already a bit larger than $\Ds$, 
adding false change-points makes the segmentation worse 
without increasing much $d_H$; 
on the contrary, $d_F^2$ readily takes into account these additional false change-points. 

\medbreak 

\paragraph{Illustration of KCP.}
%
%
\begin{figure}
  \centering
  \subfloat[Average distance ($d_F$ or $d_H$) between $\tauh(D)$ and $\taus$, as a function of $D$.]{\label{fig.Sc1.dist} 
  \includegraphics[width = .478\textwidth]{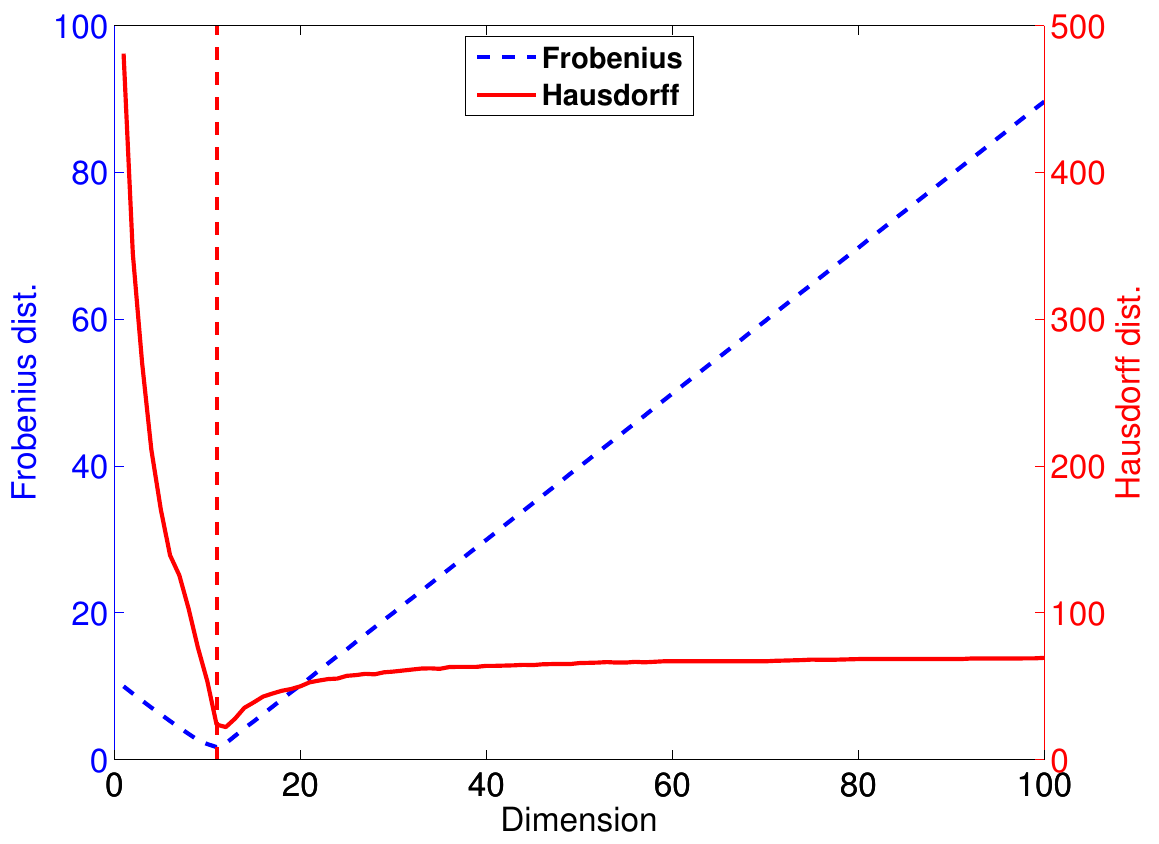}
}
 \hspace*{0.01\textwidth}
  \subfloat[Average risk $\Risk{\ERM_{\tauh(D)}}$, 
  empirical risk $\Remp{\ERM_{\tauh(D)}}$ and penalized criterion as a function of $D$.]{\label{fig.Sc1.risk}
  \includegraphics[width = .45\textwidth]{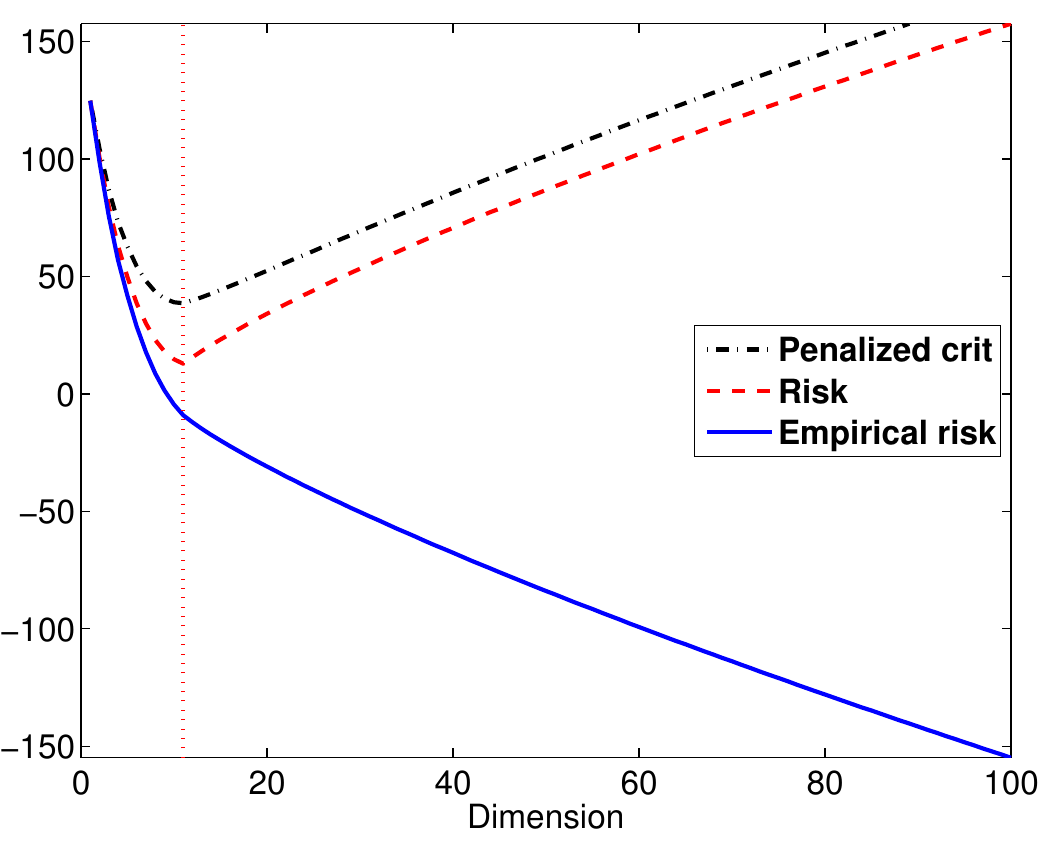}
}
\newline
  \subfloat[Distribution of $\Dh$.]{\label{fig.Sc1.Dh}
  \includegraphics[width = .45\textwidth]{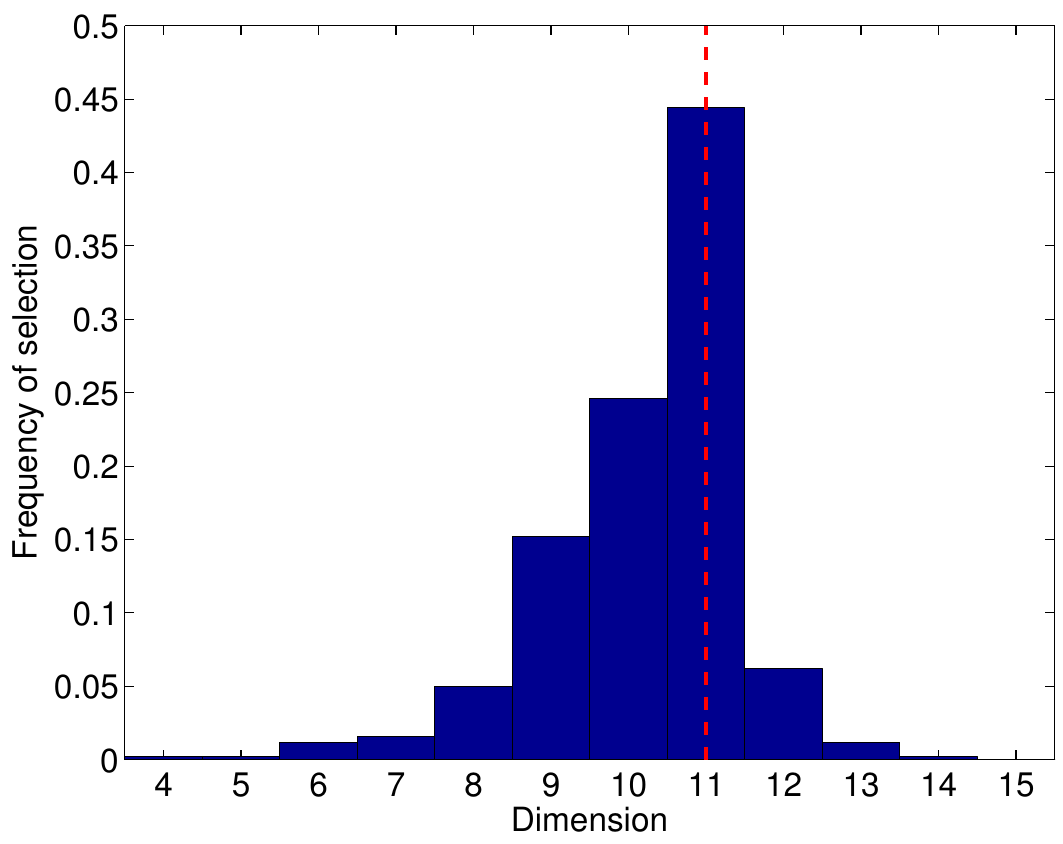}
}
 \hspace*{0.01\textwidth}
  \subfloat[Probability, for each instant $i \in \sets{1, \ldots, n}$, that $\tauh = \tauh(\Dh)$ puts a change-point at $i$.]{\label{fig.Sc1.freq} 
  \includegraphics[width = .45\textwidth]{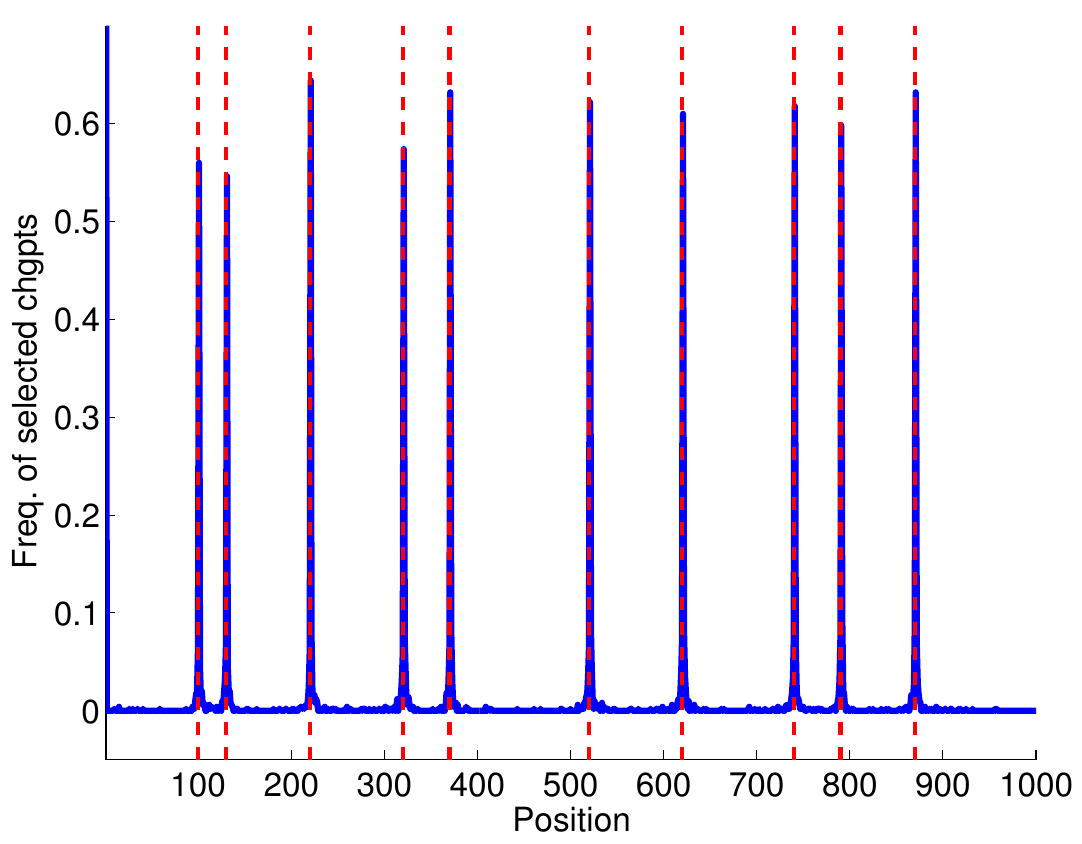}
}
  \caption{
Scenario~1: $\X=\R$, variable (mean, variance). 
Performance of KCP with kernel $\kGau_{0.1}$. 
The value $\Ds$ and the localization of the true change-points in $\taus$ are materialized by vertical red lines. 
   }
  \label{fig.Sc1} 
\end{figure}
Figure~\ref{fig.Sc1} illustrates the typical behaviour of 
KCP when $k$ is well-suited to the 
change-point problem we consider. It summarizes results
obtained in Scenario~1 with $k=\kGau$ the Gaussian 
kernel. 

%
Figure~\ref{fig.Sc1.dist} shows the expected distance 
between the true segmentation $\taus$ and the segmentations 
$(\tauh(D))_{1 \leq D \leq D_{\max}}$ 
produced at Step~1 of KCP. 
As expected, the distance is clearly minimal at $D=\Ds$, 
for both Hausdorff and Frobenius distances. Note that for each individual sample, 
$d(\tauh(D),\taus)$ behaves exactly as the expectation 
shown on Figure~\ref{fig.Sc1.dist}, up to minor fluctuations. 
Moreover, the minimal value of the distance is small 
enough to suggest that $\tauh(\Ds)$ is indeed close to $\taus$. 
For instance, 
$\E\crochs{ d_F \parens{ \tauh(\Ds),\taus } } \approx 1.71$, 
with a $95\%$ error bar smaller than $0.11$. 
The closeness between $\tauh(\Ds)$ and $\taus$ when $k=\kGau$ 
can also be visualized on Figure~\ref{fig.Sc1.freq-Ds.kGau} in the supplementary material. 

As a comparison, when $k=\klin$ in the same setting, 
$\tauh(\Ds)$ is much further from $\taus$ since 
$\E\crochs{ d_F \parens{ \tauh(\Ds),\taus } } \approx 10.39 \pm 0.24$, 
and a permutation test shows that the difference is significant, with a p-value smaller than $10^{-13}$.
See also Figures~\ref{fig.Sc1.klin.dist} and~\ref{fig.Sc1.freq-Ds.klin} 
in the supplementary material. 

%
Step~2 of KCP is illustrated by 
Figures~\ref{fig.Sc1.risk} and~\ref{fig.Sc1.Dh}. 
The expectation of the penalized criterion is minimal 
at $D=\Ds$ (as well as for the risk of $\ERM_{\tauh(D)}$), 
and takes significantly larger values when $D\neq \Ds$ 
(Figure~\ref{fig.Sc1.risk}). 
As a result, KCP often selects a number 
of change-points $\Dh-1$ close to its true value 
$\Ds-1$ (Figure~\ref{fig.Sc1.Dh}). 
Overall, this suggests that the model selection procedure 
used at Step~2 of KCP works fairly well. 

%
The overall performance of KCP as a change-point detection 
procedure is illustrated by Figure~\ref{fig.Sc1.freq}. 
Each true change-point has a probability larger than 
$0.5$ to be recovered \emph{exactly} by $\tauh$. 
If one groups the positions $i$ by blocks of six elements 
$\sets{6 j , 6j+1, \ldots, 6j+5}$, $j \geq 1$, 
the frequency of detection 
of a change-point by $\tauh$ 
in each block containing a true change-point 
is between $79$ and $89\%$.
%
%
Importantly, such figures are obtained without overestimating 
much the number of change-points, according to Figure~\ref{fig.Sc1.Dh}. 
Figures~\ref{fig.Sc1.klin.freq} and~\ref{fig.Sc1.kHer.freq} in the supplementary material 
show that more standard change-point detection algorithms ---that is, KCP with $k=\klin$ or $\kHer$--- have a slightly worse performance.

\medbreak 

\paragraph{Comparison of three kernels in Scenario~2.}
Scenario~2 proposes a more challenging change-point problem 
with real-valued data: 
the distribution of the $X_i$ changes while the mean 
\emph{and} the variance remain constant. 
The performance of KCP with three 
kernels ---$\klin$, $\kHer$ and $\kGau$--- is shown on 
Figure~\ref{fig.Sc2}. 
%
%
\begin{figure}
  \centering
\figtroishspace
%
  \subfloat[$k=\klin$]{\label{fig.Sc2.klin-freqDs} 
\includegraphics[width = \figtroiswidth]{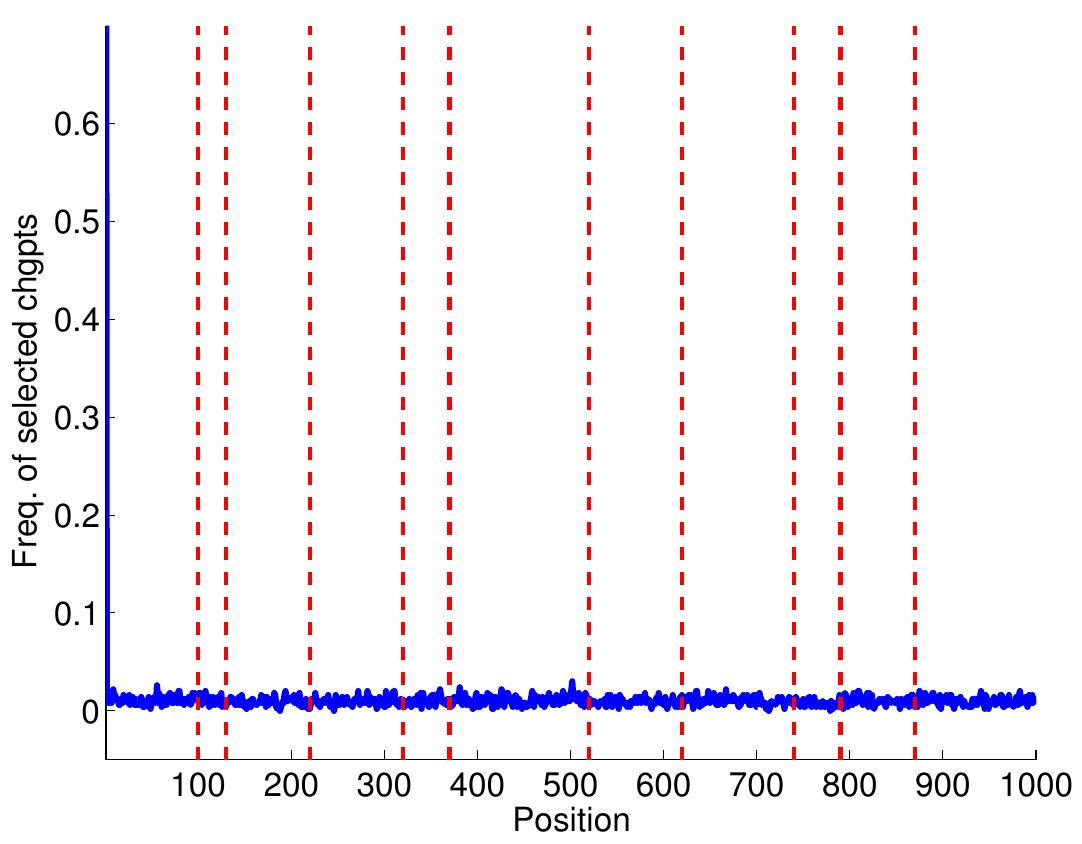}
}
 \hspace*{0.01\textwidth}
  \subfloat[$k=\kHer_{0.1}$]{\label{fig.Sc2.kHer-freqDs}
\includegraphics[width = \figtroiswidth]{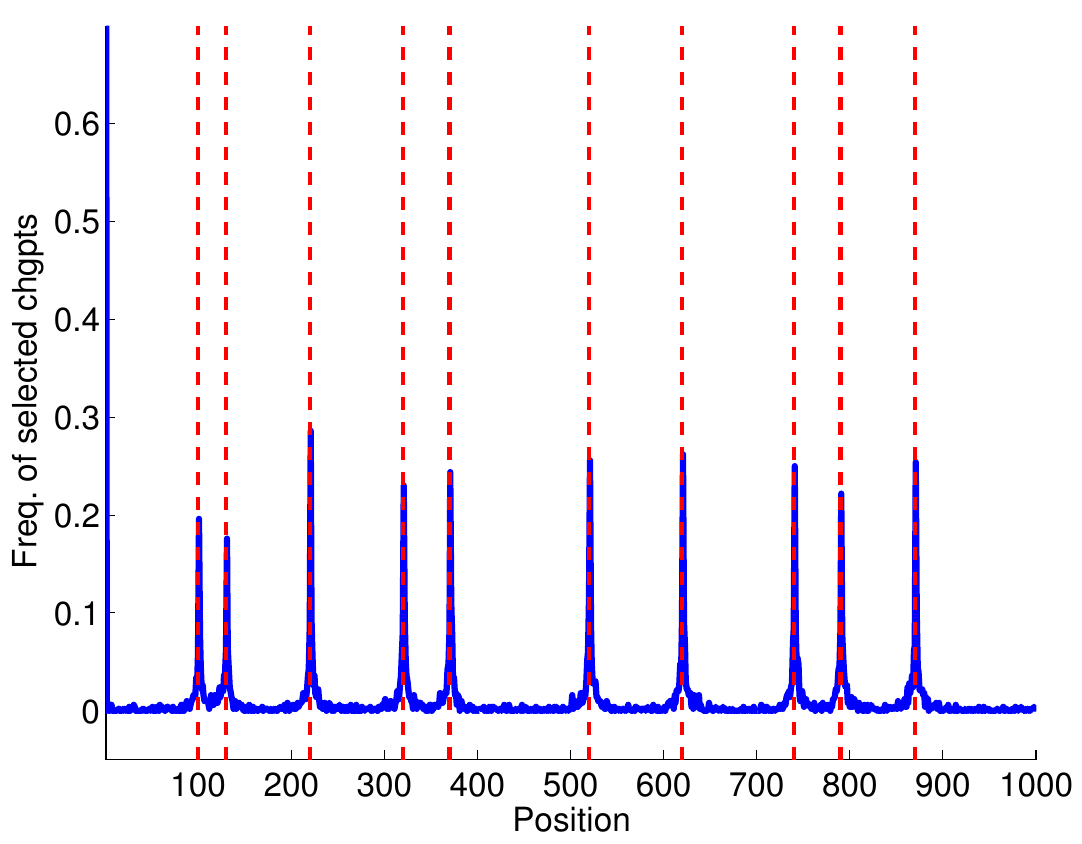}
}
 \hspace*{0.01\textwidth}
  \subfloat[$k=\kGau_{0.16}$]{\label{fig.Sc2.kGau-freqDs}
\includegraphics[width = \figtroiswidth]{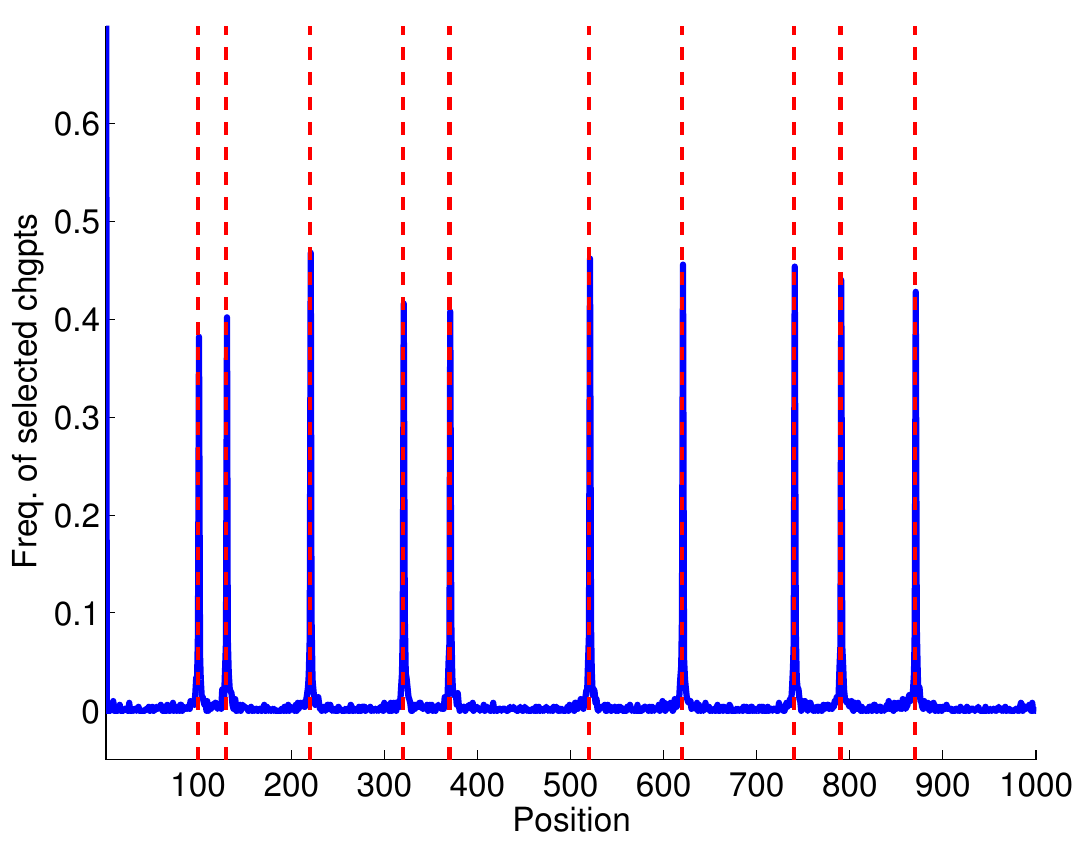}
}
\newline 
%
%
%
\figtroishspace
  \subfloat[$k=\klin$]{\label{fig.Sc2.klin-dist} 
\includegraphics[width=\figtroiswidth]{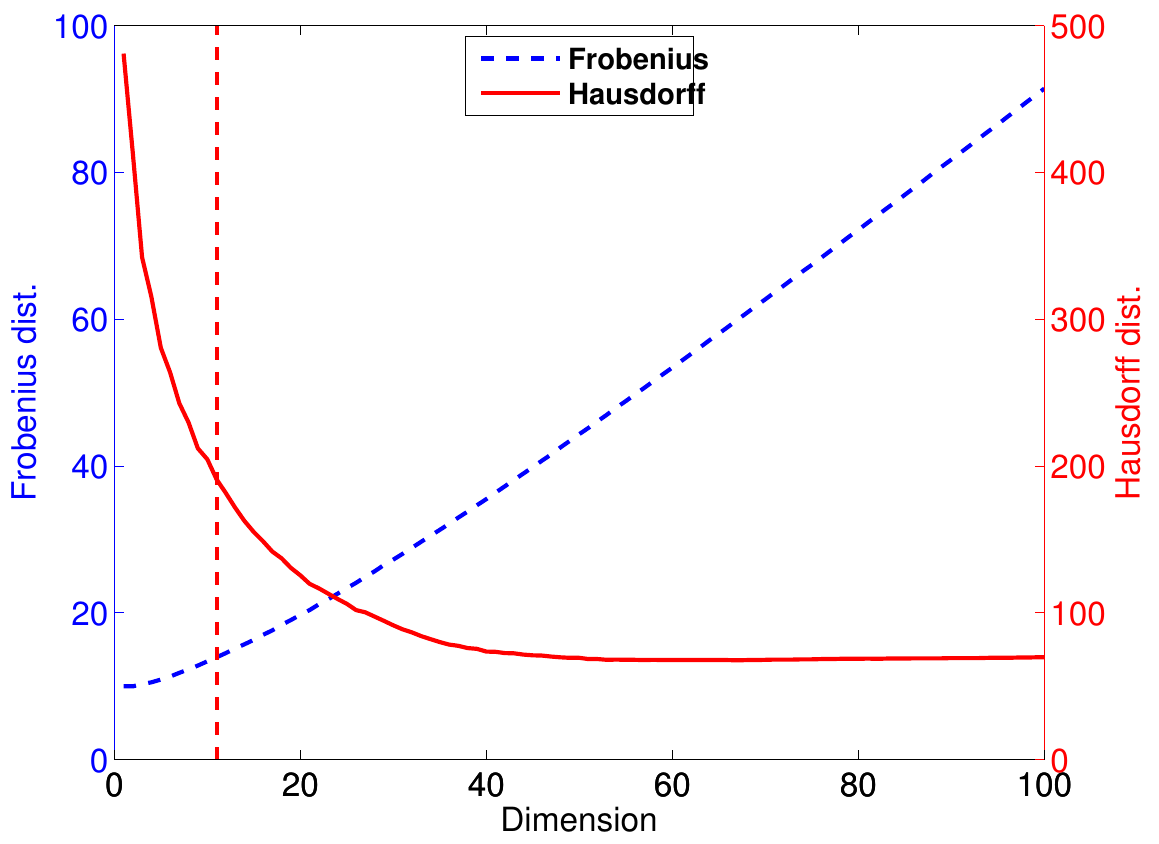}
}
 \hspace*{0.01\textwidth}
  \subfloat[$k=\kHer_{0.1}$]{\label{fig.Sc2.kHer-dist}
\includegraphics[width = \figtroiswidth]{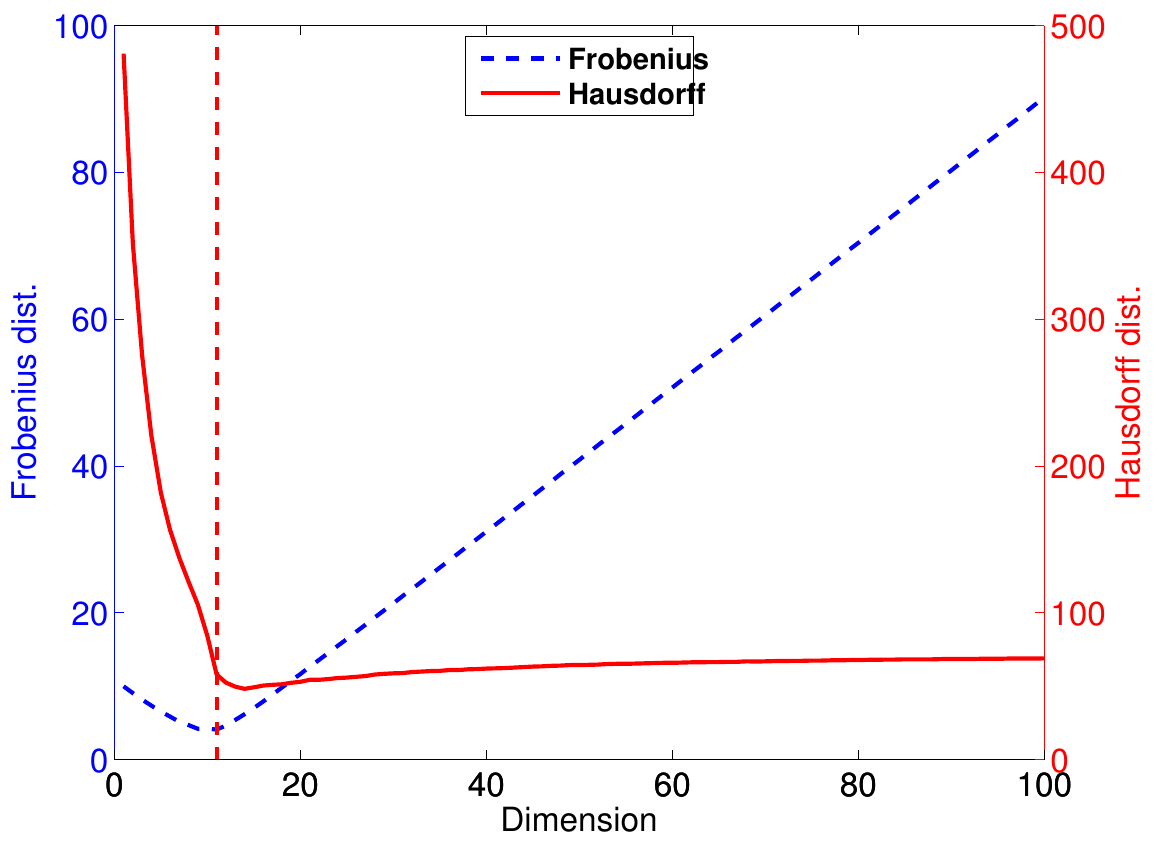}
}
 \hspace*{0.01\textwidth}
  \subfloat[$k=\kGau_{0.16}$]{\label{fig.Sc2.kGau-dist}
\includegraphics[width = \figtroiswidth]{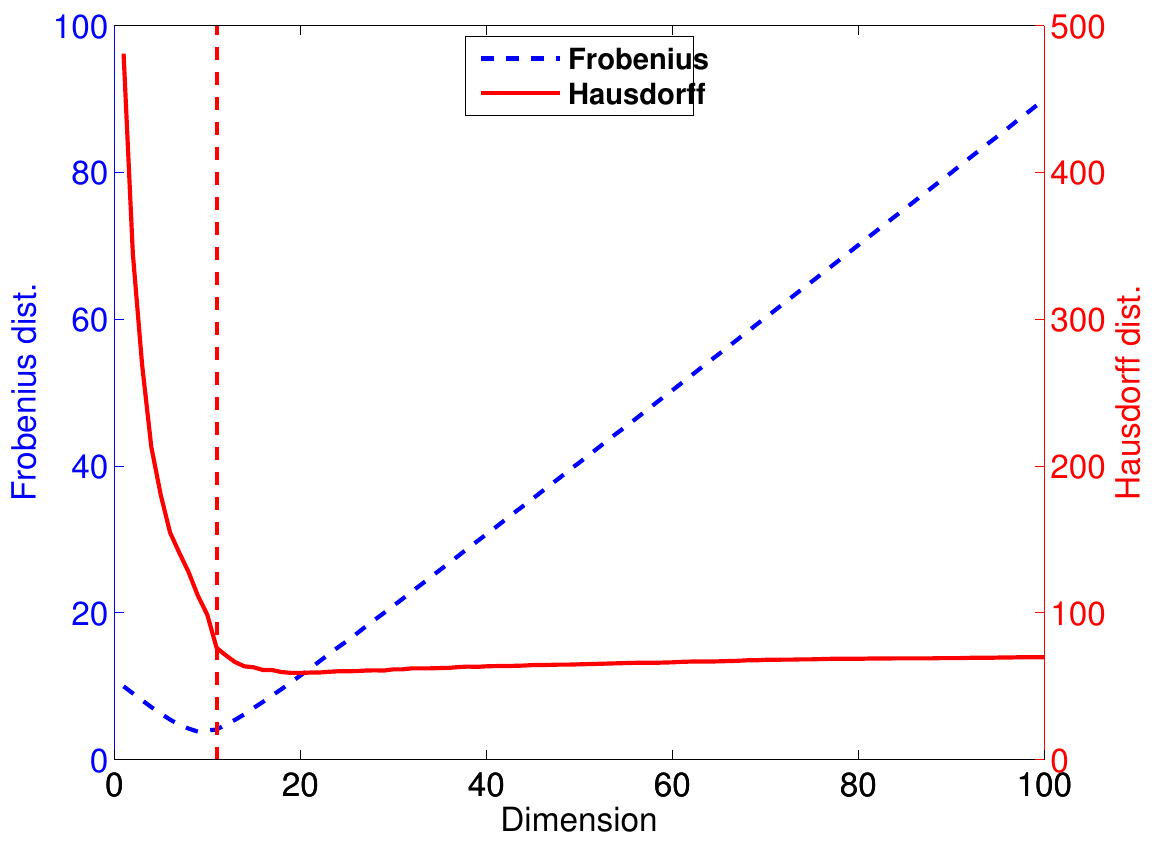}
}
\caption{%
Scenario 2: $\X=\R$, constant mean and variance. 
Performance of KCP 
with three different kernels $k$.  
The value $\Ds$ and the localization of the true change-points in $\taus$ are materialized by vertical red lines. 
\Top 
Probability, for each instant $i \in \sets{1, \ldots, n}$, that $\tauh(\Ds)$ puts a change-point at $i$.
\Bottom 
Average distance ($d_F$ or $d_H$) between $\tauh(D)$ and $\taus$, as a function of $D$.
}
\label{fig.Sc2}
\end{figure}
The linear kernel $\klin$ corresponds to the classical 
least-squares change-point algorithm \citep{Leb:2005}, 
which is designed to detect changes in the mean, 
hence it should fail in Scenario~2. 
KCP with the Hermite kernel $\kHer$ 
is a natural ``hand-made'' extension of this classical 
approach, since it corresponds to applying the 
least-squares change-point algorithm to the 
feature vectors $\parens{ H_{j,h}(X_i) }_{1 \leq j \leq 5}$. 
By construction, it should be able to detect changes in the first five moments on the $X_i$. 
On the contrary, taking $k=\kGau$ the Gaussian kernel 
fully relies on the versatility of KCP, 
which makes possible to consider (virtually) 
infinite-dimensional feature vectors $\kGau(X_i,\cdot)$. 
Since $\kGau$ is characteristic, it should be able to detect 
any change in the distribution of the $X_i$. 

%
In order to compare these three kernels within KCP, 
let us first assume that the number of change-points is known, 
hence we can estimate $\taus$ with $\tauh(\Ds)$, where $D^*$ is the true number of segments. 
Then, Figures~\ref{fig.Sc2.klin-freqDs}, \ref{fig.Sc2.kHer-freqDs} 
and~\ref{fig.Sc2.kGau-freqDs} show that $\klin$, $\kHer$ and $\kGau$ 
behave as expected: 
$\klin$ seems to put the change-points of $\tauh(\Ds)$ 
uniformly at random over $\sets{1, \ldots, n}$, 
while $\kHer$ and $\kGau$ are able to localize the true change-points 
with a rather large probability of success. 
The Gaussian kernel here shows a significantly better detection power, 
compared to $\kHer$: 
the frequency of exact detection of the true change-points is between 
$38$ and~$47\%$ with $\kGau$, 
and between $17$ and~$29\%$ with~$\kHer$. 
The same holds when considering blocks of size~$6$: 
$\kGau$ then detects the change-points with probability $70$ to~$79\%$, 
while $\kHer$ exhibits probabilities between $58$ and~$62\%$.

%
Figures~\ref{fig.Sc2.klin-dist}, \ref{fig.Sc2.kHer-dist} and~\ref{fig.Sc2.kGau-dist} 
show that a similar comparison between $\klin$, $\kHer$, and $\kGau$ holds 
over the whole set of segmentations $\parens{ \tauh(D) }_{1 \leq D \leq D_{\max}}$ 
provided by Step~1 of KCP. 
%
%
With the linear kernel (Figure~\ref{fig.Sc2.klin-dist}), 
the Frobenius distance between $\tauh(D)$ and $\taus$ is almost minimal for $D=1$, 
which suggests that $\tauh(D)$ is not far from random guessing for all~$D$. 
The shape of the Hausdorff distance ---first decreasing fastly, then almost constant--- also supports this interpretation: 
A small number of purely random guesses do lead to a fast decrease of $d_H$; 
and for large dimensions, adding a new random guess does not move away 
$\tauh(D)$ from $\taus$ if $\tauh(D)$ already contains all the worst possible 
candidate change-points (which are the furthest from the true change-points). 
%
%
The Hermite kernel does much better according to Figure~\ref{fig.Sc2.kHer-dist}: 
the Frobenius distance from $\tauh(D)$ to $\taus$ is minimal for $D$ close to $\Ds$, 
and the minimal expected distance, 
$\inf_D \E\crochs{ d_F \parens{ \tauh(D) , \taus} } \approx 4.12 \pm 0.6$ 
(with confidence $95\%$), 
is much smaller than when $k=\klin$ (in which case 
$\inf_D \E\crochs{ d_F \parens{ \tauh(D) , \taus} } \approx 10 $); 
this difference is significant (a permutation test yields a p-value 
smaller than $10^{-15}$). 
%
%
%
Nevertheless, we still obtain slightly better performance for 
$\parens{\tauh(D)}_{1 \leq D \leq D_{\max}}$ with $k=\kGau$, 
for which the minimal distance to $\taus$ is achieved at $D=9$, 
with a minimal expected value equal to $3.83 \pm 0.49$ 
(the difference between $\kHer$ and $\kGau$ is not statistically significant).
The Hausdorff distance suggests that both $\kHer$ and $\kGau$ lead to include false change-points among true ones as long as $D\leq \Ds$. 
However, the smaller Frobenius distance achieved by $\kGau$ at $D=9$ (rather than $D=11$ for $\kHer$) indicates 
that the corresponding change-points are closer to the true ones than those provided by $\kHer$ 
(which include more false positives).

%
When $D=\Dh$ is chosen by KCP, $\kGau$ clearly leads to the best performance in terms of recovering 
the exact change-points compared to $\klin$ and $\kHer$, 
as illustrated by Figures~\ref{fig.Sc2.klin-freqDh}, \ref{fig.Sc2.kHer-freqDh} 
and~\ref{fig.Sc2.kGau-freqDh} 
in the supplementary material. 

\medbreak 

%
Overall, the best performance for KCP in Scenario~2 is clearly obtained with 
$\kGau$, while $\klin$ completely fails 
and $\kHer$ yields a decent but suboptimal procedure. 

We can notice that other settings can lead to different behaviours. 
For instance, in Scenario~1, 
according to Figure~\ref{fig.Sc1.klin.freq} in the supplementary material, 
$\klin$ can detect fairly well the true change-points ---as expected since 
the mean (almost) always changes in this scenario, see Table~\ref{tab.Sc1.moy-var} in the supplementary material---, 
but this is at the price of a strong overestimation of the number of 
change-points (Figure~\ref{fig.Sc1.klin-Dh}). 
In the same setting, $\kHer$ provides fairly good results (Figure~\ref{fig.Sc1.kHer.freq}), 
while $\kGau$ remains the best choice (Figure~\ref{fig.Sc1.freq}). 

Since $\kGau$ is a characteristic kernel, 
these results suggest that KCP with a characteristic kernel $k$ 
might be more versatile than classical least-squares change-point 
algorithms and their extensions. 
A more detailed simulation experiment would nevertheless be needed to 
confirm this hypothesis. 
We also refer to Section~\ref{sec.conclusion.choix-k} for a discussion 
on the choice of $k$ for a given change-point problem. 

\medbreak 

\paragraph{Structured data.}
Figure~\ref{fig.Sc3} illustrates the performance 
of KCP on some histogram-valued data 
(Scenario~3). 
%
%
\begin{figure}
  \centering
  \subfloat[$k=\kchi_{0.1}$]{\label{fig.Sc3.chi2-Ds} 
\includegraphics[width = .45\textwidth]{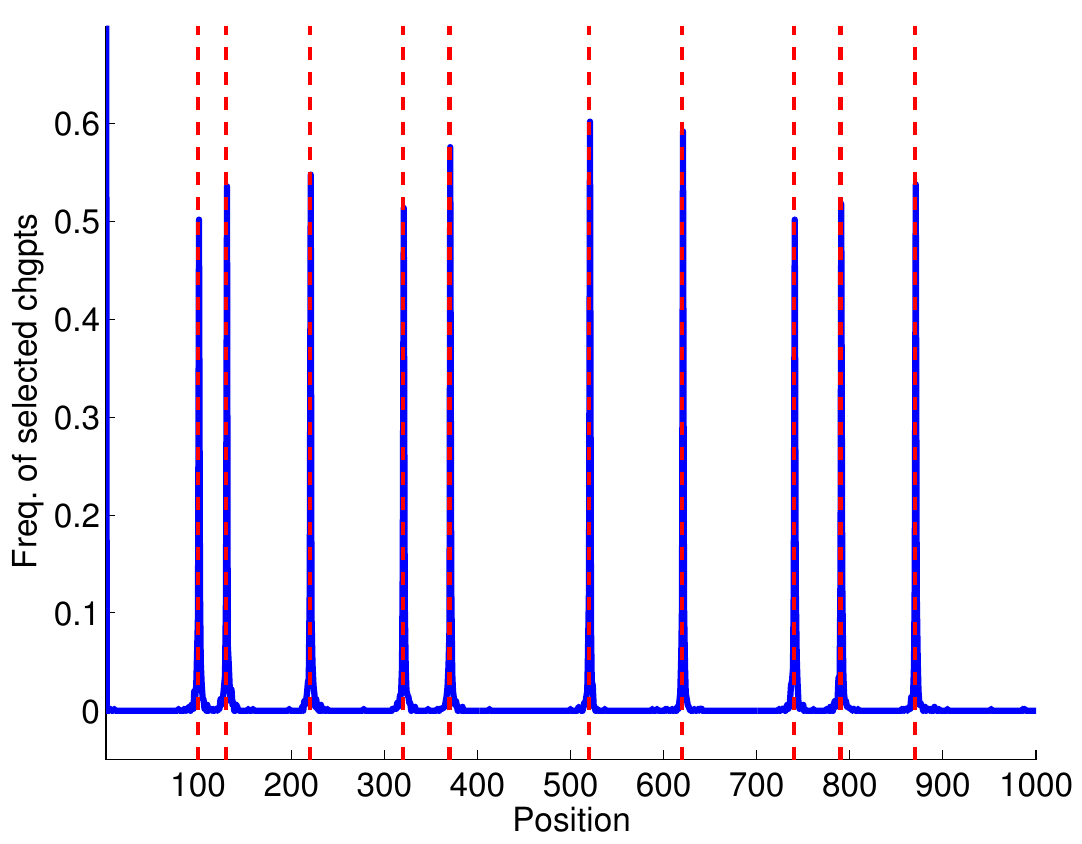}
}
 \hspace*{0.01\textwidth}
  \subfloat[$k=\kGau_{1}$]{\label{fig.Sc3.Gau-Ds}
\includegraphics[width = .45\textwidth]{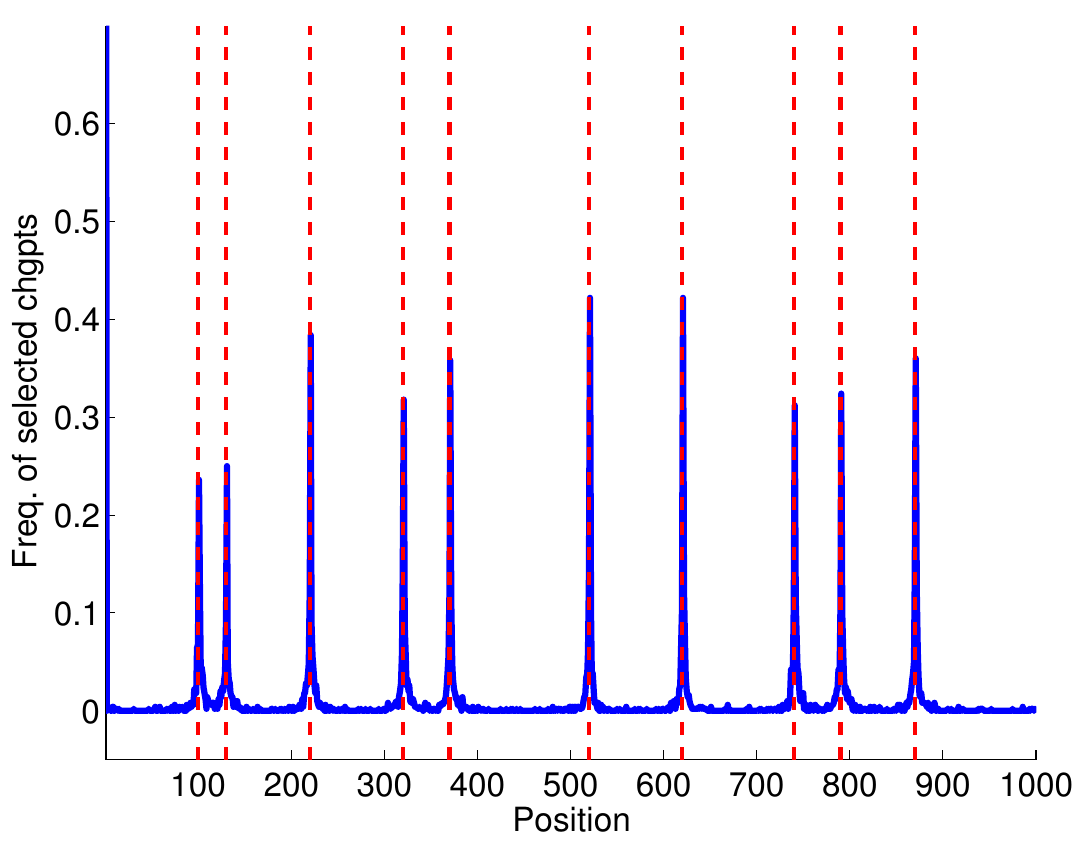}
}
  \caption{
Scenario~3: histogram-valued data. 
Performance of KCP with two different kernels $k$.   
Probability, for each instant $i \in \sets{1, \ldots, n}$, that $\tauh(\Ds)$ puts a change-point at $i$.
Vertical red lines show the true change-points locations. 
   }
  \label{fig.Sc3} 
\end{figure}
Since a $d$-dimensional histogram is also an element of $\R^d$, 
we can analyze such data either with a kernel taking 
into account the histogram structure (such as $\kchi$) 
or with a usual kernel on $\R^d$ (such as $\klin$ or $\kGau$; here, we consider $\kGau$, which seems more reliable according to our experiments in Scenarios~1 and~2). 
Assuming that the number of change-points is known, 
taking $k=\kchi$ yields quite good results according 
to Figure~\ref{fig.Sc3.chi2-Ds}, 
at least in comparison with $k=\kGau$ (Figure~\ref{fig.Sc3.Gau-Ds}). 
Similar results hold with a fully data-driven number 
of change-points, as shown by 
Figures~\ref{fig.Sc3.chi2-freqDh} and~\ref{fig.Sc3.Gau-freqDh} 
in the supplementary material. 
Hence, choosing a kernel such as $\kchi$, 
which takes into account the histogram structure of the $X_i$, 
can improve much the change-point detection performance, 
compared to taking a kernel such as $\kGau$, which ignores the structure of the $X_i$. 

Let us emphasize that Scenario~3 is quite challenging 
---changes are hard to distinguish on Figure~\ref{fig.ex.Sc3}---, 
which has been chosen on purpose.  
Preliminary experiments have been done with larger values of 
$c_3$ ---which makes the change-point problem easier, see Section~\ref{subsubsec.synthetic.framework}---, 
leading to an almost perfect localization of all 
change-points by KCP with $k=\kchi_{0.1}$.

\paragraph{Comparison to AIC/BIC-type penalty.} 
%
%
Figures \ref{fig.LinearPenalty.Dh} and~\ref{fig.LinearPenalty.freq} 
in the supplementary material 
show the results of KCP with a linear penalty 
---that is, of the form $C D / n$, $C>0$--- in step~2, 
similarly to AIC (which would correspond to $C=\sigma^2$) and BIC (for which $C=\log(n) \sigma^2 / 2$). 
Since $\sigma^2$ is unknown, we use the slope heuristics for choosing $C$ from data, 
as explained in Section~\ref{subsubsec.synthetic.proc}. 
The performance is comparable to the one of KCP, 
except that a linear penalty leads to overfitting 
---by detecting too many change-points (including false positives)--- 
with a large probability 
in Scenarios~1 and~2 (compare Figures~\ref{fig.Sc1.penlin-Dh} and \ref{fig.Sc1.kGau-Dh} for Scenario~1, 
and Figures~\ref{fig.Sc2.penlin-Dh} and \ref{fig.Sc2.kGau-Dh} for Scenario~2). 
Therefore, a linear penalty seems less reliable than the refined 
shape proposed in the definition of KCP, 
so we do not recommend to use a linear penalty in practice.

\paragraph{Comparison to the E-divise procedure (ED).} 
%
We finally consider the E-divisive procedure (ED) 
designed by \citet{Matteson_James:2014}, 
focusing on Scenarios 1--2 since this procedure is made for $\X=\R^d$ only. 
We use the \texttt{e.divisive} function from the \texttt{R}-package \texttt{ecp} 
described by \citet{JamesMatteson2015}, 
with recommended parameters $\mathtt{sig.lvl}=0.05$ 
(significance level to test any new change-point), 
$\alpha=1$, and $R=199$ permutations.
Detailed results are shown on Figures \ref{fig.EDivisive-Dh}, 
\ref{fig.EDivisive.freq-Dh} and~\ref{fig.EDivisive.freq-Ds} 
in the supplementary material. 
%
%
In both scenarios, ED provides much more conservative results 
---that is, it strongly underestimates the number of change-points--- 
compared to KCP with $k=\kGau$. 
This drawback of ED is particularly clear in Scenario~2 (more difficult) from the comparison of Figures \ref{fig.Sc2.kGau-Dh} and~\ref{fig.Sc2.EDivisive-Dh}. 
As a result, ED's detection power is much smaller than the one 
of KCP, 
with detection frequencies 2 to 5 times lower for ED in Scenario~2 
(see Figures \ref{fig.Sc2.kGau-freqDh} and~\ref{fig.Sc2.EDivisive-freqDh}). 
The performance of ED improves when $\Ds$ is given to the algorithm, 
but KCP remains significantly better than ED in terms of detection power 
(see Figures \ref{fig.Sc2.kGau-freqDs} and~\ref{fig.Sc2.EDivisive-freqDs}, for instance).  
Overall, KCP with $k=\kGau$ 
clearly outperforms ED in Scenarios~1--2, 
which can be explained by at least two reasons: 
(i) ED uses a different similarity measures than ours; 
(ii) ED relies on a greedy strategy, in which $\tauh(D+1)$ is obtained from 
$\tauh(D)$ by adding one change-point, 
so that any mistake at the beginning of the process 
impacts the final segmentation.

\section{Real-data experiment} \label{subsec.real.experiments}

\subsection{Data description}

In this section, we illustrate the behavior of KCP on a publicly available dataset corresponding to wave heights hourly-measured between January 2005 and September 2012 at a location in Northern Atlantic \citep[][Section~4.2]{killick2012optimal}. This leads to a large sample of length $n=63\,651$.

This dataset exhibits a strong difference between the wave heights during winters (high level) and summers (low level) (Figure~\ref{fig.segmentation.Wave}). Plotting the first-difference based signal following \citet[][Figure~4]{killick2012optimal} highlights strong changes in the variance of the signal (Figure~\ref{fig.segmentation.Wave.Diff}).
Automatically detecting the change-points between such successive periods is of primary interest for analyzing the environmental conditions of offshore wind farms for instance.

\begin{figure}[h!]
	\centering
	
	\subfloat[Original input]{\label{fig.segmentation.Wave} 
		
		\includegraphics[width = \textwidth]{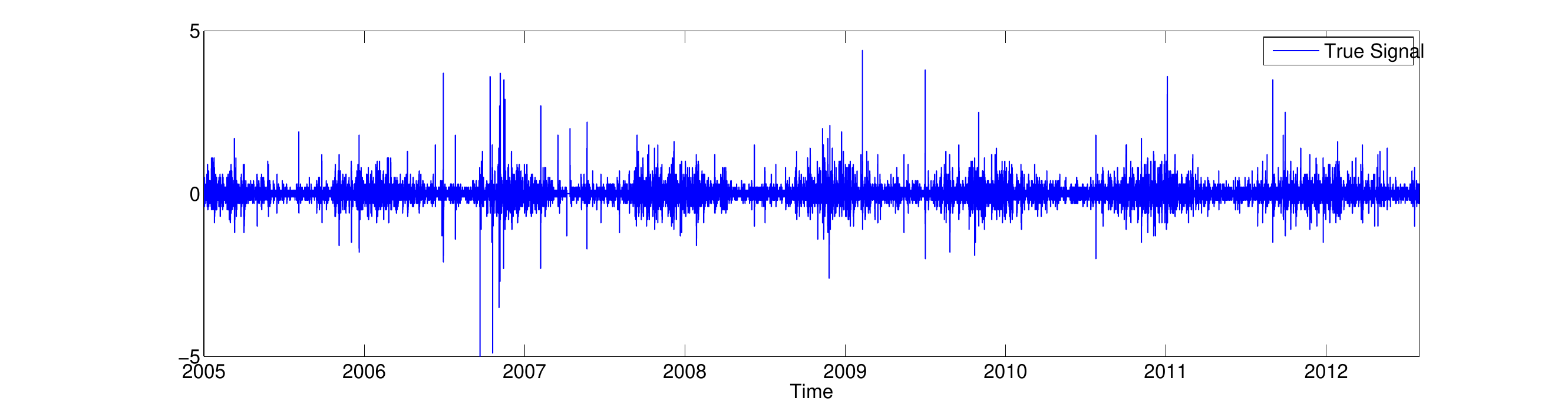}
	}
	\\
	\vspace*{-0.01\textwidth}
	\subfloat[First-difference based signal]{\label{fig.segmentation.Wave.Diff}
		\includegraphics[width =\textwidth]{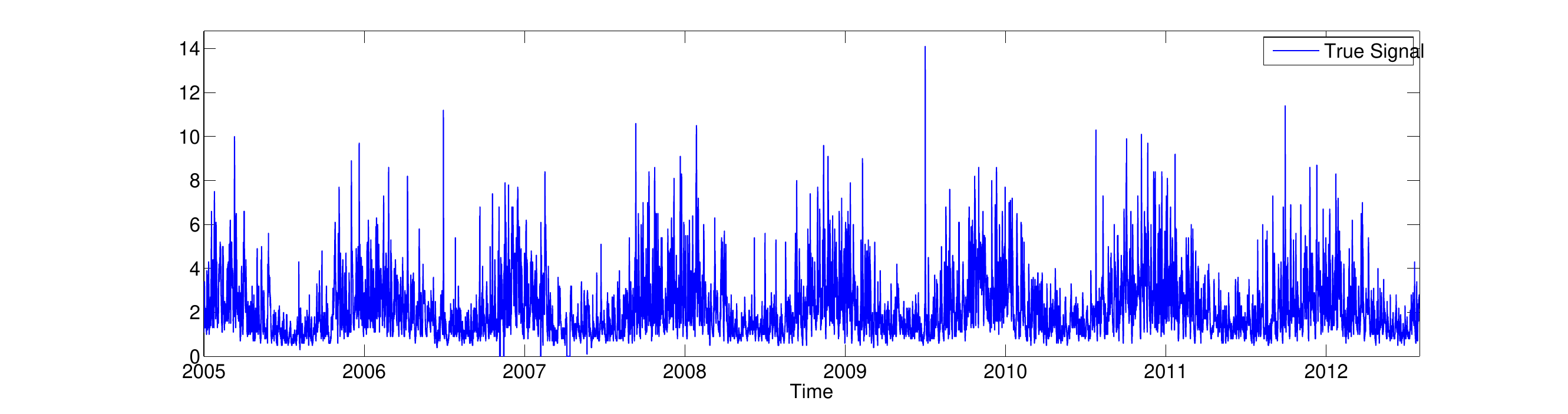}
	}
	\caption{Wave-heights time-series collected between January 2005 and September 2012.
	}
	\label{fig.segmentation.Wave.height} 
\end{figure}

\subsection{Procedures compared}

\paragraph{KCP}
We apply KCP (Algorithm~\ref{algo}) on the original data (Figure~\ref{fig.segmentation.Wave}) 
with the Gaussian kernel and bandwidth parameter equal to the empirical standard deviation of the data $\sigma_G=1.3526$ ---we here take the same data-driven bandwidth choice as \cite{Cel_Mor_Mar_Rig:2016:journal}. 
The maximum number of segments is set to $D_{\max}=50$, which seems to be large enough since about 14 changes only are expected along this period of almost 7 years.
The numerical constants $c_1,c_2\geq 0$ are estimated by using the slope heuristics as detailed in Section~\ref{subsubsec.synthetic.proc}.

\paragraph{ED}
The E-divisive procedure (ED) from \citet{Matteson_James:2014} (see Section~\ref{sec.synthetic.data.results}) is applied on the original data (Figure~\ref{fig.segmentation.Wave}) with its default parameter values: $\mathtt{sig.lvl}=0.05$, $\alpha=1$, and $R=199$ permutations.

\paragraph{PELT}
The so-called PELT procedure \citep{killick2012optimal} is considered by means of the function \texttt{cpt.var} implemented in the \texttt{R} package \texttt{changepoint} \citep{killick2014changepoint}.
It is applied to the first-difference based signal (Figure~\ref{fig.segmentation.Wave.Diff}) as done by \citet{killick2012optimal}, 
because this procedure is built for detecting changes in the variance of a zero-mean Gaussian signal.

\subsection{Results}

Figure~\ref{fig.segmented.Wave.height} displays the estimated change-points (red vertical dashed lines) output by KCP (Figure~\ref{fig.segmented.Wave.Diff.KCP}) and PELT (Figure~\ref{fig.segmented.Wave.Diff.PELT}).
The segmentation output by ED only contains one segment (no change-point), which is consistent with the trend of ED towards underestimating the number of changes.

KCP outputs 16 homogeneous segments which do not coincide with changes of the mean as one could have feared. 
PELT outputs 17 segments which are mainly similar to the ones of KCP. 
Both results are realistic, as explained by \citet{killick2012optimal}.

\begin{figure}[h!]
	\centering
	
	\subfloat[KCP]{\label{fig.segmented.Wave.Diff.KCP}
		\includegraphics[width = \textwidth]{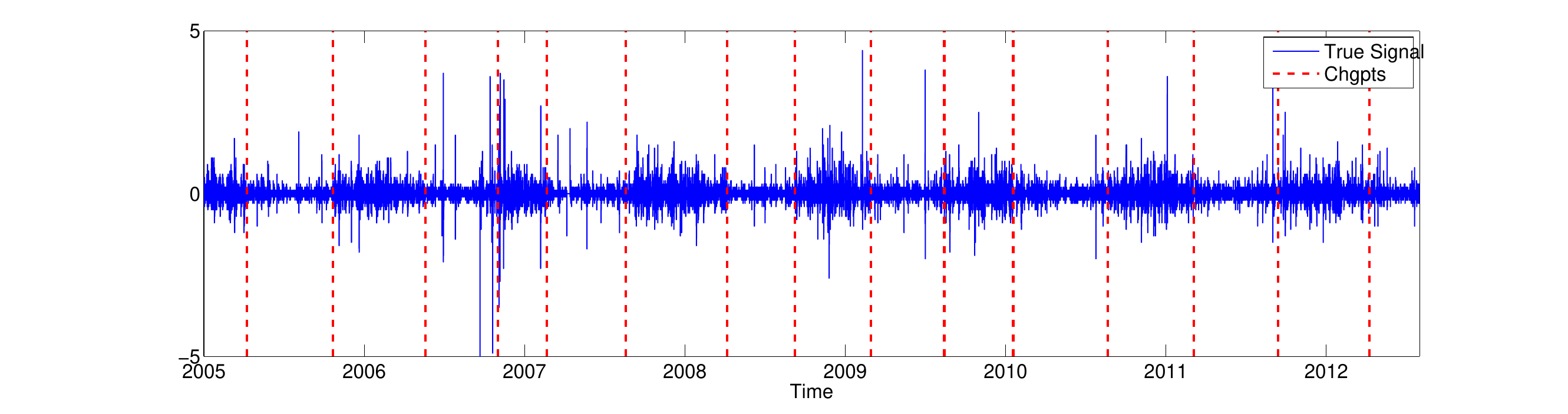}
	}
	\\
	\subfloat[PELT]{\label{fig.segmented.Wave.Diff.PELT}
		\includegraphics[width =\textwidth]{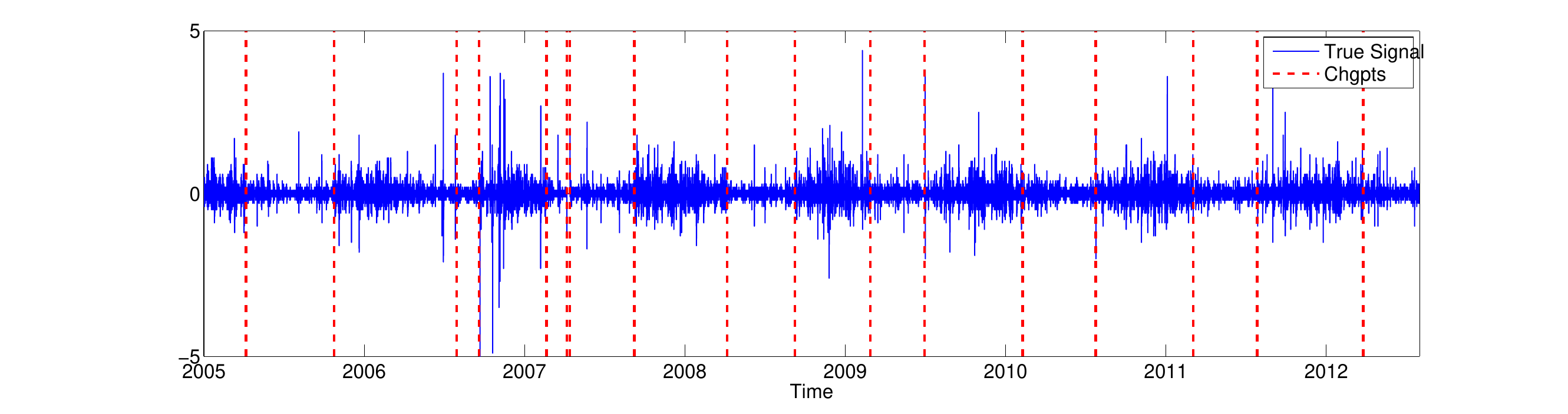}
	}
	\caption{Segmentations of the wave-heights time-series output by KCP and PELT, respectively. 
	}
	\label{fig.segmented.Wave.height} 
\end{figure}

Nevertheless there are a few differences around the year 2007 where PELT detects very narrow segments, which are likely related to outliers.
It still arises that the fourth change-point location estimated by KCP is somewhat questionable, compared to the one output by PELT, 
since it coincides with a strong change in the variance which seems, by eye, to have started a bit sooner.

A striking feature of KCP remains that no \emph{a priori} specification has been made about the type of changes we are looking for. 
This strongly contrasts with the PELT procedure,  
which makes a Gaussian assumption and relies on 
the prior knowledge that changes occur in the variance of the signal only (in this example). 

The overall conclusion is that KCP here provides reliable results without requiring 
any side information about the data distribution and the nature of its changes. 
This turns out to be a strong asset when analyzing real data, 
for which any distributional assumption is misleading when it happens to be violated.

\section{Conclusion}\label{sec.conclusion}
%
%
This paper describes a kernel change-point algorithm 
(KCP, that is, Algorithm~\ref{algo}), 
based upon a penalization procedure generalizing the one of 
\citet{Com_Roz:2004} and \citet{Leb:2005} to RKHS-valued data. 
Such an extension significantly broadens the range of 
possible applications of the algorithm, 
since it can deal with complex or structured data, 
and it can detect changes in the full distribution 
of the data ---not only the mean or the variance. 
The new theoretical tools developed in the paper ---mostly, a 
concentration inequality for some function of 
RKHS-valued random variables (Proposition~\ref{pro.conc.quad})--- could be useful in other 
settings, such as clustering in reproducing kernel Hilbert 
spaces or functional data analysis. 
Let us now end the paper with three questions about KCP: 
one that has been solved while this paper was in revision, 
and two that are still open.

\subsection{Identification of the change-point locations}
\label{sec.conclu.consistency}
A natural question for a change-point algorithm 
is its consistency for estimating the 
true change-point locations $\taus$. 
More precisely, 
let us assume that some $\taus \in \cT_n$ exists such that 
\[ 
P_{X_{\taus_{\ell-1}+1}} 
= \cdots = P_{X_{\taus_{\ell}}} 
\qquad \text{for } 1 \leq \ell \leq  D_{\taus}
\, , 
\qquad 
P_{X_{\taus_{\ell}}} \neq P_{X_{\taus_{\ell}+1}}
\qquad \text{for } 1 \leq \ell \leq  D_{\taus} -1 
\]
and $D_{\taus}$ is fixed as $n$ tends to infinity (even if $\taus$ necessarily depends on $n$). 
The goal is to prove that 
$d(\tauh,\taus)$ tends to zero almost surely 
as $n$ tends to infinity, 
where $d$ is some distance on $\cT_n$, 
for instance $n^{-1} d_F$ or $n^{-1} d_H$ as defined in 
Section~\ref{sec.synthetic.data.results}. 
Many papers prove such a consistency result 
for other change-point algorithms in various settings 
\citep[for instance,][]{Yao:1988,Lavielle:Moulines:2000,Frick_Munk_Sieling:2014,Matteson_James:2014}. 
Answering this question for KCP is 
beyond the scope of the paper. 
It has been proved by \citet{Gar_Arl:2015}, 
after the first version of this work appeared as a preprint, 
that KCP is indeed consistent under mild assumptions.

\subsection{Choosing the kernel $k$}
\label{sec.conclusion.choix-k}
A major practical and theoretical question about KCP 
is the choice of the kernel $k$. 
Fully answering this question is beyond the scope of 
the paper, but we can already provide a few guidelines, 
based upon the theoretical and experimental results 
that we already have, and review some previous works 
tackling a related question. 

\medbreak

%
First, simulation experiments in Section~\ref{sec.synthetic.data} 
show that the performance can strongly vary with $k$. 
They suggest that using a characteristic kernel ---such as 
the Gaussian kernel $\kGau$--- yields a more versatile 
procedure when the goal is to detect changes in the 
full distribution of the data. 
Nevertheless, for a given change-point problem, 
all characteristic kernels certainly are not equivalent. 
For instance, unshown experimental results suggest 
that $\kGau_{h}$ with a clearly bad choice of the bandwidth $h$ ---say, smaller than $10^{-4}$ or larger than $10^4$ in settings similar to Scenario~1--- leads 
to a poor performance of KCP, 
despite the fact that $\kGau_{h}$ is characteristic 
for any $h>0$. 

Furthermore, for a given setting, a non characteristic 
kernel can be a good choice: 
when the goal is to detect changes in the mean of $X_i \in \R^d$, 
$\klin$ is known to work very well \citep{Leb:2005}. 
\citet{Cab_etal:2018,Cab_etal:2018:AR} also show that KCP can be used for focusing on changes 
in the correlation (resp. autocorrelation) structure of multivariate time series.

\medbreak

%
Second, our theoretical interpretation of KCP 
in Section~\ref{sec.oracle.intuition} 
suggests how the performance of KCP 
depends on $k$, hence on which basis $k$ should be chosen. 
Indeed, KCP focuses on changes in the 
mean $\bayes_1, \ldots, \bayes_n$ of the time series 
$Y_1, \ldots, Y_n \in \cH$. 
A change between $P_{X_i}$ and $P_{X_{i+1}}$ should be 
detected more easily when 
\begin{align*}
  \normHs{\bayes_{i+1} - \bayes_i}^2 
 = \E\crochb{k(X_{i+1},X_{i+1})} - 2 \E\crochb{k(X_{i+1},X_{i})}  + \E\crochb{k(X_{i},X_{i})}
\end{align*}
is larger, compared to the ``noise level'' 
$\max\{v_i,v_{i+1}\}$. 
When $P_{X_i} \neq P_{X_{i+1}}$, we know that 
$\normHs{\bayes_{i+1} - \bayes_i}$ is positive 
for any characteristic 
kernel $k$, while it might be equal to zero when $k$ is not characteristic. 
But the fact that $k$ is characteristic or not 
is not sufficient to guess whether $k$ will work well 
or not, according to the above heuristic. 

\medbreak

%
The problem of choosing a kernel has been considered 
for many different tasks in the machine learning literature. 
Let us only mention here some references that are tackling 
this question in a framework close to change-point detection: 
choosing the best kernel for a two-sample or an homogeneity test. 

%
For choosing the bandwidth $h$ of a Gaussian kernel, 
a classical heuristic ---called the median heuristic--- is to take $h$ equal to some 
median of $\parens{ \norms{X_i-X_j} }_{i < j}$, 
see  
\citet[Section~8, and references therein]{Gre_etal:2012} 
and \citet{Gar_Jit_Kan:2018}. 
%

%
A procedure for choosing the best convex combination 
of a finite number of kernels has been proposed 
by \citet{Gre_etal:2012:nips}, 
with the goal of building a powerful two-sample test. 
Another idea for combining several kernels, 
for instance the family $\sets{ \kGau_h : h>0}$, 
has been studied by \citet{Sri_etal:2009} 
for homogeneity and independence tests. 
Roughly, the idea is to replace the MMD test statistics ---which depends on a kernel $k$--- by its supremum 
over the considered family of kernels. 
Nevertheless, the extension of these two ideas 
to change-point detection with KCP 
does not seem straightforward. 

\medbreak


Let us now discuss the choice of the bandwidth of a Gaussian kernel for KCP. 
If $n$ is large, the median heuristic can require a large computation time. 
When $X_i \in \R$, the empirical standard deviation of $\{ X_1, \ldots , X_n \}$ is a good proxy to it, 
easy to compute on large datasets. 
It has been used successfully with KCP by \citet{Cel_Mor_Mar_Rig:2016:journal}, 
which is the reason why we use it in Section~\ref{subsec.real.experiments}. 

Nevertheless, using the median heuristic (or a proxy) with KCP may be questionable in general, 
since two-sample test and multiple change-point detection are different tasks. 
For instance, when the mean of the $X_i \in \R$ has large jumps, 
the median-heuristic bandwidth can be much larger than the 
standard deviation of the $X_i$, so that it may not work as well. 
In such cases, another option to consider would be some median of 
$\parens{ \norms{X_{i+1}-X_i} }_{1 \leq i \leq n-1}$, 
which could be studied in future works on KCP.

\subsection{Heteroscedasticity of data in $\H$}
\label{sec.conclusion.heterosc}
A possible drawback of KCP 
is that it does not take into account the fact 
that the variance $v_i$ of $Y_i = \Phi(X_i)$ 
can change with $i$: 
in general, the $Y_i$ are heteroscedastic. 
In the case of real-valued data and the linear kernel $\klin$, 
\citet{Arlot:Celisse:2011} have shown that heteroscedastic 
data can make KCP fail, 
and that this failure cannot be fixed by changing the penalty used at Step~2: 
all the segmentations $\tauh(D)$ produced at Step~1 
can be wrong. 

We conjecture that, for the Gaussian kernel $\kGau_h$ 
at least, when the bandwidth $h$ is well chosen, 
the variances of the $Y_i$ stay within a reasonably small 
range of values for most non-degenerate distributions. 
Indeed, according to Eq.~\eqref{eq.v_i}, 
\[ 
v_i = 1 - \E\crochj{ \exp\parenj{ \frac{-\normHs{X_i - X'_i}^2}{2 h^2} } } \in [0,1]
\]
where $X'_i$ is an independent copy of $X_i$. 
If $X_i$ is not deterministic and if $h$ 
is smaller than the typical order of magnitude 
of $\normHs{X_i - X'_i}$, then, 
$v_i$ cannot be much smaller than its maximal value~$1$. 
The median heuristic and our simulation experiments 
suggest that ``good'' values of $h$ for change-point 
detection are small enough, 
but this remains to be proved. 

When heteroscedasticity is a problem for KCP, 
which probably occurs for some kernels beyond $\klin$, 
we can think of combining KCP 
with the ideas of \citet{Arlot:Celisse:2011}, 
that is, replacing the empirical risk and the penalized criterion 
in Steps~1 and~2 of KCP by 
cross-validation estimators of the risk 
$\Risks{\ERM_{\tau}}$.


\section*{Acknowledgments} 
The authors thank Damien Garreau for some discussions that lead to an improvement of the theoretical results 
---namely, Proposition~\ref{pro.conc.quad} and Theorem~\ref{thm.oracle.large}, 
which were stated with the additional assumption that 
$ \min_i v_i \geq c M^2 >0$ in a previous version of this paper 
\citep{Arl_Cel_Har:2012:v1}. 
\\
This work was mostly done while Sylvain Arlot was financed by CNRS 
and member of the Sierra team in the Departement d'Informatique de l'Ecole normale superieure (CNRS/ENS/INRIA UMR 8548), 
45 rue d'Ulm, F-75230 Paris Cedex 05, France, and Zaid Harchaoui was a member of the LEAR team of Inria.
Sylvain Arlot and Alain Celisse were also supported by Institut des Hautes \'Etudes Scientifiques (IHES, Le Bois-Marie, 35, route de Chartres,
91440 Bures-Sur-Yvette, France) at the end of the writing of this paper.
Sylvain Arlot is also member of the Select project-team of Inria Saclay. 
\\
The authors acknowledge the support of the French Agence Nationale de la Recherche (ANR) under reference ANR-09-JCJC-0027-01 ({\sc Detect} project) 
and ANR-14-CE23-0003-01 ({\sc Macaron} project), 
the GARGANTUA project funded by the Mastodons program of CNRS, 
the LabEx Persyval-Lab (ANR-11-LABX-0025), 
the BeFast project funded by the PEPS Fascido program of CNRS, 
and the Moore-Sloan Data Science Environment at
NYU.


\appendix

\renewcommand{\theequation}{\thesection.\arabic{equation}}
\renewcommand{\thetheorem}{\thesection.\arabic{theorem}}
\renewcommand{\thelemma}{\thesection.\arabic{lemma}}
\renewcommand{\thedefinition}{\thesection.\arabic{definition}}
\renewcommand{\theremark}{\thesection.\arabic{remark}}
\renewcommand{\thetable}{\thesection.\arabic{table}}
\renewcommand{\thefigure}{\thesection.\arabic{figure}}

\section{Additional proofs}
\label{app.additional.proofs}

\subsection{Proofs of Section~\ref{subsec.abstract.formulation.algorithm}}
\label{app.proofs.orthog.proj}

\subsubsection{Proof of Eq.~\eqref{eq.regressogram.Hilbert}}
Let $f \in F_\tau$ and $g \in \H^n$. 
For any $\ell \in \inter{1,D_{\tau}}$, we define 
$I^{\tau}_{\ell} \egaldef \inter{\tau_{\ell-1} + 1 , \tau_{\ell}}$ the $\ell$-th interval of $\tau$, 
$f_{I^{\tau}_{\ell}}$ the common value of $(f_i)_{i \in I^{\tau}_{\ell}}$ and 
\begin{align} \label{eq.proj.interval.value}
 \overline{g}_{I^{\tau}_{\ell}} \egaldef \frac{1}{\card(I^{\tau}_{\ell})} \sum_{i \in I^{\tau}_{\ell}} g_i 
= \frac{1}{\tau_{\ell} - \tau_{\ell-1}} \sum_{i \in I^{\tau}_{\ell}} g_i
\enspace . 
\end{align}
Then, 
\begin{align*}
\norm{f-g}^2 
&= \sum_{\ell=1}^{D_\tau} \sum_{i \in I^{\tau}_{\ell}} 
\croch{ \normH{f_{I^{\tau}_{\ell}} - \overline{g}_{I^{\tau}_{\ell}}}^2 + \normH{g_i - \overline{g}_{I^{\tau}_{\ell}}}^2 + 
2 \prodH{f_{I^{\tau}_{\ell}} - \overline{g}_{I^{\tau}_{\ell}}}{\overline{g}_{I^{\tau}_{\ell}} - g_i}}
\\
&= \sum_{\ell=1}^{D_\tau} \croch{ (\tau_{\ell}-\tau_{\ell-1}) \normH{f_{I^{\tau}_{\ell}} - \overline{g}_{I^{\tau}_{\ell}}}^2 }
+ \sum_{\ell=1}^{D_\tau} \sum_{i \in I^{\tau}_{\ell}} \normH{g_i - \overline{g}_{I^{\tau}_{\ell}}}^2 
\enspace.
\end{align*}
since $\sum_{i \in I^{\tau}_{\ell}} \parens{\overline{g}_{I^{\tau}_{\ell}} - g_i }=0$. 
So, $\norm{f-g}^2$ is minimal over $f \in F_{\tau}$ if and only if $f_{I^{\tau}_{\ell}} = \overline{g}_{I^{\tau}_{\ell}}$ for every $\ell \in \inter{1,D_\tau}$. 
\hfill \blackbox

\subsubsection{Proof of Eq.~\eqref{empirical.risk.loss}}

We use the notations introduced in the proof of Eq.~\eqref{eq.regressogram.Hilbert}. 
Then, 
\begin{align*}
\norm{ Y - \ERM_{\tau} }^2 
&= \sum_{\ell=1}^{D_{\tau}} \sum_{i \in I^{\tau}_{\ell}} \normHs{Y_i - \overline{Y}_{I^{\tau}_{\ell}}}^2
= \sum_{\ell=1}^{D_{\tau}} \sum_{i \in I^{\tau}_{\ell}} \parenj{ \normHs{Y_i}^2  - \normHs{\overline{Y}_{I^{\tau}_{\ell}}}^2 }
\end{align*} 
where we used Eq.~\eqref{eq.regressogram.Hilbert} for the first equality, 
and that 
\[
\sum_{i \in I^{\tau}_{\ell}} \prodH{Y_i}{\overline{Y}_{I^{\tau}_{\ell}}} 
= \card(I^{\tau}_{\ell}) \normH{\overline{Y}_{I^{\tau}_{\ell}}}^2 
\] 
for the second equality. 
Therefore, 
\begin{align*}
\norms{ Y - \ERM_{\tau} }^2 
&= \sum_{i=1}^n \normHs{Y_i}^2 
- \sum_{\ell=1}^{D_{\tau}}
\frac{1}{\tau_{\ell} - \tau_{\ell-1}} 
\normHBb{ \sum_{i \in I^{\tau}_{\ell}} Y_i}^2
\\
&= \sum_{i=1}^n \normHs{Y_i}^2 
- \sum_{\ell=1}^{D_{\tau}}
\frac{1}{\tau_{\ell} - \tau_{\ell-1}} 
\sum_{i,j \in I^{\tau}_{\ell}} \prodH{Y_i}{Y_j}
\\
&= \sum_{i=1}^n k(X_i,X_i)
- \sum_{\ell=1}^{D_{\tau}}
\frac{1}{\tau_{\ell} - \tau_{\ell-1}} 
\sum_{i,j \in I^{\tau}_{\ell}} k(X_i,X_j)
\enspace , 
\end{align*}
which proves Eq.~\eqref{empirical.risk.loss}. 
\hfill \blackbox

\subsection{Concentration of the linear term: proof of Proposition~\ref{prop.concentration.lineaire}}
\label{sec.proof.conc-linear}

Let us define $\bayes_\tau = \Pi_\tau \bayes $ and 
\begin{align*}
	S_\tau = \prodscal{\bayes-\bayes_\tau}{\varepsilon}  = \sum_{i=1}^n  Z_i 
	\quad \text{with} \quad Z_i = \prodscalb{(\bayes-\bayes_\tau)_i}{\varepsilon_i}_{\H} \enspace ,
\end{align*}
The $Z_i$s are independent and centered, so Eq.~\eqref{eq.lem.upperbounds.Bernstein.Linear}--\eqref{eq.lem.upperbounds.Bernstein.Linear.2} in Lemma~\ref{lem.upperbounds.Bernstein.Linear} below (which requires assumption \eqref{hyp.donnees-bornees}) show that the conditions of Bernstein's inequality are satisfied (see Proposition~\ref{thm.Bernstein.ineq}). 
Therefore for every $x \geq 0$, with probability at least $1-2\e^{-x}$, 
\begin{align*} 
\absj{\sum_{i=1}^n Z_i} 
&\leq \sqrt{2 \vmax \norm{\bayes - \bayes_\tau}^2 x} + \frac{4 M^2 x}{3}
\\ 
&\leq \theta \norm{\bayes - \bayes_\tau}^2 + \paren{ \frac{\vmax}{2 \theta} + \frac{4 M^2}{3} } x 
\end{align*}
for every $\theta >0$, using $2ab \leq \theta a^2 + \theta^{-1}b^2$. 
\hfill\blackbox

A key argument in the proof is the following lemma. 
\begin{lemma}\label{lem.upperbounds.Bernstein.Linear}
For every $\mM_n$, if \eqref{hyp.donnees-bornees} holds true, 
the following holds true with probability one\textup{:} 
\begin{gather} 
\label{eq.maj.normH-(mum-mu)i.donnees-bornees}
\forall i \in \set{1, \ldots, n} \, , \quad 
 \normH{\bayes_i}  \leq M 
 \, , 
\qquad 
 \normH{\varepsilon_i} \leq  2 M
\\
\label{eq.lem.upperbounds.Bernstein.Linear}
\text{and} \quad 
\normH{(\bayes-\bayes_\tau)_i}  \leq 2 M
\quad \text{so that} \quad  \absj{Z_i} \leq 4 M^2 \enspace .
\\
\label{eq.lem.upperbounds.Bernstein.Linear.2}
\quad \text{In addition,} \quad 
	\sum_{i=1}^n \var\paren{Z_i} \leq \vmax \norm{\bayes-\bayes_\tau}^2 \enspace .
\end{gather}
\end{lemma}
\begin{proof}[of Lemma~\ref{lem.upperbounds.Bernstein.Linear}]
First, remark that for every $i$, 
\[ v_i = \E\crochb{ \norms{\varepsilon_i}^2 } = \E\crochb{k(X_i,X_i)} - \normHs{\bayes_i}^2 \geq 0 \enspace , \]
so that with \eqref{hyp.donnees-bornees},
\[ \normHs{\bayes_i}^2 \leq \E\crochb{k(X_i,X_i)} \leq M^2 \enspace , \]
which proves the first bound in Eq.~\eqref{eq.maj.normH-(mum-mu)i.donnees-bornees}. 
As a consequence, by the triangular inequality, 
\[ \normH{\varepsilon_i} \leq \normH{Y_i} + \normH{\bayes_i} \leq 2 M 
\enspace , \]
that is, the second inequality  in Eq.~\eqref{eq.maj.normH-(mum-mu)i.donnees-bornees} holds true. 

Let us now define for every  $i \in \set{1, \ldots, n}$, the integer $K(i) \in \sets{1, \ldots,  D_\tau}$ 
such that $I_{K(i)}^\tau= \inter{ \tau_{K(i)-1}+1, \tau_{K(i)} }$ is the unique interval of the segmentation $\tau$ such that $ i\in I^\tau_{K(i)}$. 
Then, 
\[ (\bayes - \bayes_\tau)_i = \frac{1}{ \tau_{K(i)}-\tau_{K(i)-1} } \sum_{j \in I^\tau_{K(i)}} (\bayes_i - \bayes_j) , \]
so that the triangular inequality and Eq.~\eqref{eq.maj.normH-(mum-mu)i.donnees-bornees} imply 
\begin{align*} 
\normH{(\bayes-\bayes_\tau)_i} &\leq \sup_{j \in I^\tau_{K(i)}} \normH{\bayes_i - \bayes_j}
\leq \sup_{1 \leq j,k \leq n} \normH{\bayes_k - \bayes_j}
\leq 2 \sup_{1 \leq j \leq n} \normH{\bayes_j} 
\leq 2 M
\enspace ,
\end{align*}
that is, the first part of Eq.~\eqref{eq.lem.upperbounds.Bernstein.Linear} holds true. 
The second part of Eq.~\eqref{eq.lem.upperbounds.Bernstein.Linear} directly follows from Cauchy-Schwarz's inequality. 
For proving Eq.~\eqref{eq.lem.upperbounds.Bernstein.Linear.2}, we remark that 
\begin{align*} 
\E\croch{Z_i^2} 
&= \E \crochB{  \prodHb{(\bayes-\bayes_\tau)_i}{\varepsilon_i}^2 } 
\\ 
&\leq \normH{(\bayes-\bayes_\tau)_i}^2 \E\croch{\normH{\varepsilon_i}^2} 
\qquad \text{by Cauchy-Schwarz's inequality}
\\
&
= \normH{(\bayes-\bayes_\tau)_i}^2 v_i \leq \normH{(\bayes-\bayes_\tau)_i}^2 \vmax  \enspace ,
\\
\text{so that} \quad 
	\sum_{i=1}^n \var\paren{Z_i} &\leq \vmax \norm{\bayes-\bayes_\tau}^2 \enspace.
\end{align*}
\end{proof}



\bibliography{kernelchpt}

\clearpage
\section{Supplementary material}
\label{sec.supmat}

\subsection{Classical concentration inequalities}
This section collects a few results that are used throughout the paper.

\subsubsection{Bernstein's inequality}

\begin{proposition}[Bernstein's inequality, as stated by \citealt{Mas:2003:St-Flour}, Proposition~2.9]\label{thm.Bernstein.ineq}
Let $X_1,\ldots,X_n$ be independent real-valued random variables.
Assume that some positive constants $v$ and $c$ exist such that, for every $k\geq 2$
\begin{align}
\label{constraint.Bernstein.moment}
  \sum_{i=1}^n \E\crochj{ \absj{X_i}^k} 
  \leq \frac{k!}{2} v c^{k-2}\enspace.
\end{align}
Then, for every $x>0$,
\begin{align*} 
  \P\parenj{ \sum_{i=1}^n \parenb{ X_i - \E\crochs{ X_i } } > \sqrt{ 2 v x  }  + c x } \leq \e^{-x} \enspace. 
\end{align*}
In particular, if for every $i \in \sets{1, \ldots, n}$, $\abss{X_i} \leq 3c$ almost surely, Eq.~\eqref{constraint.Bernstein.moment} holds true with $v=\sum_{i=1}^n \var\parens{X_i}$.
\end{proposition}

\subsubsection{Pinelis-Sakhanenko's inequality}

\begin{proposition}[\citet{Pin:Sak:1986}, Corollary~1] \label{prop.Pinelis.Sakhanenko}
  Let $X_1,\ldots,X_n$ be independent random variables with values in some Hilbert space $\mathcal{H}$. 
  Assume the $X_i$ are centered and that constants $\sigma^2,c>0$ exist such that for every $p\geq 2$, 
\begin{align*}
  \sum_{i=1}^n \E\croch{\normH{X_i}^p} \leq \frac{p!}{2} \sigma^2 c^{p-2} \enspace,
\end{align*}
Then, for every $x>0$,
\begin{align*}
  \P\croch{ \normH{ \sum_{i=1}^n X_i} > x } \leq 2 \exp\croch{- \frac{x^2}{2\paren{\sigma^2 + cx}}} \enspace.
\end{align*}
\end{proposition}

\subsubsection{Talagrand's inequality}

The following proposition is a refined version of 
Talagrand's concentration inequality \citep{Tal:1996a}, 
as it is stated by \citet[Corollary~12.12]{Bou_Lug_Mas:2011:livre}. 

\begin{proposition}[Corollary~12.12 of \cite{Bou_Lug_Mas:2011:livre}]
\label{pro.talagrand}
Let $X_1, \ldots, X_n$ be independent vector-valued random variables and let 
\[ 
Z = \sup_{f \in \F} \sum_{i=1}^n X_{i,f}
\, . 
\]
Assume that for all $i \in  \sets{1, \ldots, n}$ and $f \in \F$, $\E[X_{i,f}]=0$ and $| X_{i,f} | \leq 1$. 
Define 
\[ \sigma^2 = \sup_{f \in \F} \sum_{i=1}^n \E \croch{ X_{i,f}^2 } 
\qquad \text{and} \qquad 
v = 2 \E[Z] + \sigma^2
\, . \]
Then, for all $x \geq 0$, 
\begin{align} 
\label{eq.talagrand.deviations-sup}
\P \parenj{ Z \geq \E[Z] + \sqrt{2 v x} + \frac{x}{3} } &\leq \e^{-x} \\
\label{eq.talagrand.deviations-inf}
\P \parenj{ Z \leq \E[Z] - \sqrt{2 v x} - \frac{x}{8} } &\leq \e^{-x} 
\, .
\end{align}
\end{proposition}

\subsection{Proof of Proposition~\ref{pro.concentration.quadratic.talagrand}}
\label{app.Talagrand-details}
The first step is to write $\norm{\Pi_\tau \varepsilon}$ of the form of $Z$ in Proposition~\ref{pro.talagrand} for some well-chosen $(X_{i,f})_{1 \leq i \leq n , \, f \in \mathcal{G}_\tau}$.
With for every $1\leq K\leq D_\tau$, $\overline{f}_K = 1/(\tau_K - \tau_{K-1}) \sum_{i=\tau_{K-1}+1}^{\tau_K} f_i$, it comes
\begin{align*}
\norms{\Pi_\tau \varepsilon} 
&= \sup_{f \in \H^n , \, \norm{f} \leq 1} \absb{ \prodscal{f}{\Pi_\tau \varepsilon} } 
\\
&= \sup_{f \in \H^n , \, \norm{f} \leq 1} \absb{ \prodscal{\Pi_\tau f}{ \varepsilon} } 
\\
&= \sup_{f \in \H^n , \, \sum_{K=1}^{D_\tau} (\tau_K - \tau_{K-1}) \norm{\overline{f}_{K}}^2 \leq 1} \absj{ \sum_{K=1}^{D_\tau} \sum_{i =\tau_{K-1}+1}^{\tau_K} \prodscal{\overline{f}_{K}}{ \varepsilon_i}_{\mathcal{H}} } 
\\
&= \sup_{f \in \mathcal{G}_\tau} \sum_{i=1}^n \overline{X}_{i,f}
\end{align*}
where $\mathcal{G}_\tau$ is some countable dense subset of 
\[ 
\setj{ f \in \H^n , \, \sum_{K=1}^{D_\tau} \paren{\tau_K-\tau_{K-1}} \norm{\overline{f}_{K}}_{\mathcal{H}}^2 \leq 1 } 
\]
(such a set $\mathcal{G}_\tau$ exists since $\H$ is separable), 
and for every $i \in \set{1, \ldots, n}$ and $f \in \mathcal{G}_\tau$, 
\[ \overline{X}_{i,f} = \prodscal{\overline{f}_{K(i)}}{ \varepsilon_i}_{\mathcal{H}} \]
where we recall that $K(i)$ is defined in the proof of Lemma~\ref{lem.upperbounds.Bernstein.Linear}. 

Let us now check that the assumptions of Proposition~\ref{pro.talagrand} are satisfied: $(\overline{X}_{1,f})_{f \in \mathcal{G}_\tau}, \ldots, (\overline{X}_{n,f})_{f \in \mathcal{G}_\tau}$ are independent since $\varepsilon_1, \ldots, \varepsilon_n$ are assumed independent. 
For every $i \in \set{1, \ldots, n}$ and $f \in \mathcal{G}_\tau$, 
\[ \E\croch{ \overline{X}_{i,f} } = \E\croch{ \prodscal{\overline{f}_{K(i)}}{ \varepsilon_i}_{\mathcal{H}} } = 0 \]
since $\overline{f}_{K(i)} \in \H$ is deterministic, and for every $f \in \mathcal{G}_{\tau}$, 
\[ 
\absj{ \overline{X}_{i,f} } 
= \absj{\prodscal{\overline{f}_{K(i)}}{ \varepsilon_i}_{\mathcal{H}}} 
\leq \norm{\overline{f}_{K(i)}}_{\mathcal{H}} \norm{\varepsilon_i}_{\mathcal{H}} 
\leq \frac{2M}{\sqrt{\tau_{K(i)} - \tau_{K(i)-1}}} \leq 2M \]
by Cauchy-Schwarz's inequality, assumption \eqref{hyp.donnees-bornees} and Lemma~\ref{lem.upperbounds.Bernstein.Linear}. 
So, we can apply Proposition~\ref{pro.talagrand} to 
\[ 
Z = \frac{1}{2M} \norm{\Pi_\tau \varepsilon} = \sup_{f \in \mathcal{G}_\tau} \sum_{i=1}^n X_{i,f} 
\]
where $X_{i,f} \egaldef (2M)^{-1} \overline{X}_{i,f}$. 

Before writing the resulting concentration inequality, let us first compute (and bound) the quantity denoted by $\sigma^2$ in the statement of Proposition~\ref{pro.talagrand}. 
For every $f \in \mathcal{G}_\tau$, 
\begin{align*}
4 M^2 \sum_{i=1}^n \E \croch{ X_{i,f}^2 } 
=  \sum_{i=1}^n \E \croch{ \prodscal{\overline{f}_{K(i)}}{ \varepsilon_i}_{\mathcal{H}} ^2 } 
&\leq  \sum_{i=1}^n \croch{ \norm{\overline{f}_{K(i)}}_{\mathcal{H}}^2 \E \croch{ \norm{\varepsilon_i}_{\mathcal{H}} ^2 } }
\\
&=  \sum_{K=1}^{D_\tau}  \croch{ \norm{\overline{f}_{K}}_{\mathcal{H}}^2 \sum_{i =\tau_{K-1}+1}^{\tau_K} v_i }
\\
&=  \sum_{K=1}^{D_\tau}  \croch{ (\tau_K - \tau_{K-1}) \norm{\overline{f}_{K}}_{\mathcal{H}}^2 v^\tau_K } 
\end{align*}
by Cauchy-Schwarz's inequality. 
So, by definition of $\mathcal{G}_{\tau}$ and $\sigma^2$, 
\[
\sigma^2 \leq \frac{1}{4 M^2} \max_{1\leq K\leq D_\tau}  v^\tau_{K} 
\, . 
\]

We can now write what Proposition~\ref{pro.talagrand} proves about the concentration of $\norms{\Pi_\tau \varepsilon}$: for every $x \geq 0$, with probability at least $1 - \e^{-x}$, 
\begin{align*} 
\norms{\Pi_\tau \varepsilon} - \E\crochb{\norms{\Pi_\tau \varepsilon}} 
\leq 2 M \sqrt{2 v x} + \frac{2 M x}{3}
\leq \sqrt{ 2 x \parenj{ 4 M  \E\crochb{\norms{\Pi_\tau \varepsilon}} +  \max_{1\leq K\leq D_\tau}  v^\tau_{K} }} + \frac{2 M x}{3} 
\enspace , 
\end{align*} 
and similarly, with probability at least $1 - \e^{-x}$,
\begin{align*} 
\norms{\Pi_\tau \varepsilon} - \E\crochb{\norms{\Pi_\tau \varepsilon}} 
\geq - \sqrt{ 2 x \paren{ 4 M  \E\crochb{\norms{\Pi_\tau \varepsilon}} + \max_{1\leq K\leq D_\tau}  v^\tau_{K} }} - \frac{M x}{4} 
\enspace . 
\end{align*} 
So, using a union bound, we have just proved Eq.~\eqref{eq.pro.concentration.quadratic.talagrand.1}. 
\hfill \blackbox

\subsection{Second method for choosing $c_1,c_2$ in KCP}
\label{sec.choix-c1-c2.variance}

We describe an alternative to the slope heuristics for choosing  $c_1,c_2$ in KCP. 

When prior information guarantee that the ``variance'' is almost constant 
and that no change occurs in some parts of the observed time series ---say, at the start and at the end---, 
we can estimate this ``variance'' within each of these parts and take $c_1=c_2$ equal to  
\begin{align}\label{eq.estimated.variance.Zaid}
  \widehat{c}_{var} \egaldef 2 \max\parenb{ \widehat{v}_s, \widehat{v}_e }, 
\end{align}
where
\begin{align*}
\widehat{v}_s 
& \egaldef 
\frac{1}{|I_s|-1} \sum_{i \in I_s} \crochj{ k(X_i,X_i) + \frac{1}{|I_s|^2} \sum_{j,\ell \in I_s} k(X_j,X_\ell) - \frac{2}{|I_s|} \sum_{j \in I_s} k(X_i,X_j) } 
\end{align*}
denotes the empirical variance of the start $(X_i)_{i \in I_s}$ of the time series, 
and $\widehat{v}_e$ is defined similarly from the end $(X_i)_{i \in I_e}$ of the time series. 
The fact that an estimate of the variance multiplied by 2 is a good choice for $c_1=c_2$ is justified by the numerical experiments made by \citet{Leb:2005} in the case of the linear kernel and one-dimensional data. 
This strategy was used successfully in the real-data 
experiments of an earlier version of the present paper \citep[][Section~6.2]{Arl_Cel_Har:2012:v1}.

\clearpage

\subsection{Additional details about the synthetic experiments}

\subsubsection*{Data generation process}

Table~\ref{tab.Sc1.moy-var} provides the values of the mean and variance of the seven distribution considered in Scenario~1. 
It shows that the pair (mean, variance) changes at every change-point in Scenario~1, but the mean sometimes stays constant.

\begin{table}
\centering
\begin{tabular}{l@{\hspace{0.025\textwidth}}l@{\hspace{0.025\textwidth}}l}
Distribution & Mean & Variance 
\\ \hline
$\mathcal{B}(10, 0.2)$       & $2$            & $1.6$ \\
$\mathcal{NB}\paren{3, 0.7}$ & $9/7 \approx 1.29$         & $90/49 \approx 1.84$ \\
$\mathcal{H}(10, 5, 2)$      & $1$             & $4/9 \approx 0.44$ \\
$\mathcal{N}(2.5, 0.25)$     & $2.5$          & $0.25$ \\
$\gamma\paren{0.5, 5}$       & $2.5$          & $12.5$ \\
$\mathcal{W}(5,2)$           & $\frac{5\sqrt{\pi}}{2} \approx 4.43$ & $25 ( 1- \frac{\pi}{4} ) \approx 5.37$ \\
$\mathcal{P}ar(1.5, 3)$      & $9/4=2.25$            & $27/16 \approx 1.69$ 
\end{tabular}
\caption{\label{tab.Sc1.moy-var} Scenario~1, mean and variance for the seven distributions considered.}
\end{table}
\subsubsection*{Further results on synthetic data}
\label{sec.appendice.simul}
This section gathers some additional results concerning the experiments of Section~\ref{sec.synthetic.data}.

\clearpage 


\begin{figure}
\centering
\includegraphics[width = .478\textwidth]{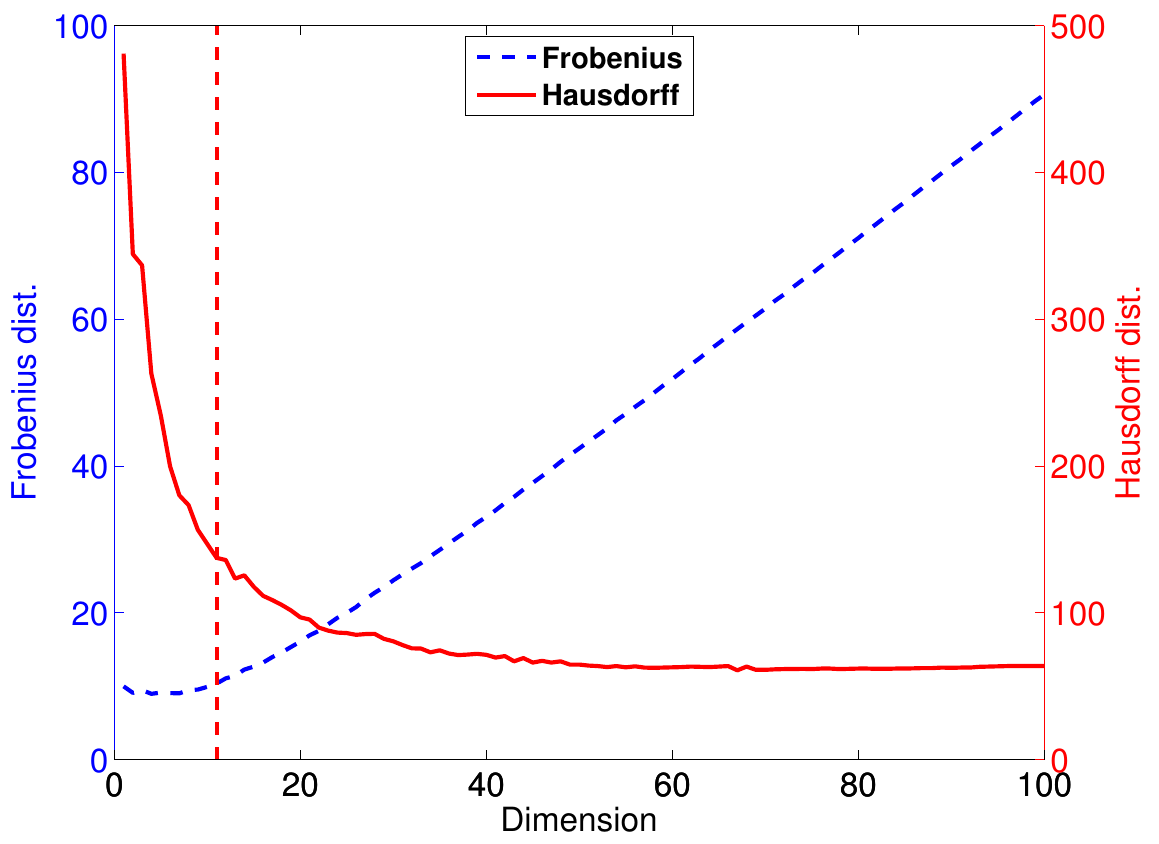}
\caption{\label{fig.Sc1.klin.dist} 
Scenario~1: $\X=\R$, variable (mean, variance). 
Performance of KCP with kernel $\klin$. 
Average distance ($d_F$ or $d_H$) between $\tauh(D)$ and $\taus$, as a function of $D$.
}
\end{figure}

\begin{figure}
  \centering
  \figtroishspace
  \subfloat[$k=\klin$]{\label{fig.Sc1.klin-Dh} 
\includegraphics[width = \figtroiswidth]{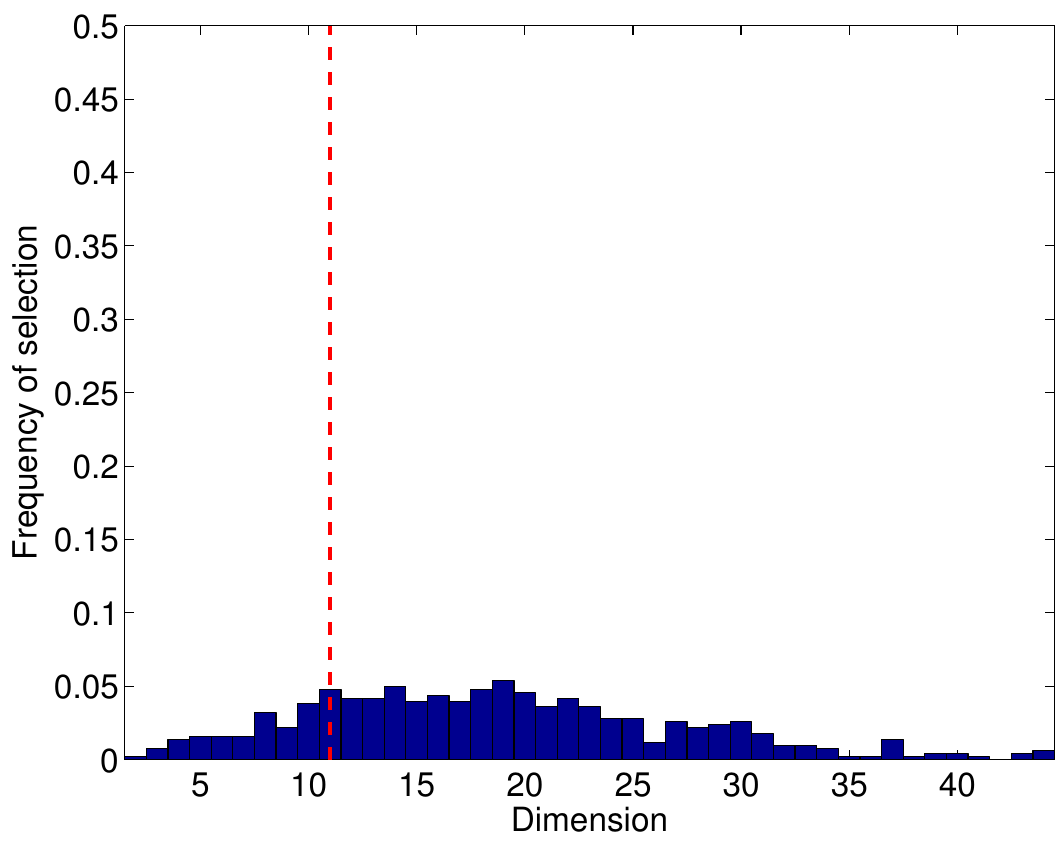}
}
 \hspace*{0.01\textwidth}
  \subfloat[$k=\kHer_{1}$]{\label{fig.Sc1.kHer-Dh}
\includegraphics[width = \figtroiswidth]{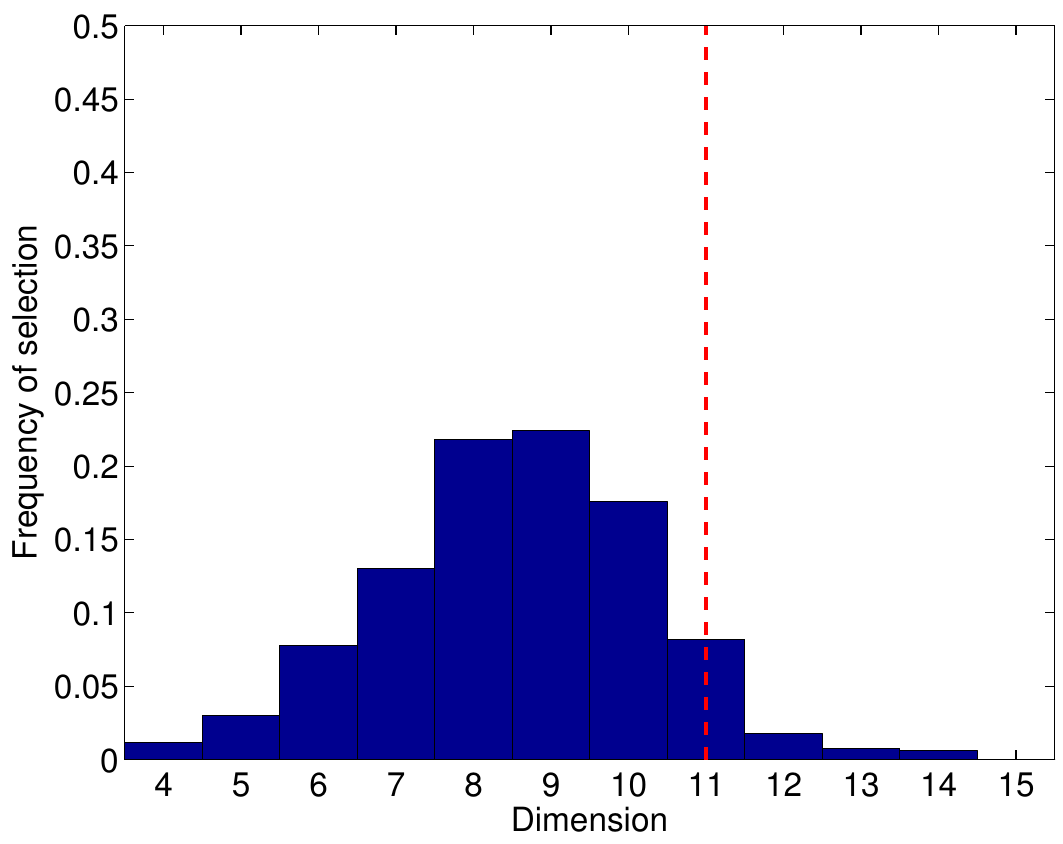}
}
 \hspace*{0.01\textwidth}
  \subfloat[$k=\kGau_{0.1}$]{\label{fig.Sc1.kGau-Dh}
\includegraphics[width = \figtroiswidth]{fig_Distrib_Dimension_type_2kerneltype_L2_bw_1_Nb500}
}
\caption{%
Scenario~1: $\X=\R$, variable (mean, variance). 
KCP with three different kernels $k$.  
Distribution of $\Dh$. 
(Figure~\ref{fig.Sc1.kGau-Dh} is a copy of Figure~\ref{fig.Sc1.Dh}, that we repeat here for making comparisons easier.) 
}
\label{fig.Sc1.Dh.3k}
\end{figure}

\clearpage

\begin{figure}
  \centering
\figtroishspace
  \subfloat[$k=\klin$]{\label{fig.Sc1.freq-Ds.klin} 
\includegraphics[width = \figtroiswidth]{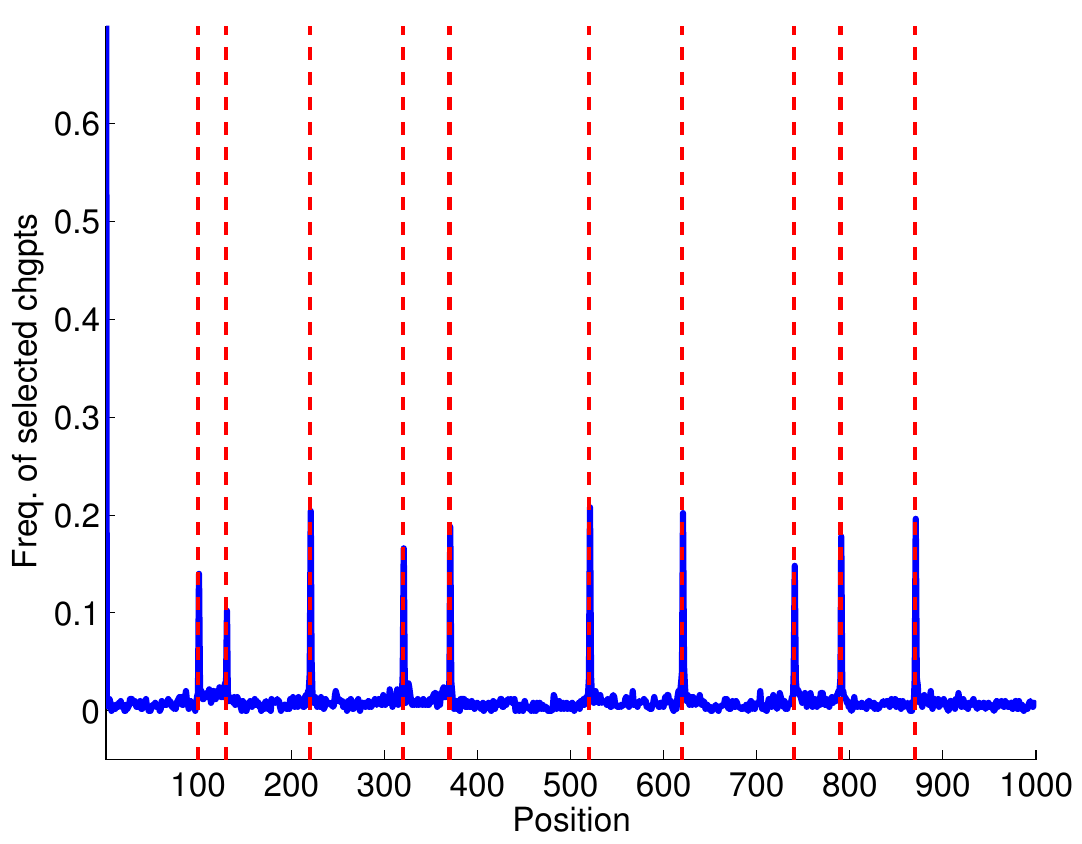}
}
 \hspace*{0.01\textwidth}
  \subfloat[$k=\kHer_{1}$]{\label{fig.Sc1.freq-Ds.kHer} 
\includegraphics[width = \figtroiswidth]{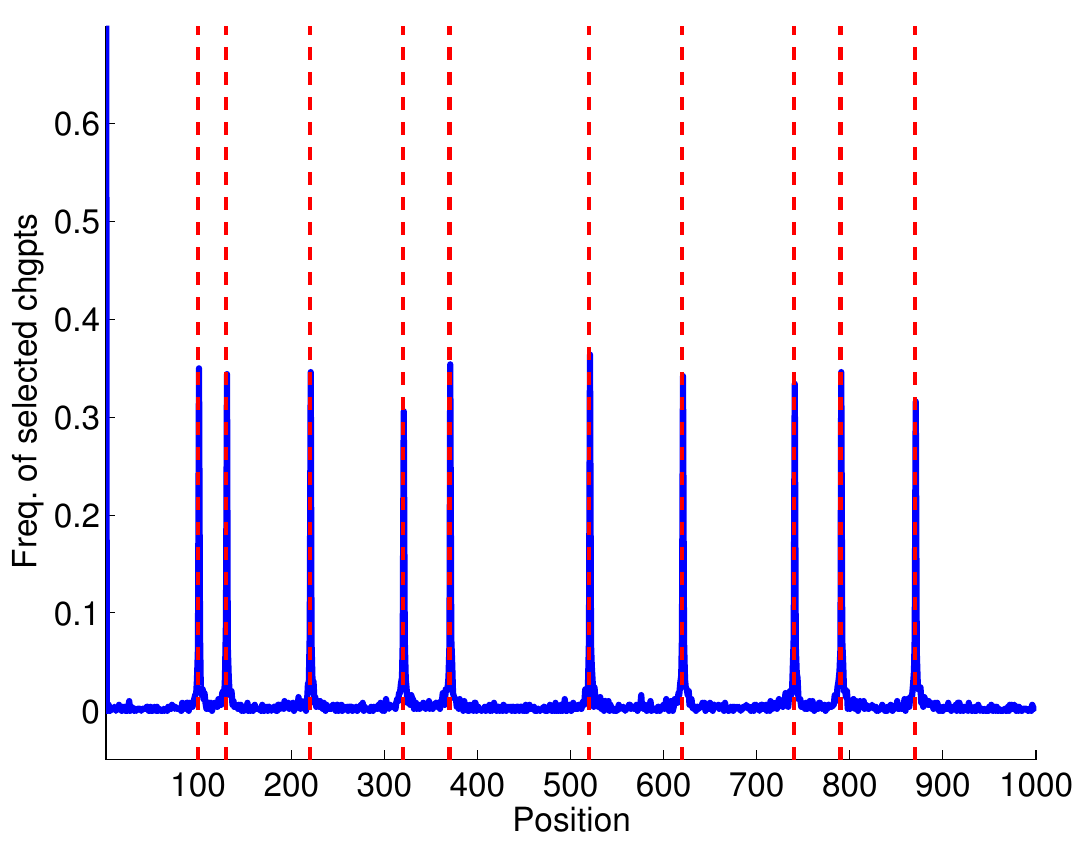}
}
 \hspace*{0.01\textwidth}
  \subfloat[$k=\kGau_{0.1}$]{\label{fig.Sc1.freq-Ds.kGau}
\includegraphics[width = \figtroiswidth]{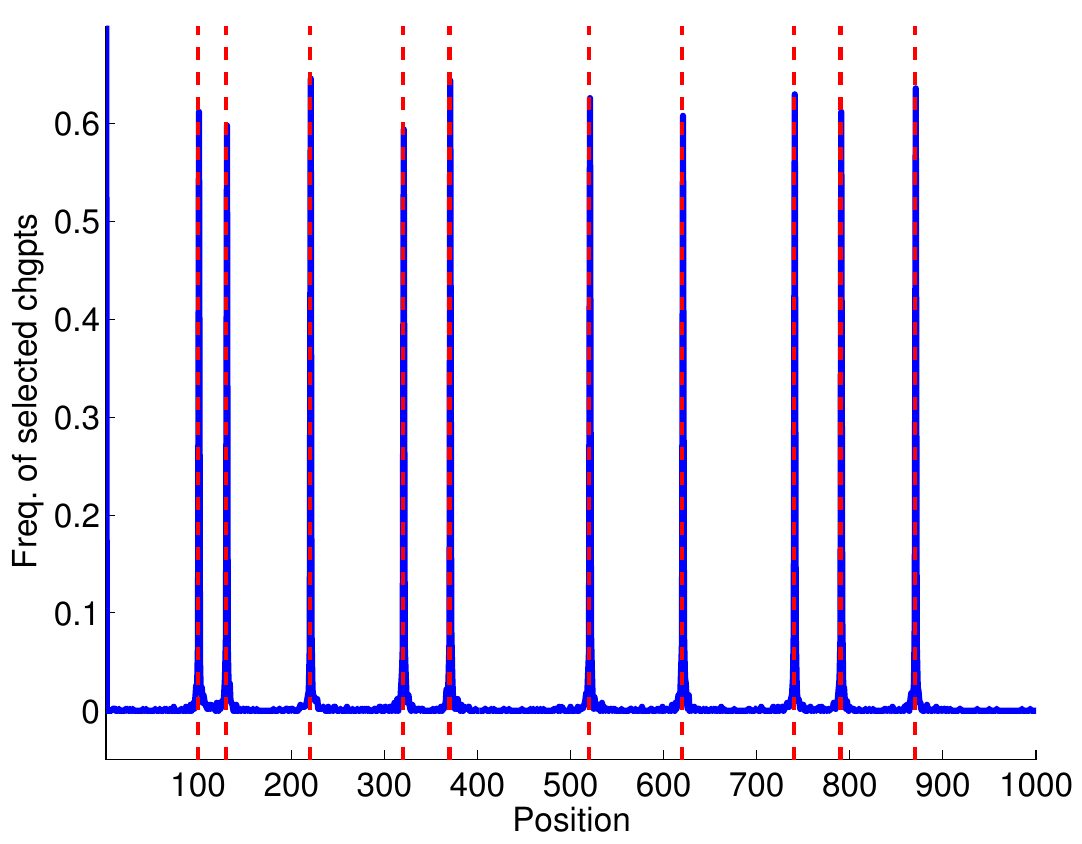}
}
  \caption{
Scenario~1: $\X=\R$, variable (mean, variance). 
Performance of KCP with three different kernels. 
Probability, for each instant $i \in \sets{1, \ldots, n}$, that $\tauh(\Ds)$ puts a change-point at $i$.
}
  \label{fig.Sc1.freq-Ds} 
\end{figure}

\begin{figure}
  \centering
\figtroishspace
  \subfloat[$k=\klin$]{\label{fig.Sc1.klin.freq} 
\includegraphics[width = \figtroiswidth]{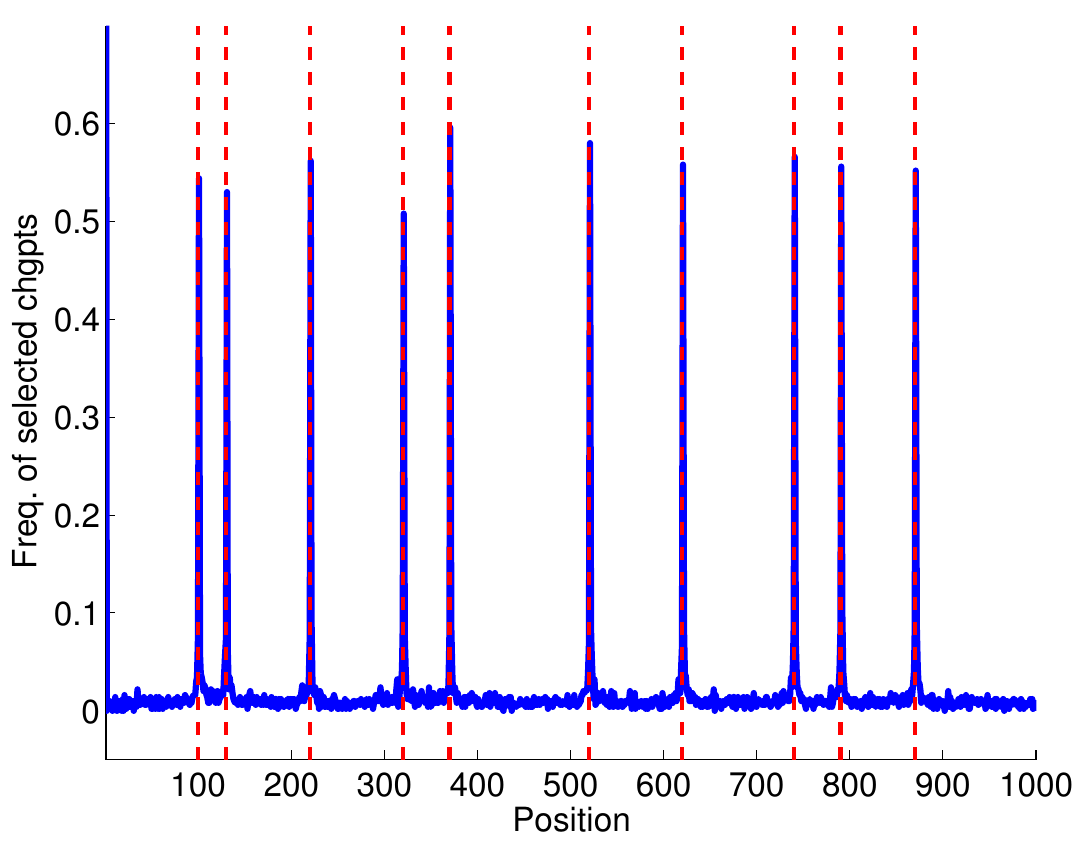}
}
 \hspace*{0.01\textwidth}
  \subfloat[$k=\kHer_{1}$]{\label{fig.Sc1.kHer.freq} 
\includegraphics[width = \figtroiswidth]{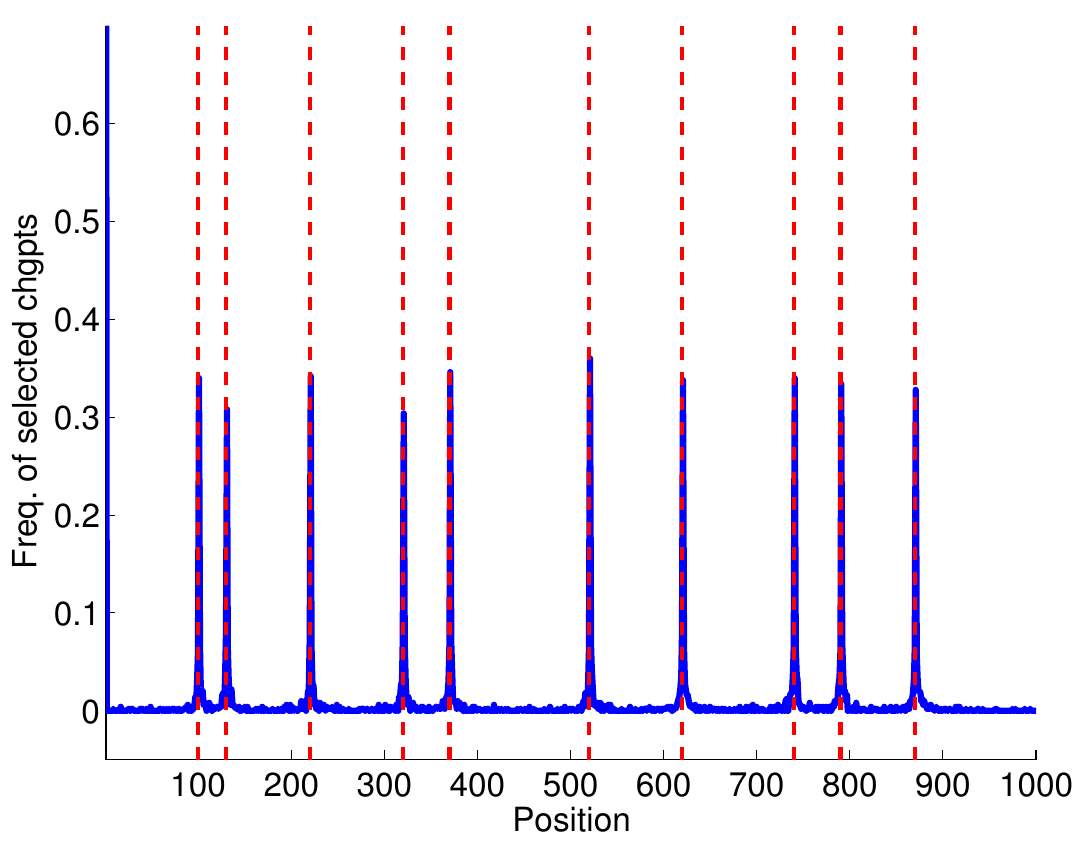}
}
 \hspace*{0.01\textwidth}
  \subfloat[$k=\kGau_{0.1}$]{\label{fig.Sc1.kGau.freq} 
\includegraphics[width = \figtroiswidth]{fig_Freq_chgpts_type_2kerneltype_L2_bw_1_Nb500}
}
  \caption{
Scenario~1: $\X=\R$, variable (mean, variance). 
Performance of KCP with three different kernels. 
Probability, for each instant $i \in \sets{1, \ldots, n}$, that $\tauh=\tauh(\Dh)$ puts a change-point at $i$. 
\\
For $k=\klin$, notice the high `baseline' level of (wrong) detection of change-points, which is due to a frequent overestimation of the number of change-points, see Figure~\ref{fig.Sc1.klin-Dh}. 
\\
(Figure~\ref{fig.Sc1.kGau.freq} is a copy of Figure~\ref{fig.Sc1.freq}, that we repeat here for making comparisons easier.) 
}
  \label{fig.Sc1.freq-Dh} 
\end{figure}

\clearpage


\begin{figure}
  \centering
\figtroishspace
  \subfloat[$k=\klin$]{\label{fig.Sc2.klin-freqDh} 
\includegraphics[width = \figtroiswidth]{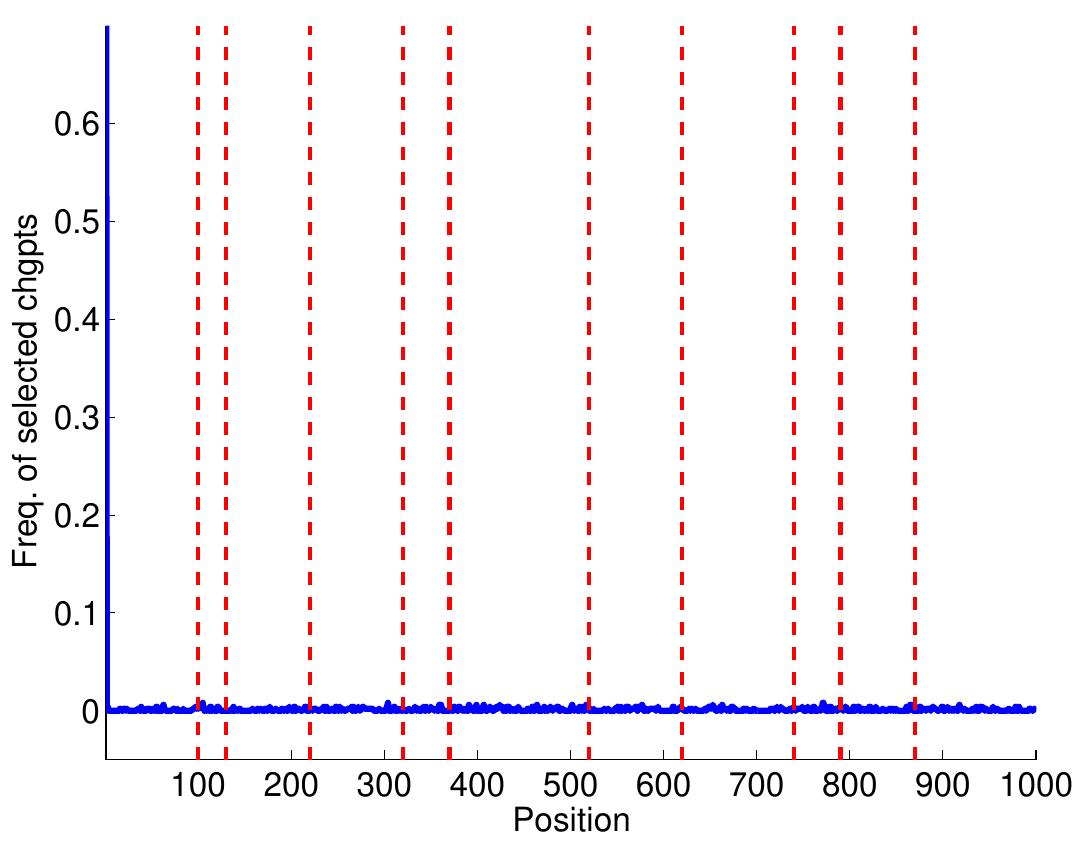}
}
 \hspace*{0.01\textwidth}
  \subfloat[$k=\kHer_{0.1}$]{\label{fig.Sc2.kHer-freqDh}
\includegraphics[width = \figtroiswidth]{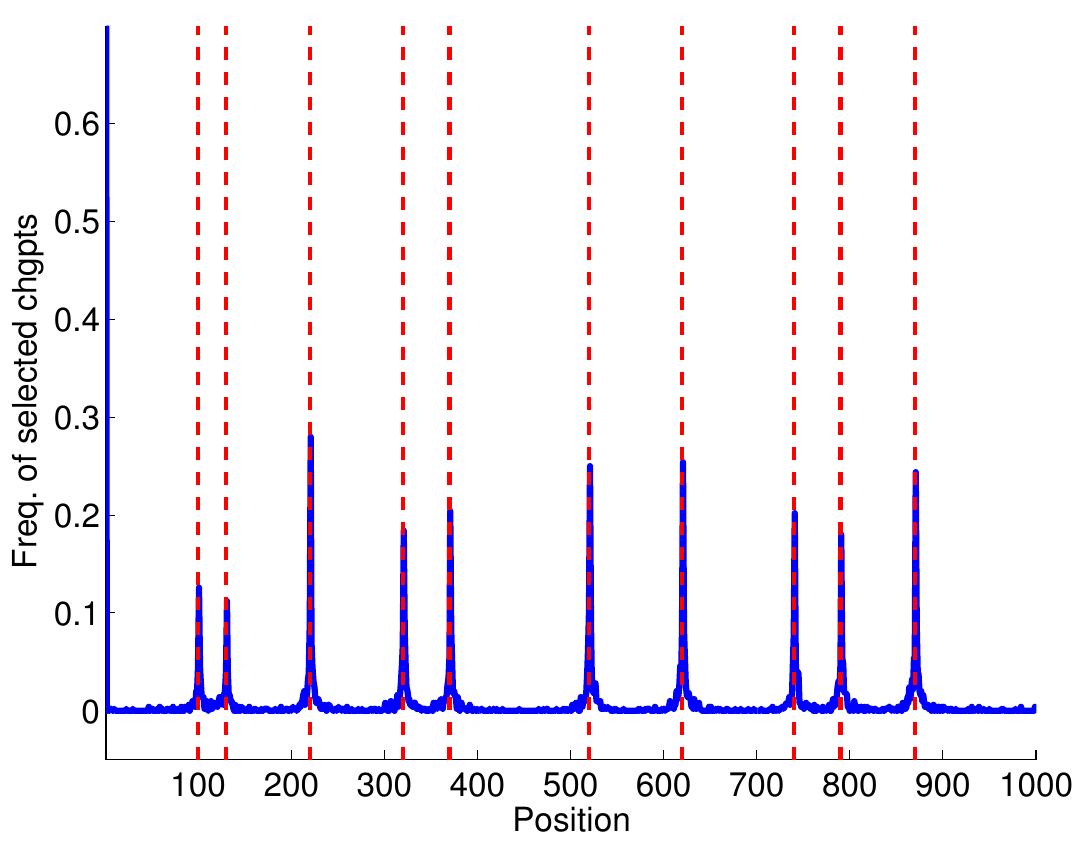}
}
 \hspace*{0.01\textwidth}
  \subfloat[$k=\kGau_{0.16}$]{\label{fig.Sc2.kGau-freqDh} 
\includegraphics[width = \figtroiswidth]{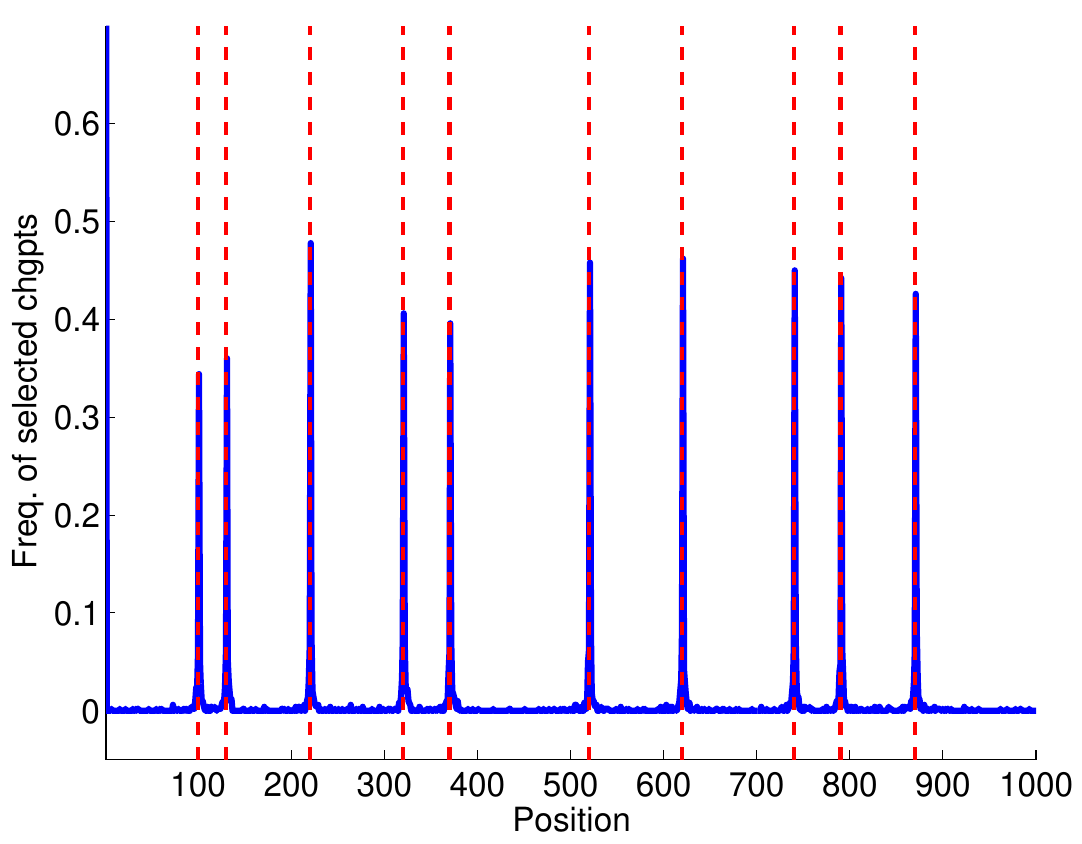}
}
\caption{%
Scenario 2: $\X=\R$, constant mean and variance. 
Performance of KCP with three different kernels $k$.  
Probability, for each instant $i \in \sets{1, \ldots, n}$, that $\tauh = \tauh(\Dh)$ puts a change-point at $i$.
}
\label{fig.Sc2.freqDh}
\end{figure}

\begin{figure}
  \centering
\figtroishspace
  \subfloat[$k=\klin$]{\label{fig.Sc2.klin-Dh} 
\includegraphics[width = \figtroiswidth]{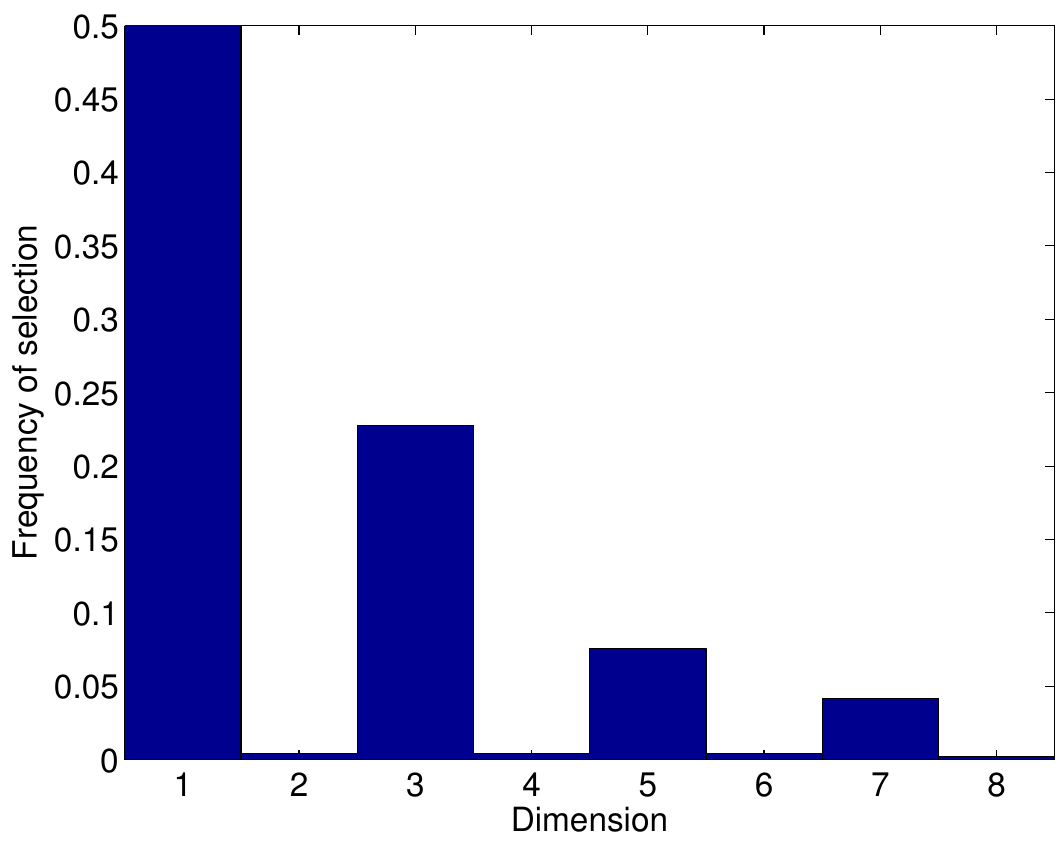}
}
 \hspace*{0.01\textwidth}
  \subfloat[$k=\kHer_{0.1}$]{\label{fig.Sc2.kHer-Dh}
\includegraphics[width = \figtroiswidth]{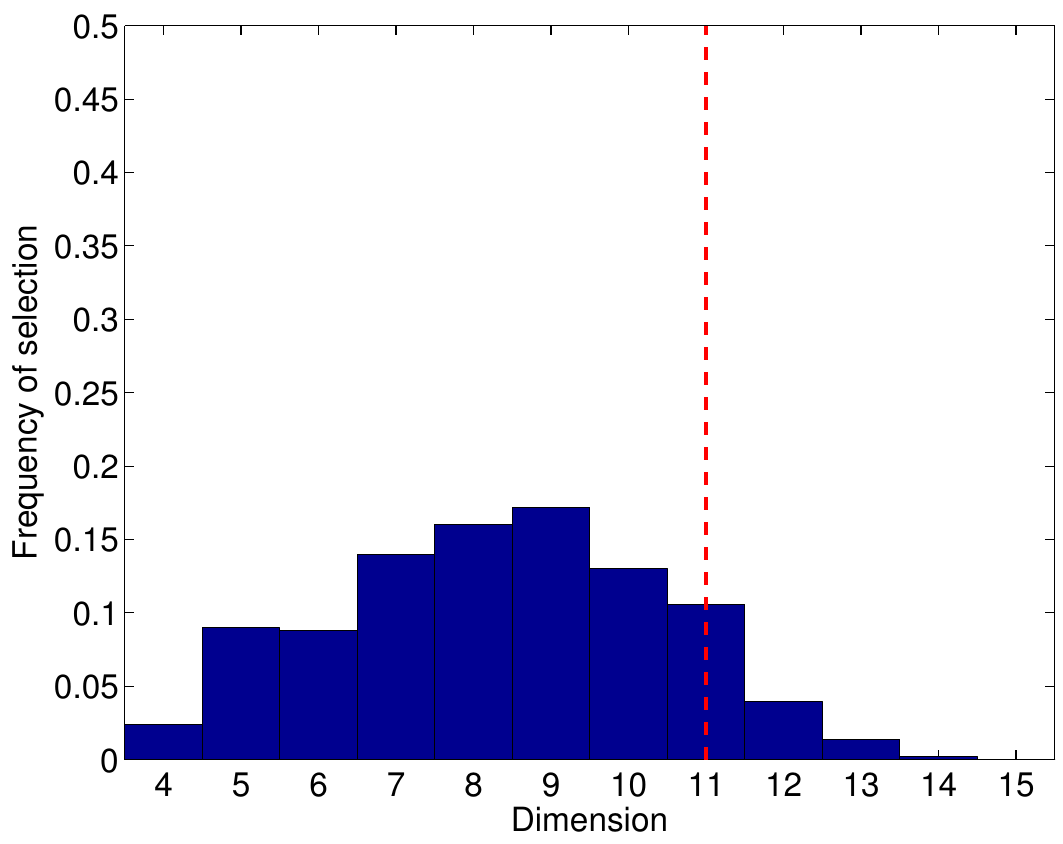}
}
 \hspace*{0.01\textwidth}
  \subfloat[$k=\kGau_{0.16}$]{\label{fig.Sc2.kGau-Dh}
\includegraphics[width = \figtroiswidth]{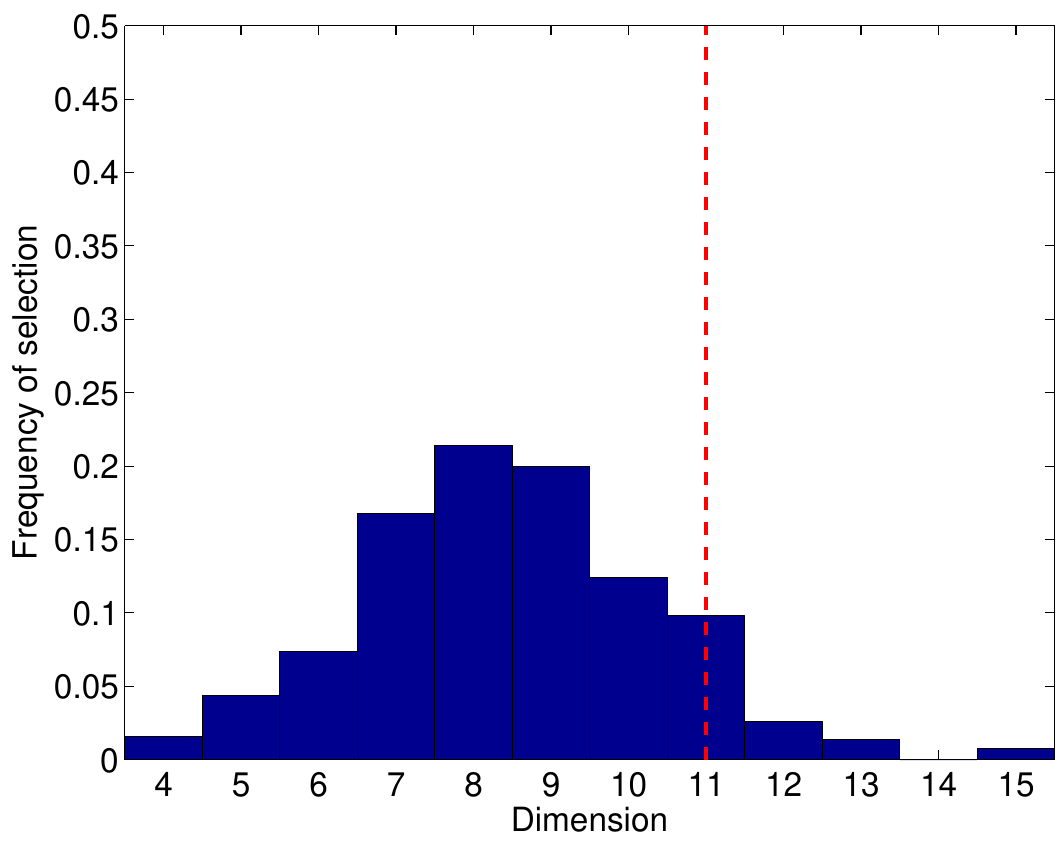}
}
\caption{%
Scenario 2: $\X=\R$, constant mean and variance. 
KCP with three different kernels~$k$.  
Distribution of~$\Dh$.
}
\label{fig.Sc2.Dh}
\end{figure}

\clearpage


\begin{figure}
  \centering
  \subfloat[$k=\kchi_{0.1}$]{\label{fig.Sc3.chi2-freqDh} 
\includegraphics[width = .45\textwidth]{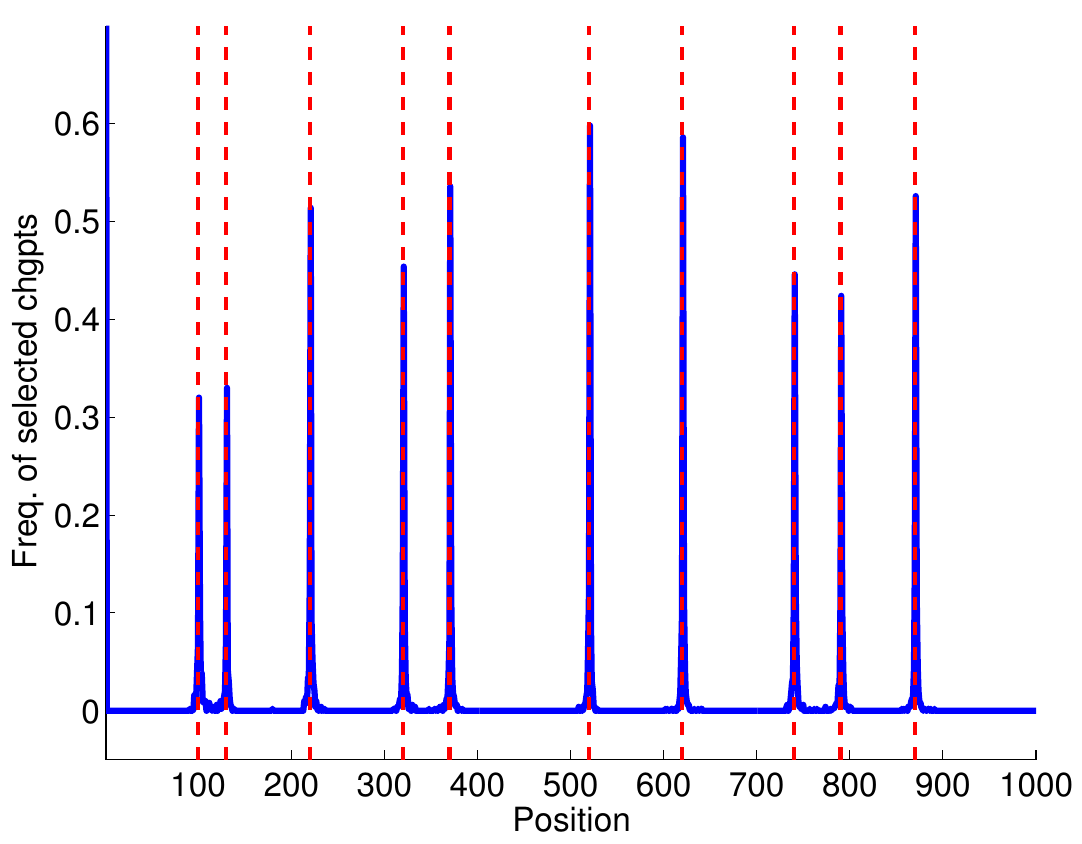}
}
 \hspace*{0.01\textwidth}
  \subfloat[$k=\kGau_{1}$]{\label{fig.Sc3.Gau-freqDh}
\includegraphics[width = .45\textwidth]{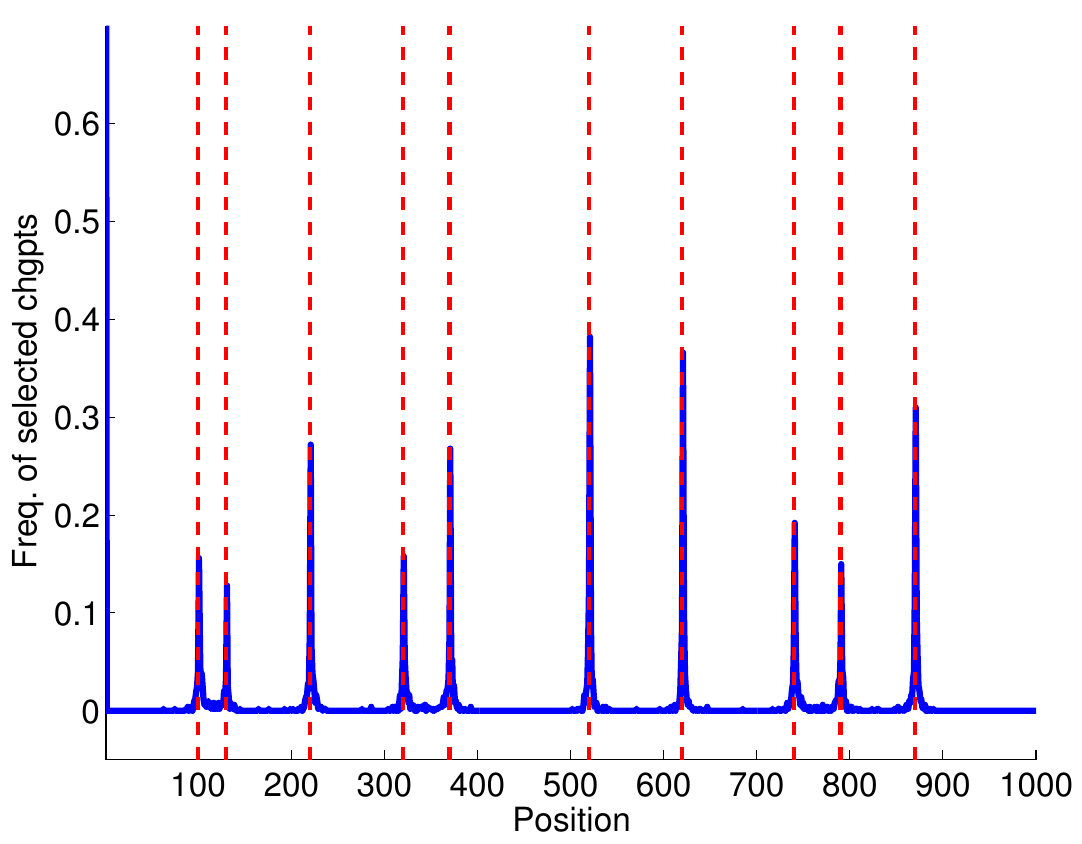}
}
  \caption{
Scenario~3: histogram-valued data. 
Performance of KCP with two different kernels. 
Probability, for each instant $i \in \sets{1, \ldots, n}$, that $\tauh = \tauh(\Dh)$ puts a change-point at $i$.
}
  \label{fig.Sc3.freq-Dh} 
\end{figure}

\clearpage


\begin{figure}[h!]
	\centering
	\figtroishspace
	\subfloat[Scenario~1, $k=\kGau_{0.1}$]{\label{fig.Sc1.penlin-Dh} 
		\includegraphics[width = \figtroiswidth]{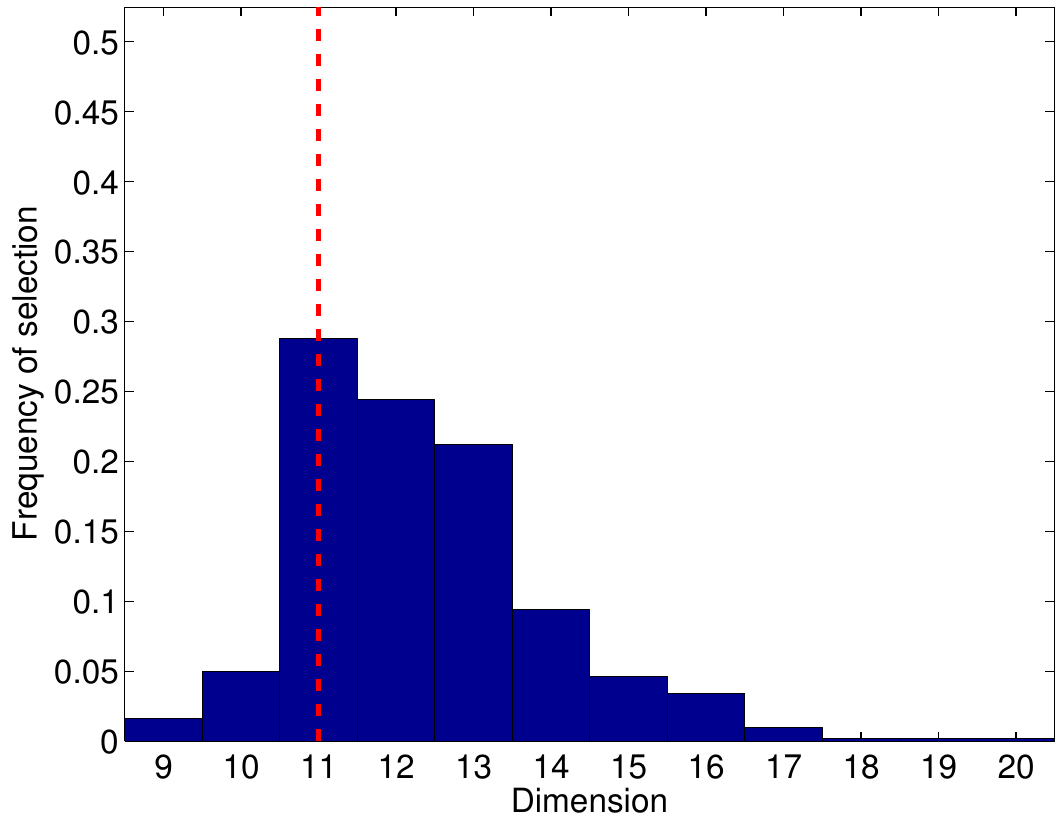}
	}
	\hspace*{0.01\textwidth}
	\subfloat[Scenario~2, $k=\kGau_{0.16}$]{\label{fig.Sc2.penlin-Dh} 
		\includegraphics[width = \figtroiswidth]{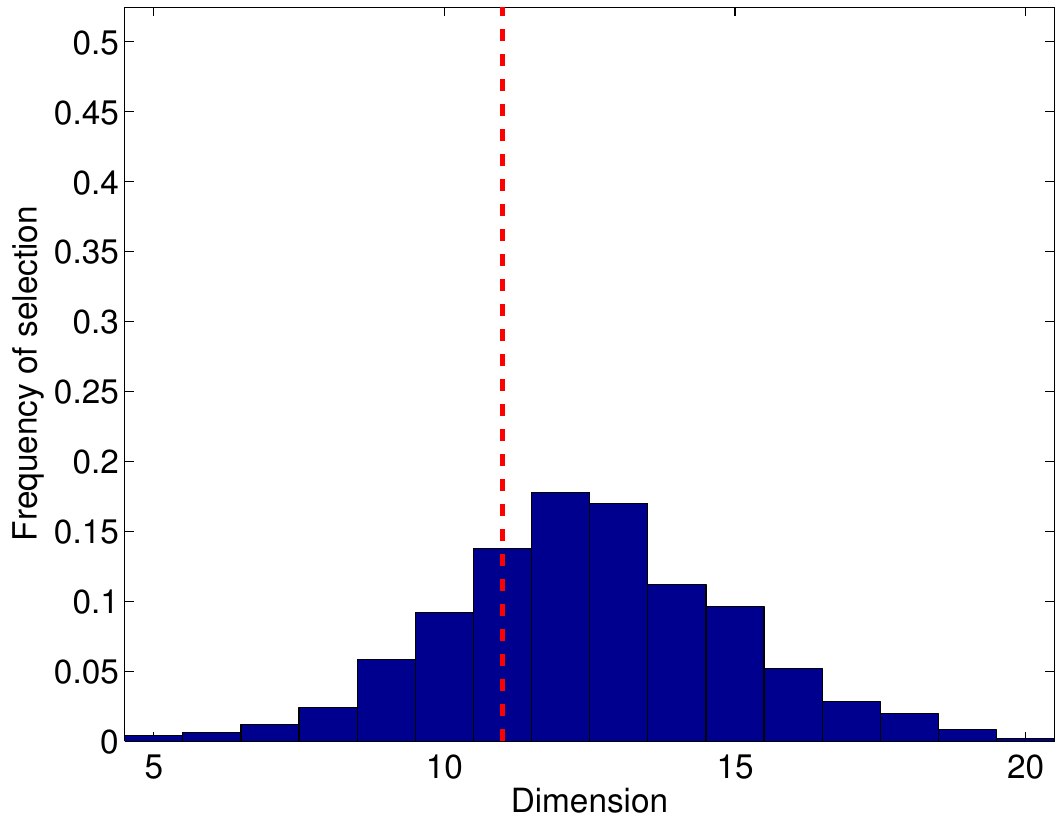}
	}
	\hspace*{0.01\textwidth}
	\subfloat[Scenario~3, $k=\kchi_{0.1}$]{\label{fig.Sc3.penlin-Dh} 
		\includegraphics[width = \figtroiswidth]{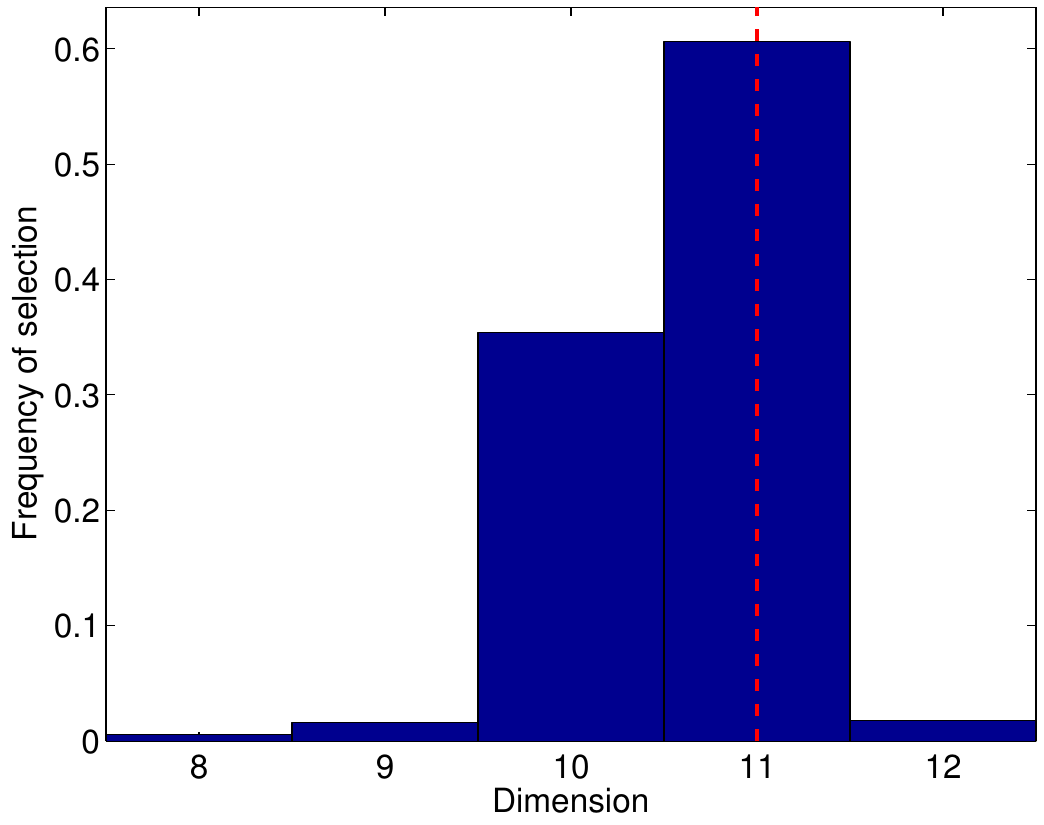}
	}
	\caption{%
	KCP with a linear penalty (see Section~\ref{sec.synthetic.data.results}): 
	distribution of $\Dh$.
	}
	\label{fig.LinearPenalty.Dh}
\end{figure}

\begin{figure}[h!]
	\centering
	\figtroishspace
	\subfloat[Scenario~1, $k=\kGau_{0.1}$]{\label{fig.Sc1.penlin-freq} 
		\includegraphics[width = \figtroiswidth]{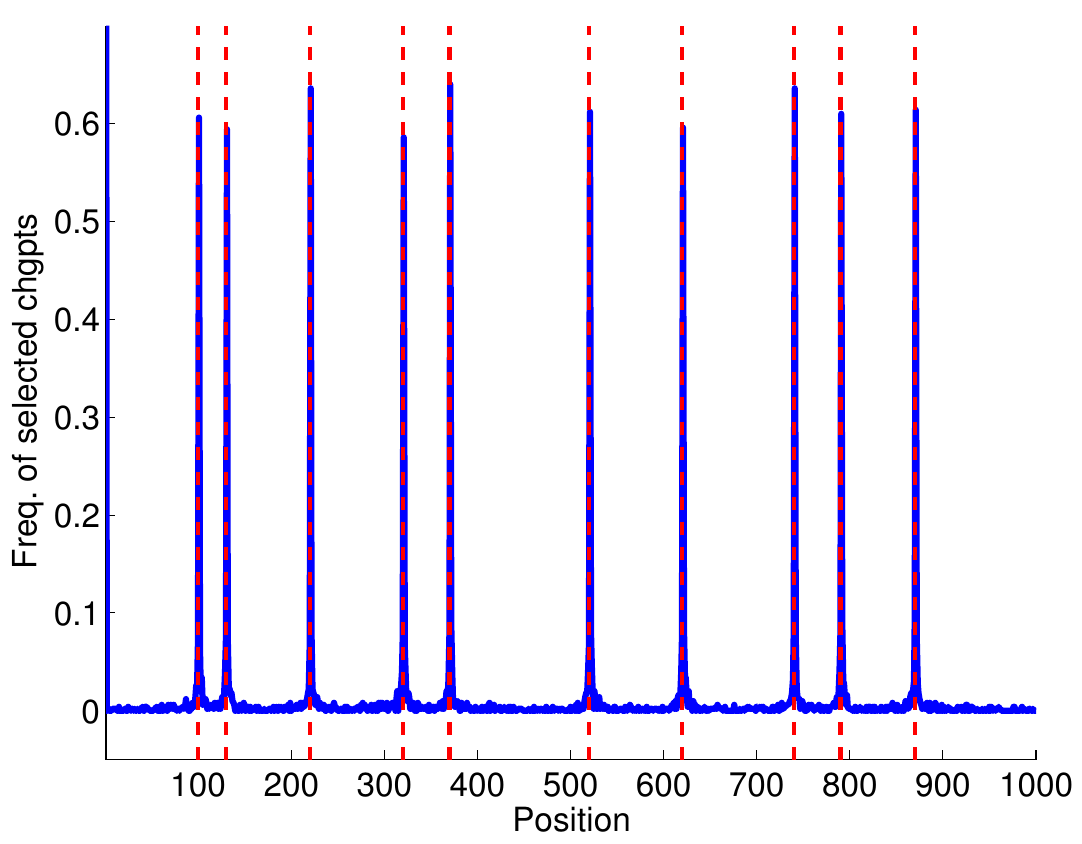}
	}
	\hspace*{0.01\textwidth}
	\subfloat[Scenario~2, $k=\kGau_{0.16}$]{\label{fig.Sc2.penlin-freq} 
		\includegraphics[width = \figtroiswidth]{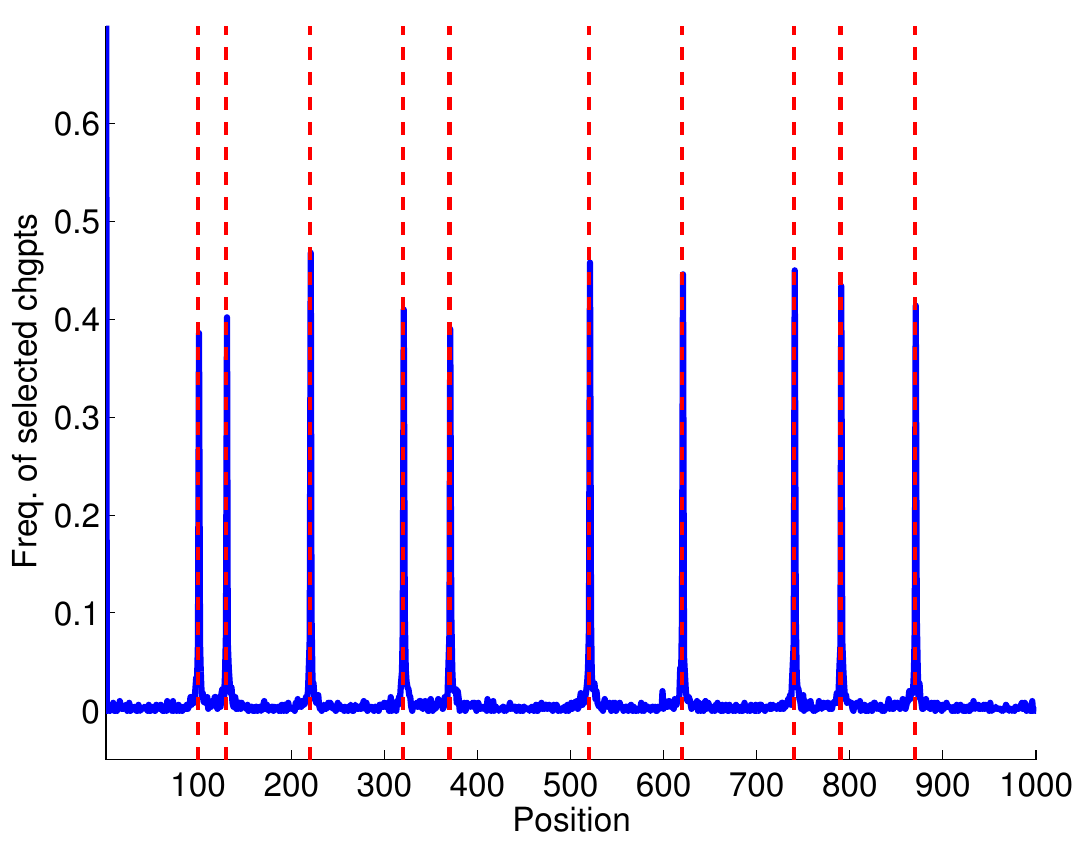}
	}
	\hspace*{0.01\textwidth}
	\subfloat[Scenario~3, $k=\kchi_{0.1}$]{\label{fig.Sc3.penlin-freq} 
		\includegraphics[width = \figtroiswidth]{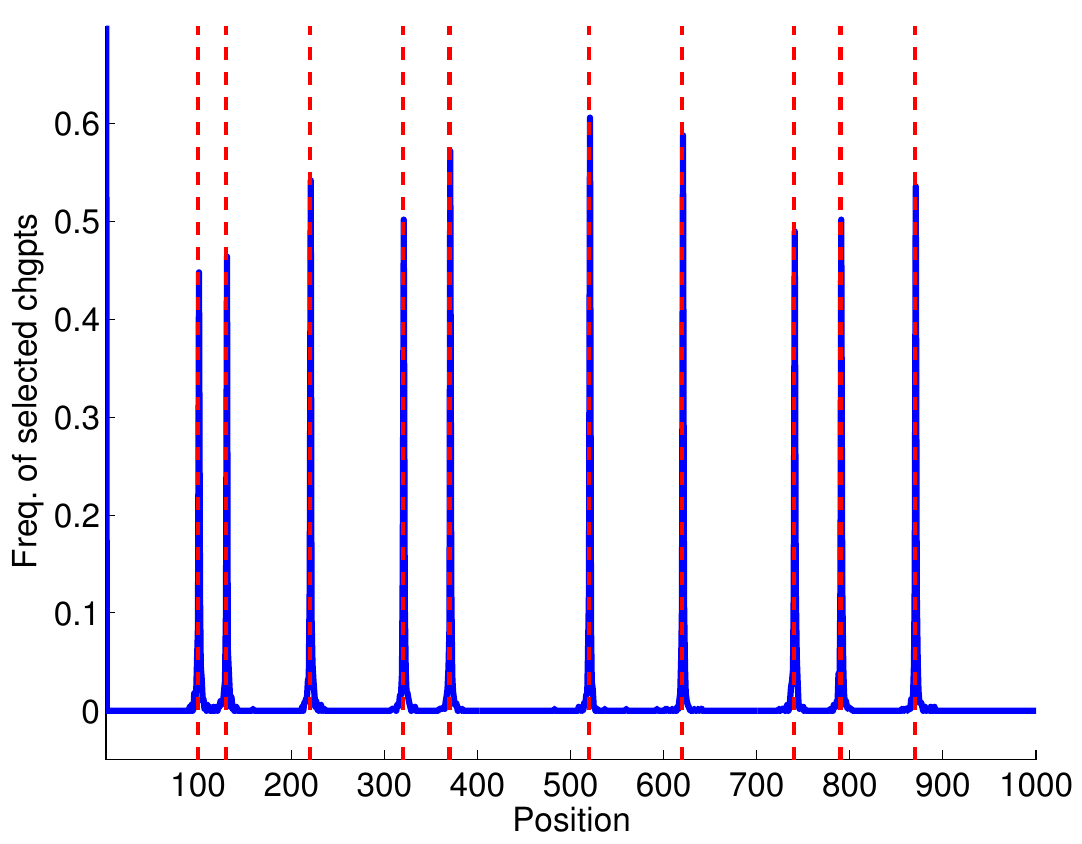}
	}
	\caption{%
	KCP with a linear penalty (see Section~\ref{sec.synthetic.data.results}): 
	Probability, for each instant $i \in \sets{1, \ldots, n}$, that $\tauh = \tauh(\Dh)$ puts a change-point at $i$.
	}
	\label{fig.LinearPenalty.freq}
\end{figure}

\clearpage


\begin{figure}[h!]
	\centering
	\subfloat[Scenario~1]{\label{fig.Sc1.EDivisive-Dh} 
		\includegraphics[width = .45\textwidth]{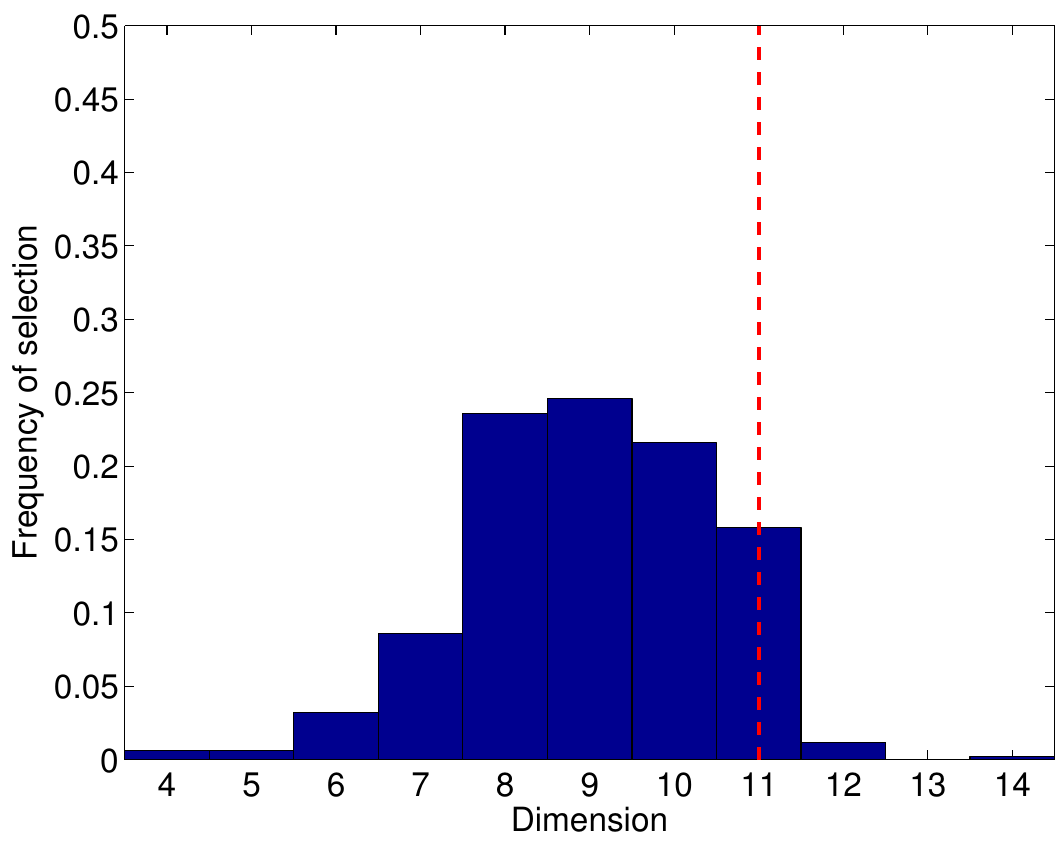}
	}
	\hspace*{0.01\textwidth}
	\subfloat[Scenario~2]{\label{fig.Sc2.EDivisive-Dh}
		\includegraphics[width = .45\textwidth]{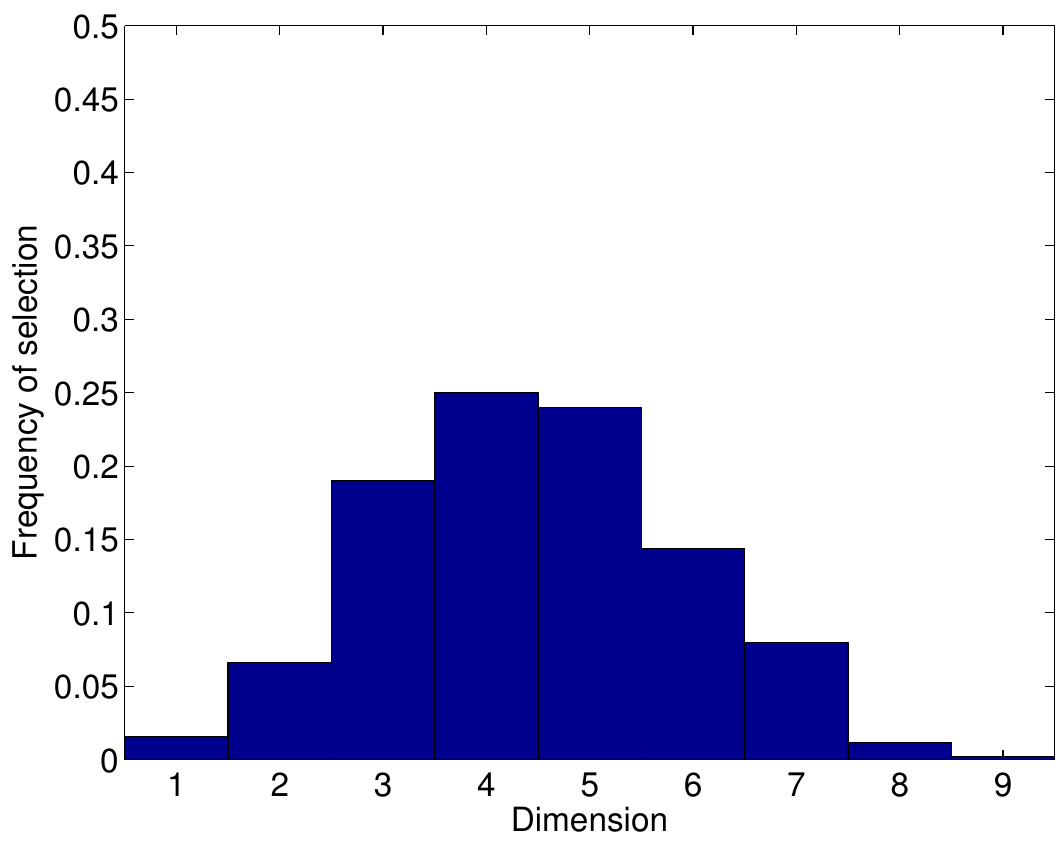}
	}
	\caption{
	E-divisive procedure (ED, see Section~\ref{sec.synthetic.data.results}) with type-I error level $\mathtt{sig.lvl}=0.05$, $\alpha=1$, and $R=199$: 
	Distribution of $\Dh$, the number of segments selected.
	}
	\label{fig.EDivisive-Dh} 
\end{figure}

\begin{figure}[h!]
	\centering
	\subfloat[Scenario~1]{\label{fig.Sc1.EDivisive-freqDh} 
		\includegraphics[width = .45\textwidth]{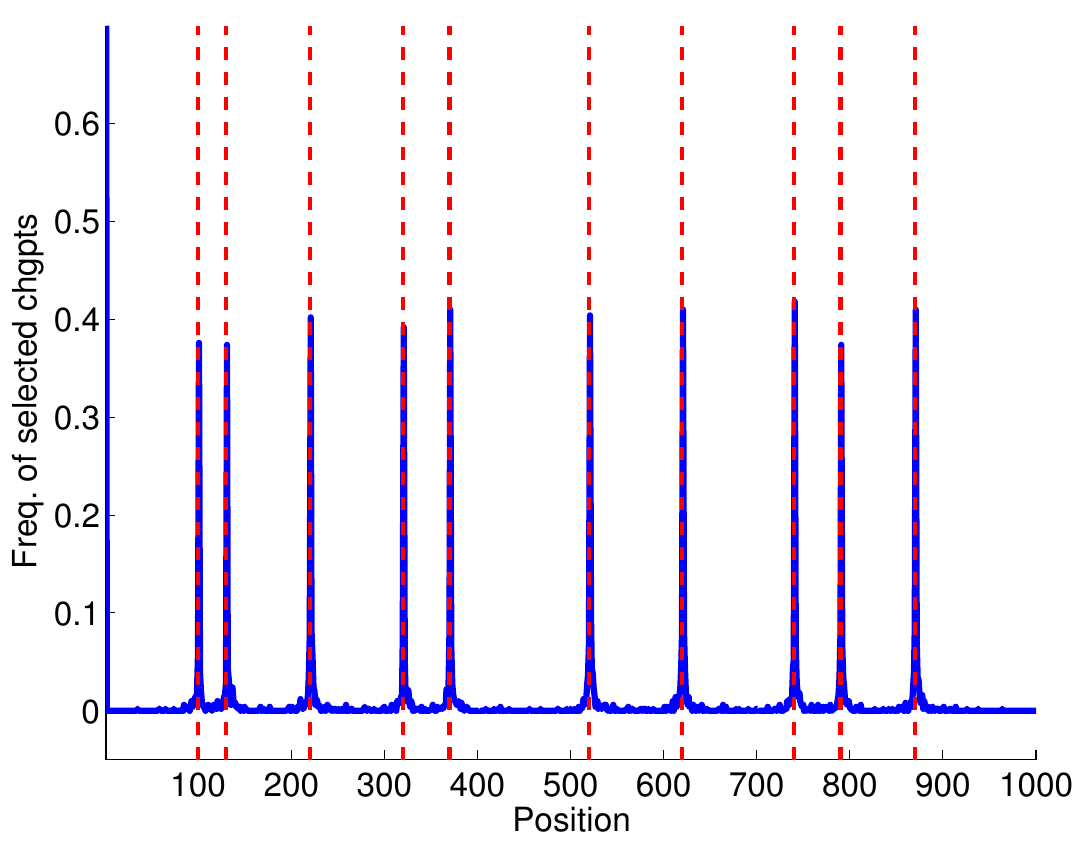}
	}
	\hspace*{0.01\textwidth}
	\subfloat[Scenario~2]{\label{fig.Sc2.EDivisive-freqDh}
		\includegraphics[width = .45\textwidth]{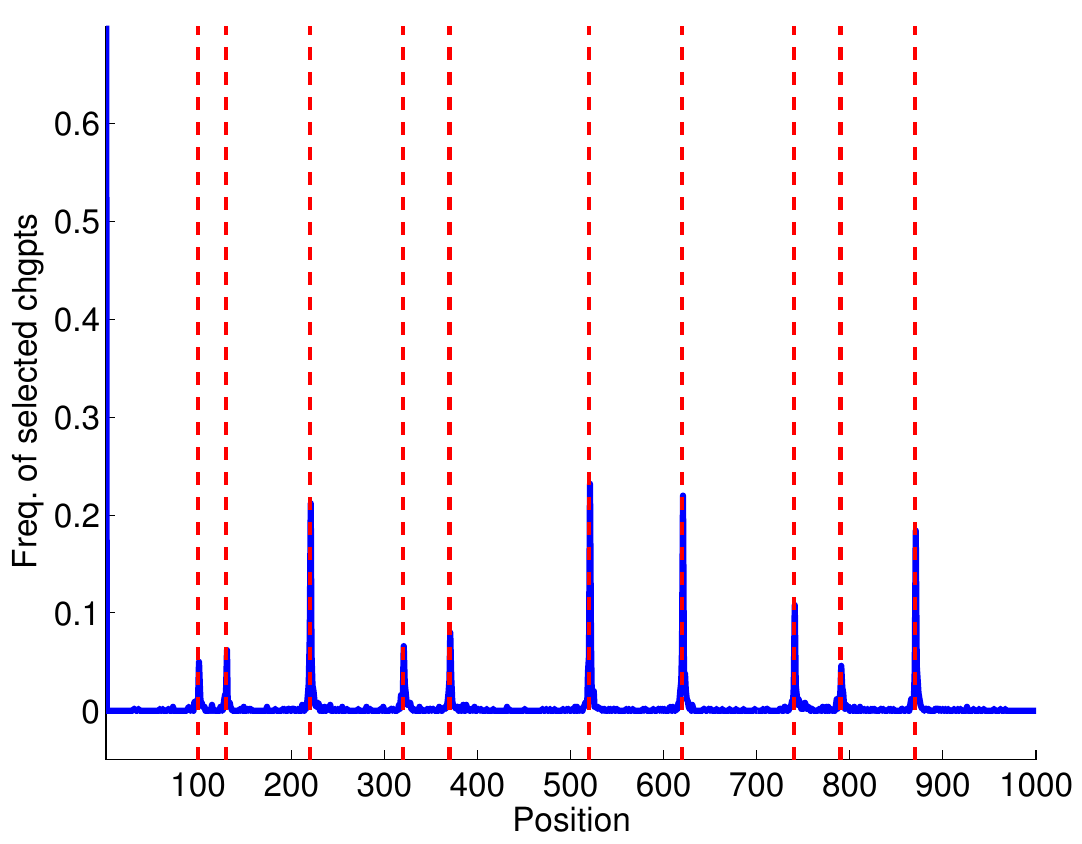}
	}
	\caption{
	E-divisive procedure (ED, see Section~\ref{sec.synthetic.data.results}) with type-I error level $\mathtt{sig.lvl}=0.05$, $\alpha=1$, and $R=199$: 
	Probability, for each instant $i \in \sets{1, \ldots, n}$, that $\tauh_{\mathrm{ED}}$ puts a change-point at $i$.
	}
	\label{fig.EDivisive.freq-Dh} 
\end{figure}

\begin{figure}[h!]
	\centering
	\subfloat[Scenario~1]{\label{fig.Sc1.EDivisive-freqDs} 
		\includegraphics[width = .45\textwidth]{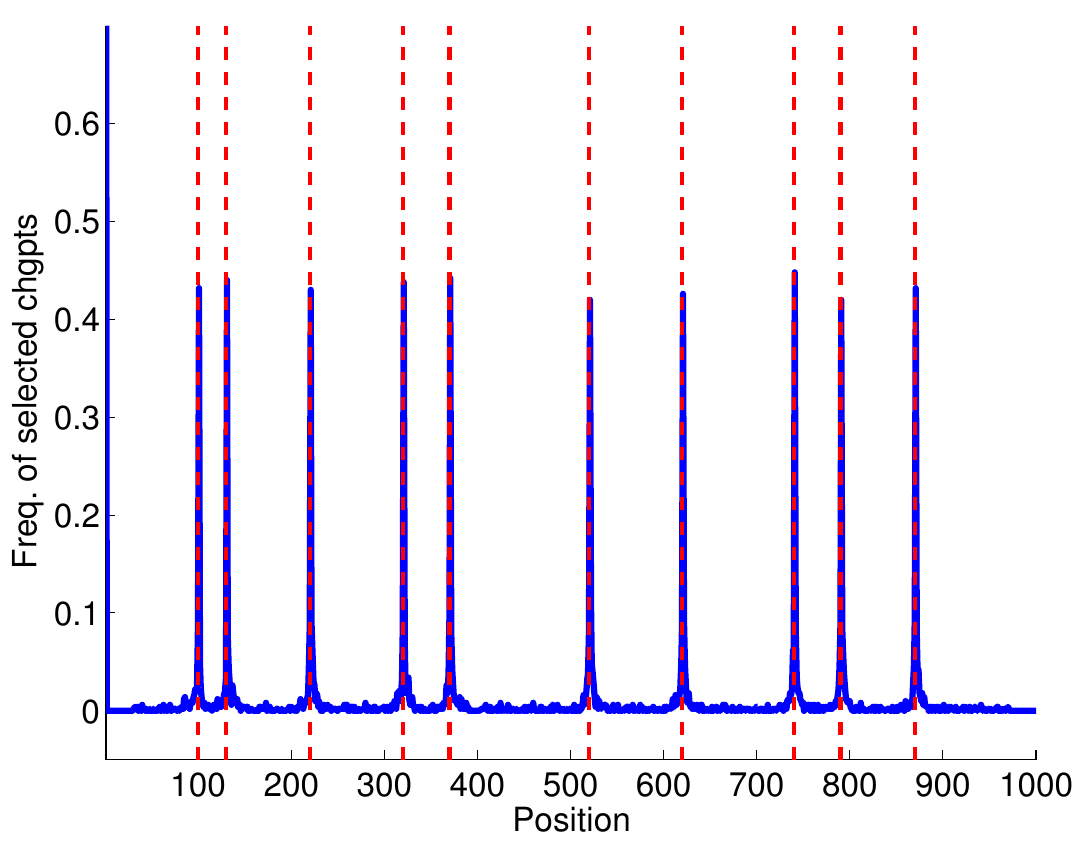}
	}
	\hspace*{0.01\textwidth}
	\subfloat[Scenario~2]{\label{fig.Sc2.EDivisive-freqDs}
		\includegraphics[width = .45\textwidth]{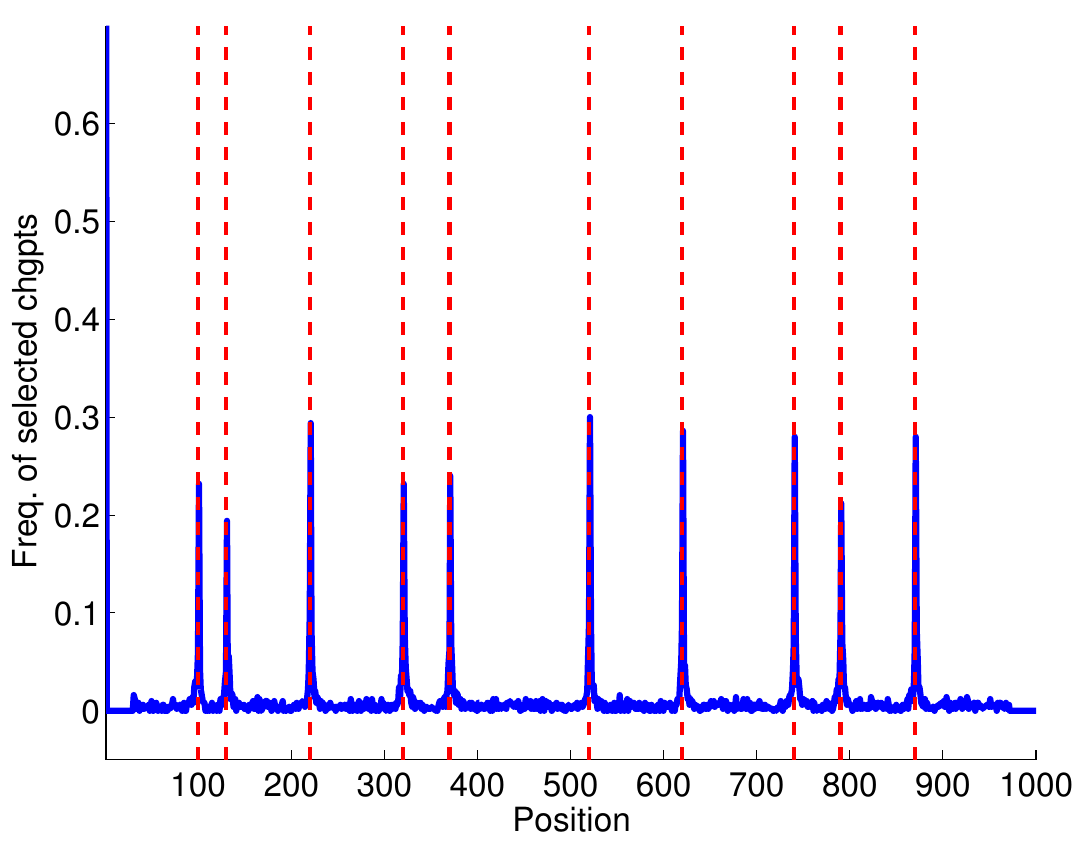}
	}
	\caption{
	E-divisive procedure (ED, see Section~\ref{sec.synthetic.data.results}) with $\alpha=1$ and $D=\Ds=11$ known: 
	Probability, for each instant $i \in \sets{1, \ldots, n}$, that $\tauh_{\mathrm{ED}} (\Ds) $ puts a change-point at $i$. 
	}
	\label{fig.EDivisive.freq-Ds} 
\end{figure}

\end{document}